\documentclass[12pt]{article}
 \usepackage[english]{babel}
\usepackage{amsfonts,amsmath,amsxtra,amsthm,amssymb,latexsym}
\usepackage{color,soul}

\newcommand\sumprime{\operatornamewithlimits{\sum\nolimits^{\prime}}}
\def\C{{\mathbb{C}}}
\def\N{{\mathbb{N}}}
\def\R{{\mathbb{R}}}

\newcommand{\gG}{{\Gamma}}
\newcommand{\kA}{{\mathcal A}}
\newcommand{\kB}{{\mathcal B}}
\newcommand{\kS}{{\mathcal S}}
\newcommand{\kM}{{\mathcal M}}
\newcommand{\kH}{{\mathcal H}}

\newcommand{\kC}{{\mathcal C}}

\newcommand{\kT}{{\mathcal T}}
\newcommand{\bN}{{\mathbb N}}
\newcommand{\bR}{{\mathbb R}}
\newcommand{\bC}{{\mathbb C}}
\def\cH{{\mathcal H}}
\newcommand\dom{\operatorname{dom}}
\newcommand{\gotH}{{\mathfrak H}}
\def\Ext{{\rm Ext\,}}
\def\gr{{\rm gr\,}}
\newcommand{\supp}{\mathop{\rm supp}\nolimits}
\newcommand{\ran}{{\mathrm{ran\,}}}
\newcommand{\diag}{\mathop{\rm diag}\nolimits}
\newcommand{\dist}{\mathop{\rm dist}\nolimits}

\newcommand{\comp}{\mathop{\rm comp}\nolimits}

\newtheorem{theorem}{Theorem}[section]

 \newtheorem{corollary}[theorem]{Corollary}
 \newtheorem{lemma}[theorem]{Lemma}
 \newtheorem{proposition}[theorem]{Proposition}
 \theoremstyle{definition}
 \newtheorem{definition}[theorem]{Definition}
 \theoremstyle{remark}
 \newtheorem{remark}[theorem]{Remark}
 
 \numberwithin{equation}{section}

\DeclareMathOperator{\Span}{span}\DeclareMathOperator{\imm}{Im}

\DeclareMathOperator{\rank}{rank} \DeclareMathOperator{\cl}{cl}

\def\mul{{\rm mul\,}}

\author{Mark~M.~Malamud  and Konrad Schm\"udgen}

\begin{title}
{ Spectral theory  of   Schr\"odinger operators with infinitely many point
interactions and radial  positive   definite functions}
\end{title}
\date{}
\begin{document}
\maketitle

\begin{abstract}
A number of  results on radial positive definite functions on
${\mathbb R^n}$ related to Schoenberg's integral representation
theorem are obtained. They are applied to the study of spectral
properties of self-adjoint realizations   of two- and
three-dimensional Schr\"odinger operators with countably many
point interactions. In particular, we find conditions on the
configuration of point interactions such that any  self-adjoint
realization has purely absolutely continuous non-negative
spectrum. We also apply some results on Schr\"odinger operators to
obtain new results on completely monotone functions.
\end{abstract}

\textbf{Mathematics  Subject  Classification (2000)}.
 47A10, 47B25.

\textbf{Key  words}. Schr\"odinger operator, point interactions,
self-adjoint extension, spectrum, positive definite  function


\renewcommand{\contentsname}{Contents}
\tableofcontents

\section{Introduction}
An important topic in quantum mechanics is the spectral theory of
Schr\"odinger Hamiltonians with point interactions. These are
Schr\"odinger operators on the  Hilbert space $L^2({\mathbb R^d})$,
$1\le d\le 3,$ with potentials supported on a discrete (finite or
countable) set of points of ${\R^d}$. There is an extensive
literature on such operators, see e.g. \cite{AGHH88,
AG2000, BF61, GMZ11, GrHKM80,  HHM83, HHJ84, Kar83, Koc82, Pos08}
and references therein.

Let $X=\{x_j\}_1^{m}$ be the set of points in ${\R^d}$ and let
$\alpha=\{\alpha_j\}_1^{m}$ be a sequence of real numbers, where  $m\in \bN\cup \{\infty\}$.
The mathematical problem is to associate a self-adjoint operator
(Hamiltonian) on $L^2(\mathbb{R}^d)$ with the differential
expression
   \begin{equation}\label{eq0}
 \mathfrak{L}_d :=  \mathfrak{L}_d(X,\alpha) := -\Delta + \sum\limits_{j=1}^m\alpha_j\delta(\cdot - x_j),\qquad \alpha_j\in \mathbb{R},\quad   m \in\mathbb{N}\cup \{\infty\},
   \end{equation}
and to describe its spectral properties.

 There are at least two    natural ways to associate  a self-adjoint Hamiltonian
$H_{X,\alpha}$  with the differential expression \eqref{eq0}.
The first  one  is the form approach. That is, the Hamiltonian $H_{X,\alpha}$ is defined by the self-adjoint operator associated with the
quadratic form
    \begin{equation}\label{0.2}
\widetilde{\mathfrak t}^{(d)}_{X,\alpha}[f] = \int_{\R^d}|\nabla f|^2 dx + 
\sum\limits_{j=1}^m \alpha_j|f(x_j)|^2, \qquad \dom(\widetilde{\mathfrak t}^{(d)}_{X,\alpha}) = W^{2,2}_{\comp}(\R^d). 
    \end{equation}
This is possible for $d=1$ and  finite $m\in\N$, since in this case \emph{the quadratic form $\widetilde{\mathfrak
t}^{(1)}_{X,\alpha}$  is  semibounded below and closable} (cf. \cite{ReeSim75}).
Its closure   $\mathfrak t^{(1)}_{X,\alpha}$ is defined by the same expression
\eqref{0.2} on the domain
$\dom(\mathfrak t^{(1)}_{X,\alpha}) =
W^{1,2}(\R).$ 
For $m=\infty$
the form \eqref{0.2} is also closable whenever it is semibounded (see \cite[Corollary 3.3]{AKM2010}).

Another way to  introduce   local interactions  on
$X:=\{x_j\}_{j=1}^m\subset\mathbb{R}$   is to  consider the minimal
operator corresponding to the expression $\mathfrak{L}_1$  and  to impose boundary
conditions at the points $x_j.$ 
For instance, in the case $d=1$ and $m<\infty$ the domain of the corresponding
Hamiltonian $H_{X,\alpha}$ is given by
$$
\dom(H_{X,\alpha})=\{f\in W^{2,2}({\R}\setminus X)\cap W^{1,2}(\R):  f'(x_j+)-f'(x_j-)=\alpha_j f(x_j)\}.
$$

In contrast to the  one-dimensional case,
the quadratic form \eqref{0.2} 
\emph{is not closable} in $L^2(\mathbb{R}^d)$ for $d\ge 2$, so it
 does not define a self-adjoint operator. The latter
happens because the point evaluations $\,f\rightarrow
f(x)$  are no longer continuous  on the Sobolev space
$W^{1,2}(\mathbb{R}^d)$  in the case $d\ge 2$.

However, it is still possible to apply the  extension theory of symmetric operators.
F. Berezin and L. Faddeev proposed in their pioneering paper \cite{BF61}
 to consider the expression \eqref{eq0} (with $m=1$ and $d=3$) in this framework.
They defined the minimal symmetric  operator $H$ as a restriction of
$-\Delta$ to the domain
$\dom H = \{f \in W^{2,2}(\bR^d): f(x_1)=0\}$ and studied
the spectral properties of \emph{all its self-adjoint extensions}. 
Self-adjoint extensions (or realizations) of $H$ for finitely many
point interactions have been investigated since then in numerous papers (see
\cite{AGHH88}).
In the case of infinitely many point interactions
$X=\{x_j\}_1^\infty$ the minimal operator $H_{\min}$ is defined by
  \begin{equation} 
H_d := H_{d,\min} := -\Delta   \!\upharpoonright \dom H, \qquad \dom(H_d)
= \{f \in W^{2,2}(\bR^d): f(x_j)=0, \quad j\in
\bN\}.
  \end{equation}

In this paper  we investigate the "operator" \eqref{eq0} (with $d=3$ and $m=\infty$)
in the framework of boundary triplets. This is a new approach to the extension
theory of  symmetric operators that has been  developed
during the last three decades (see \cite{GG, DerMal91, BGPankr2008,Sch2012}).
A boundary triplet $\Pi = \{\kH,\gG_0,\gG_1\}$  for the adjoint of a
 densely defined symmetric operator $A$ consists of an auxiliary Hilbert space $\kH$ and two linear mappings
$\Gamma_0,\Gamma_1:\  \dom(A^*)\rightarrow\kH$ such that the mapping $\Gamma :=(\Gamma_0,\Gamma_1): \dom(A^*)
\rightarrow \kH \oplus\kH$ is surjective.
The main requirement is the abstract Green identity
%
\begin{equation}\label{GI_Intro}
(A^*f,g)_\gotH - (f,A^*g)_\gotH = (\gG_1f,\gG_0g)_\kH -
(\gG_0f,\gG_1g)_\kH,\qquad f,g\in\dom(A^*).
\end{equation}
A boundary triplet for $A^*$ exists whenever $A$ has equal deficiency indices, but it is not unique.
It  plays the role of a "coordinate system" for the quotient space $\dom(A^\ast)/ \dom(\overline{A})$  and leads to
a natural  parametrization   of the  self-adjoint extensions of $A$
by means of self-adjoint linear relations (multi-valued operators) in $\kH$, see \cite{GG} and \cite{Sch2012} for detailed treatments.

The main analytical tool in this approach is  the abstract Weyl
function $M(\cdot)$ which was introduced and studied in
\cite{DerMal91}. This  Weyl function plays a similar role in the
theory of  boundary triplets as the classical Weyl-Titchmarsh
function does in the  theory of Sturm-Liouville operators. In
particular, it allows one to investigate spectral properties of
extensions (see \cite{BraMal02, DerMal91, Mal92, MalNei09}).

When studying boundary value problems for differential
operators, one is searching for an appropriate boundary triplet
 such that:
 \begin{itemize}
 \item the properties of the  mappings $\Gamma=\{\Gamma_0,\Gamma_j\}$ should correlate with  
 trace properties of functions from the maximal domain $\dom(A^*)$, 
\item  the  Weyl function and the boundary operator should  have "good" explicit forms.
\end{itemize}
Such a  boundary triplet  was constructed and applied to
differential operators with infinite deficiency indices in the
following  cases:

\begin{enumerate}
\item[(i)]  smooth elliptic operators in bounded or unbounded
domains (\cite{Gr68}, \cite{Vi63}, see also \cite{Gru08}),
\item[(ii)] the maximal Sturm-Liouville operator $-d^2/dx^2+T$ in
$L^2([0,1]; \cH)$ with an unbounded operator potential $T=T^*\ge a
I$, $T\in \mathcal{C}(\mathcal{H})$ (\cite{GG}, see also
\cite{DerMal91} for the case of $L^2(\R_+; \cH)$),
\item[(iii)] the 1D Schr\"odinger operator $\mathfrak L_{1,X}$ in the cases
$d_*(X) > 0$ (\cite{Koc89}, \cite{Mik94}) and $d_*(X) = 0$
(\cite{KosMal10}), where $d_\ast(X)$ is defined by  \eqref{sparsexn}  below.
 \end{enumerate}
Constructing  such a "good" boundary triplet  involves always
nontrivial analytic results. For instance,  Grubb's
contruction \cite{Gr68} for (i) (see also the adaptation to the case
of Definition \ref{bound}  in \cite{Mal10}) is  based on trace
theory for elliptic operators developed by Lions and Magenes
\cite{LioMag72} (see also \cite{Gru08}). The approach in (iii)
is based on a general construction of a (regularized) boundary
triplet for direct sums of symmetric operators (see \cite[Theorem
5.3]{MalNei09} and \cite[Theorem 3.10]{KosMal10}).

In this paper  we study \emph{all} (that is,  not necessarily local) self-adjoint extensions of the operator
$H=H_3$ (realizations of  $\mathfrak{L}_3$) in the framework of
boundary triplets approach. As in \cite{AGHH88} our crucial assumption is
\begin{align}\label{sparsexn}
d_\ast(X):= {\rm inf}_{j\neq k}~|x_k-x_j|>0.
\end{align}
Our construction of a boundary triplet $\Pi$ for $H^\ast$ is 
based on the following result:  \emph{The sequence
}%
  \begin{align}\label{RieszBasisIntro}
\left\{\frac{e^{-|x - x_j|}} {|x - x_j|}\right\}^{\infty}_{j=1}
   \end{align}
\emph{forms a Riesz basis of the defect subspace} $\mathfrak{N}_{-1}(H) = \ker(H^* + I)$  \emph{of} $H^*$
(cf. Theorem \ref{rieszbasisN1}).
Using this boundary triplet $\Pi$ we parameterize
the set of self-adjoint extensions of $H,$  compute the
corresponding Weyl function $M(\cdot)$ and
investigate various  spectral properties of self-adjoint
extensions (semi-boundedness, non-negativity,
negative spectrum, resolvent comparability, etc.)

Our main  result on spectral properties of Hamiltonians with point interactions concerns the absolutely
continuous spectrum  ($ac$-spectrum). For
instance, if
    \begin{equation}\label{1.8Intro}
C:= \sum_{|j- k| >0}\frac{1}{|x_j-x_k|^2} < \infty,
    \end{equation}
 we prove that the  part ${\widetilde H}E_{\widetilde
H}(C,\infty)$ of \emph{every self-adjoint extension} ${\widetilde
H}$ of $H$ is absolutely continuous (cf. Theorems
\ref{AcspecTheorem} and \ref{AcspecTheorem2}). Moreover, under
 additional assumptions on $X$, we show that the singular
part of
$\widetilde{H}_+ := \widetilde{H}E_{\widetilde H}(0,\infty)$
 is trivial, i.e. $\widetilde{H}_+ =\widetilde{H}^{ac}_+.$

The absolute continuity of
self-adjoint realizations $\widetilde H$ of $H$ has been studied
only  in very few cases. Assuming that $X = Y+\Lambda$,
where $Y=\{y_j\}^N_1\in{\mathbb R}^3$ is a finite set and $\Lambda =
\{\sum^3_1 n_j a_j\in{\R}^3: (n_1,n_2,n_3)\in{\mathbb Z}^3\}$ is a
Bravais  lattice, it was proved in  \cite{AGHH87, AGHHK86, GrHKM80,
HHM83, HHJ84, Kar83, AG2000} (see also \cite[Theorems 1.4.5,
1.4.6]{AGHH88} and the references in \cite{AGHH88} and
\cite{AG2000}) that the  spectrum of some periodic
realizations is absolutely continuous and  has
a band structure with a finite number of gaps.

An important feature of our investigations is an appearently new
\emph{connection between the spectral theory of  operators
\eqref{eq0} for $d=3$  and the class $\Phi_3$ of radial positive
definite functions on $\R^3.$}  We exploit this connection in both
directions. In Section \ref{radialpositive} we combine the
extension theory of  the operator  $H$ with Theorem
\ref{rieszbasisN1} to obtain results on positive definite
functions and the corresponding Gram matrices (\ref{grf}), while
in Section \ref{d=3}   positive definite functions are applied to
the spectral theory of self-adjoint realizations of  operators
\eqref{eq0} with infinitely many point interactions.

The paper consists of two parts and  is organized as follows.

 Section \ref{radialpositive}  deals with   radial positive definite functions  on ${\mathbb R}^d$ and has been inspired
 by possible applications to the spectral theory of operators \eqref{eq0}.
 If $f$ is such a function and  $X=\{x_n\}_1^\infty$ is a sequence of points of ${\mathbb R}^d$, we say that $f$
 is {\it strongly $X$-positive definite} if there exists a constant $c>0$ such that  for all $\xi_1,\dots,\xi_m\in {\mathbb C}$,
  \begin{align*}
  \sum_{j,k=1}^m \xi_k{\overline \xi}_j f(x_k-x_j)\geq c\sum_{k=1}^m |\xi_k|^2, \qquad m\in \bN.
  \end{align*}
Using  Schoenberg's  theorem   we  derive a number of results showing
under certain assumptions on $X$ that  $f$ is strongly $X$-positive definite and that
the  Gram matrix
   \begin{align}\label{grf}
    Gr_X(f) :=  \big(f(|x_k-x_j|)\big)_{k,j\in \bN}
   \end{align}
defines a bounded operator on $l^2(\bN)$.
 The latter  results correlate with the properties
of the sequence $\{e^{i(\cdot,x_k)}\}_{k\in \bN}$  of exponential functions
to form   a Riesz-Fischer sequence or a Bessel sequence, respectively,
 in  $L^2(S^n_{r};\sigma_n)$ for some $r>0$.  

In Section  \ref{rieszbasisn1}  we prove that the sequence \eqref{RieszBasisIntro}  forms a Riesz basis in the closure of its linear span
if and only if $X$ satisfies  \eqref{sparsexn}. This result is
applied  to prove that \emph{for such $X$ and any non-constant
absolute monotone function  $f$ on $\R_+$ the function
$f(|\cdot|_3)$
is strongly $X$-positive definite}. Under an additional assumption it is  shown that the  
  matrix \eqref{grf} defines a boundedly invertible bounded operator on $l^2(\bN)$  (see Theorem \ref{propositionstronglydef}).

 The second part  of the paper is devoted to the spectral theory of self-adjoint
 operators associated with the expression \eqref{eq0}   for countably many point interactions.
Throughout this part we assume that $X$ satisfies condition \eqref{sparsexn}.

 In  Section \ref{preliminaries}  we collect some basic definitions and facts on boundary triplets,
 the corresponding  Weyl functions and  spectral properties of self-adjoint extensions.

In  Subsection \ref{sec5.1} we construct a boundary triplet
for the adjoint operator $H^*$ for $d=3$ and compute the corresponding Weyl function $M(\cdot).$
The explicit form of the Weyl function  given by \eqref{W3Operator} plays crucial role
in the sequel.  For the proof of
the  surjectivity  of the  mapping $\Gamma=(\Gamma_0,\Gamma_1)$      %
 the  strong $X$-positive  definiteness of the function
$e^{-|\cdot|}$ on  $\mathbb{R}^3$ is essentially used. The latter
follows  from the absolute monotonicity of the function $e^{-t}$
on $\R_+.$

 In Subsection \ref{sec5.2} we  describe the quadratic form generated  by  the semibounded
 operator $M(0)$  on $l^2(\N)$ as   strong resolvent limit of the  corresponding   Weyl  function
${M}(-x)$   as $x\to +0.$  For this we use the  strong  $X$-positive  definiteness
of the function   $\tfrac{1-e^{-|\cdot|}}{|\cdot|}$ on $\R^3$ which follows from
the absolute  monotonicity of the function  $\tfrac{1-e^{-t}}{t}$  on $\R_+.$
The operator $M(0)$
enters into the description of  the Krein extension of $H$ for $d=3$
and allow us to  characterize  all  non-negative self-adjoint extensions as well as all self-adjoint extensions with $\kappa (\le\infty)$ negative eigenvalues.
Using the behaviour of the Weyl function  at $-\infty$
we  show that any self-adjoint extension $H_B$ of $H$
  is  semibounded from below  if and only if
the corresponding boundary operator $B$ is.
A similar result for elliptic operators on exterior domains
has recently been  obtained by G. Grubb \cite{Gru2011}.

In Subsection \ref{AcSection}  we  apply a technique elaborated in
\cite{BraMal02, MalNei09} as well as a new general result (Lemma
\ref{lem4.22})  to investigate  the $ac$-spectrum of self-adjoint
realizations. In particular, we prove that the  part ${\widetilde
H}E_{\widetilde H}(C,\infty)$ of \emph{any self-adjoint
realization ${\widetilde H}$ of $\mathfrak{L}_3$  is absolutely
continuous} provided that condition  \eqref{1.8Intro} holds.
Moreover, under some additional assumptions on $X$ we show that
the singular non-negative part $\widetilde{H^s}E_{\widetilde
H}(0,\infty)$  of any realization $\widetilde{H}$ is trivial.
Among others, Theorems \ref{AcspecTheorem} and
\ref{AcspecTheorem2} provide explicit examples which show that an
analog of the Weyl--von Neumann theorem does not hold for
non-additive (singular) compact (and even non-compact)
perturbations.
The proof of  these  results is based on the fact that the
function  $\tfrac{\sin st}{t}$ belongs to $ \Phi_3$ for each $s>0$. Then, by
Propositions \ref{stronglyxdet} and   \ref{pointsr3},
$\tfrac{\sin s|\cdot|}{|\cdot|}$ is  strongly  $X$-positive definite 
for certain subsets $X$ of $\R^3$ and any $s>0.$ The latter  is equivalent to the
invertibility of the matrices
   $$
\kM(t):=\left(\delta_{kj}+\frac{\sin(\sqrt{t}|x_k-x_j|)}{\sqrt{t}|x_k-x_j|+\delta_{kj}}\right)_{j,k=1}^\infty
\qquad\text{for}\qquad t\in \R_+
     $$
and plays  a crucial role  in the proof  of Lemma \ref{lem4.22}.

\medskip

\textbf{Notation.} Throughout the paper  $\gotH$ and $\kH$ are  separable complex
Hilbert spaces. We denote by $\kB({\kH,\gotH})$    the bounded
linear operators from ${\kH}$ into $\gotH$, by $\kB(\kH)$ the set $\kB(\kH,\kH)$,  by $\kC(\kH)$
the   closed linear operators on $\kH$ and by
${\mathfrak S}_p(\kH)$  the Neumann-Schatten ideal on $\kH$.
In particular, ${\mathfrak S}_\infty(\kH)$ and ${\mathfrak S}_1(\kH)$ are the ideals of
compact operators and  trace class operators on $\kH$, respectively.

For  closed linear operator $T$ on  $\mathfrak{H}$, we write $\dom (T)$, $\ker (T)$, $\ran (T)$, $\gr(T)$ for
  the domain, kernel, range, and graph of $T$, respectively, and $\sigma(T)$ and $\rho (T)$
for the spectrum and the resolvent set of $T$. The symbols $\sigma_c(T),\
\sigma_{ac}(T),\  \sigma_{s}(T),\  \sigma_{sc}(T),\ \sigma_{p}(T)$
denote the  continuous, absolutely continuous, singular, singularly
continuous and point spectrum, respectively, of a self-adjoint  operator $T$. Note that $\sigma_{s}(T) = \sigma_{sc}(T)\cup
\sigma_{p}(T)$ and $\sigma(T) = \sigma_{ac}(T)\cup \sigma_{s}(T).$
The defect subspaces of a symmetric operator $T$ are denoted by $\mathfrak {N}_z$. For basic notions and results on operator theory we refer to \cite{ReeSim75}, \cite{ReeSim78}, \cite{Sch2012}, and \cite{Kato66}.

By
\,$C[0,\infty)$ we mean the Banach space of
continuous  bounded functions on $[0,\infty)$ and by
$S^n_{r}$  the sphere in $\R^n$ of  radius $r$ centered
at the origin and   $S^{n}:= S^n_{1}.$\ Further,   $\sum^\prime_{k\in
\bN}$ denotes the sum over all $k$ such that $k\neq j$ and $\sum_{|k-j|>0}$
is the sum over all  $k, j\in \N$ with $k\neq j.$

\section{Radial positive definite  functions}\label{radialpositive}

\subsection{Basic definitions }\label{prelim1}

 Let  $(u,v)=u_1v_1+\ldots+u_nv_n$ be  the scalar
product  of  two vectors  $u=(u_1,\ldots,u_n)$ and
$v=(v_1,\ldots,v_n)$ from $\R^n$, $n\in\N$,   and  let $|u| =
|u|_n =\sqrt{(u,u)}$ be the Euclidean norm of $u$. First we
recall some basic facts and  notions  about  positive  definite
functions \cite{Akh65}.
\begin{definition}\label{defpoz}\cite{Akh65}
 A function $g:\R^n\to\C$ is called {\it positive definite}
if $g$ is
continuous at $0$ and for arbitary finite sets
$\{x_1,\dots,x_m\}$ and
$\{\xi_1,\dots,\xi_m\}$, where $x_k\in \bR^n$ and $\xi_k\in \bC$, we have
\begin{equation}\label{positiv}
\sum_{k,j=1}^{m}\xi_k\overline{\xi}_jg(x_k-x_j)\ge 0.\,
\end{equation}
  \end{definition}
The set of positive definite function on $\bR^n$ is denoted by $\Phi(\mathbb{R}^n)$.

Clearly, a function $g$ on $\bR^n$ is positive definite if and only if it  is continuous at $0$ and the  matrix
$G(X)=\left(g_{kj}{:=}g(x_k{-}x_j)\right)_{k,j=1}^m$ is {\it positive semi-definite} for
any finite subset $X=\{x_j\}_1^m$ of $\R^n.$

The  following  classical  {\it Bochner theorem} gives a   description of   the  class $\Phi(\mathbb{R}^n)$.

 \begin{theorem}\cite{Boch}\label{boch}
A  function $g(\cdot)$  is  positive  definite  on  $\R^n$ if  and only if there is a  finite nonnegative Borel  measure $\mu$ on  $\R^n$ such that
\begin{equation}\label{bochequa}
g(x)=\int\nolimits_{\R^n}e^{i(u,x)}d\mu(u)~~{\rm for ~all}~~x\in \bR^n.
\end{equation}
    \end{theorem}
Let us continue with a number  of further basic definitions.
   \begin{definition}
Let $g$ be a positive definite  function on $\bR^n$ and let $X$ be a subset of $\bR^n$.

 (i)   We say that $g$ is {\it strongly $X$-positive definite} if there exists a constant $c>0$ such that
   \begin{equation}\label{stronglyxpositive}
\sum_{k,j=1}^{m}\xi_k\overline{\xi}_jg(x_k-x_j) >  c \sum_{k=1}^m |\xi_k|^2, \qquad \xi=\{\xi_1,\dots,\xi_m\}\in \bC^m\setminus\{0\}.
   \end{equation}
for any finite set $\{x_j\}_{j=1}^m$  of distinct points $x_j\in X.$

 (ii)   It is said that $g$ is {\it strictly $X$-positive definite} if \eqref{stronglyxpositive}
 is satisfied with  $c = 0.$
    \end{definition}

Any  strongly $X$-positive definite  $g$  is also strictly $X$-positive definite.
For {\rm finite} sets $X=\{x_j\}_1^m$  both notions are equivalent by  the compactness of the sphere in $\bC^m.$

The following problem seems to be important and difficult.

{\bf{Problem:}} {\it Let $g$ be  a  positive definite function on $\bR^n$. Characterize those countable subsets $X$ of $\bR^n$ for which $g$ is strictly $X$-positive definite and strongly $X$-positive definite, respectively. }

\medskip
We now define three other  basic concepts which will be crucial in what follows.
    \begin{definition} \cite{young80}
Let $F=\{f_k\}_{k=1}^\infty$ be a sequence of vectors of a Hilbert space $\kH$.

\item (i) This sequence is called a {\it Riesz-Fischer sequence}
if there exists a constant $c>0$ such that
   \begin{align}\label{rieszfischerse}
\bigg\|\sum_{k=1}^m \xi_k f_k\bigg\|^2_\kH \geq  c~ \sum_{k=1}^m |
\xi_k|^2 \quad \text{for all}\quad
(\xi_1,\cdots,\xi_m)\in \bC^m \quad\text{and}\quad m\in \bN.
  \end{align}

\item (ii) The sequence $F$ is said to be a {\it Bessel sequence}
if there is a constant $C>0$ such that
   \begin{align}\label{besselse}
\bigg\|\sum_{k=1}^m \xi_k f_k\bigg\|^2_\kH  \leq C~ \sum_{k=1}^m |
\xi_k|^2 \quad \text{for all}\quad
(\xi_1,\cdots,\xi_m)\in \bC^m \quad\text{and}\quad m\in \bN.
   \end{align}

\item (iii)  The sequence $F$ is called a {\it Riesz basis} of the
Hilbert space $\kH$ if its linear span is dense in $\kH$ and $F$
is both a Riesz-Fischer sequence and a Bessel sequence.
    \end{definition}
Note that the definitions of  Riesz-Fischer and Bessel sequences
given in
 \cite{young80} are different, but they are equivalent to the preceding  definition according to \cite[Theorem 4.3]{young80}.

The following proposition contains some slight reformulations of these notions.

If $\kA=(a_{kj})_{k,j\in \bN}$ is an infinite  matrix of complex entries $a_{kj}$ we shall say that $\kA$ defines  a bounded
operator $A$ on the  Hilbert space $l^2(\bN)$ if
  \begin{align}\label{matrixoperator}
\langle Ax,y\rangle =\sum_{k,j=1}^\infty
a_{kj}x_k\overline{y_j}\qquad {\rm for} \qquad  x=\{x_k\}_{k\in \N},\
y=\{y_k\}_{k\in \N}\in l^2(\bN).
  \end{align}
Clearly, if $\kA$ defines a bounded operator $A$, then $A$ is
uniquely determined by equation (\ref{matrixoperator}).
\begin{proposition}
Suppose that $X=\{x_k\}_{1}^\infty$ is a sequence of pairwise distinct points of $\bR^n$ and  $g$ is a positive definite function given by (\ref{bochequa}) with measure $\mu$. Let $F=\{f_k:=e^{i ( \cdot, x_k)}\}_{k=1}^\infty$ denote the sequence of  exponential functions in the Hilbert space $L^2(\bR^n;\mu)$. Then:
\begin{itemize}
\item[\em (i)]  $F$ is a Riesz-Fischer sequence in  $L^2(\bR^n;\mu)$ if and only if $g$ is strongly $X$-positive definite.
\item[\em (ii)] $F$ is a Bessel sequence if and only if the Gram matrix
\begin{align}\label{grammatrix}
Gr_F = \big(\langle f_k,f_j \rangle_{L^2(\bR^n;\mu)}\big)_{k,j\in
\bN} = \big(g(x_k-x_j)\big)_{k,j\in \bN} =: Gr_X(g)
 \end{align}
defines a bounded operator on $l^2(\bN)$.
   \end{itemize}
      \end{proposition}
   \begin{proof}
Using equation (\ref{bochequa}) we easily derive
\begin{equation}\label{equrfbs}
\sum_{k,j=1}^{m}\xi_k\overline{\xi}_jg(x_k-x_j)=
 \int_{\bR^n}\left|\sum_{k=1}^{m}\xi_k e^{i(u,x_k)}\right|^2\,d\mu(u)=\int_{\bR^n}\left|\sum_{k=1}^{m}\xi_k f_k(u)\right|^2\,d\mu(u)
 = \|\sum_{k=1}^{m}\xi_k f_k\|_{L^2(\bR^n;\mu)}  \,
   \end{equation}
for arbitrary $m\in \bN$ and $\xi=\{\xi_1,\dots,\xi_m\}\in \bC^m$.
Both statements are immediate from  \eqref{equrfbs}.
     \end{proof}

Taking in mind further applications to the spectral theory
of self-adjoint realizations of  $\mathfrak{L}_3$
we will be concerned with
radial positive definite functions. Let us recall the corresponding
concepts.
   \begin{definition} Let $n\in \bN$.
  A function $f\in
C([0,+\infty))$    is called a {\it radial   positive definite}
function  of the class   $\Phi_n$  if $f(|\cdot|_n)$ is a
positive definite  function on  $\R^n$, i.e., if
$f(|\cdot|_n)\in\Phi(\mathbb{R}^n)$ .
\end{definition}
It is known that  $\Phi_{n+1}\subset\Phi_n$  and $\Phi_n\neq \Phi_{n+1}$ for any $n\in \bN$ (see, for instance, \cite{Trigub_1989}, \cite{Zastavnyi_2000}).

A characterization of  the class $\Phi_n$ is given by the  following {\it Schoenberg
 theorem}  \cite{Sch38_1,Sch38}, see,  e.g., \cite[Theorem~5.4.2]{Akh65}  or \cite{BCR,SK}. Let $\sigma_n$ denote
 the normalized surface   measure on  the unit  sphere $S^{n}.$
  \begin{theorem}\label{schoenbergtheorem}
 A function  $f$ on $[0,+\infty)$   belongs to  the   class $\Phi_n$  if and only if there exists
 a positive  finite Borel measure $\nu$ on $[0,\infty)$ such that
  \begin{equation}\label{schonberg}
f(t)=\int_{0}^{+\infty}\Omega_n(rt)\,d\nu(r), \qquad 
\,t\in [0,+\infty)\,,
   \end{equation}
where
   \begin{equation}\label{Omaga_n}
\Omega_n(|x|)=\int_{S^{n}}e^{i(u,x)}d\sigma_n(u), \qquad x\in\R^n.
  \end{equation}
Moreover, we have
  \begin{equation}\label{kernel}
  \Omega_n(t)=\Gamma\left(\frac{n}{2}\right)\,
  \left(\frac{2}{t}\right)^{\frac{n-2}{2}}J_{\frac{n-2}{2}}(t)=
  \sum_{p=0}^{\infty}\left(-\frac{t^2}{4}\right)^p\frac{\Gamma\left(\frac{n}{2}\right)}{p!\,\Gamma\left(\frac{n}{2}+p\right)},\qquad t\in [0,+\infty).
 \end{equation}
  \end{theorem}
The first three functions $\Omega_n$, $n=1,2,3,$ can be computed as
 \begin{equation}\label{3.13}
   \Omega_1(t)=\cos t,\quad  \Omega_2(t)=J_0(t),\quad  \Omega_3(t)=\frac{{\rm sin} t}{t}\, ,
 \end{equation}
 where $J_0$ is the Bessel function of first kind and order zero (see e.g., \cite{rw},  p. 261).

 It was proved in \cite{GMZ11}  using  Schoenberg's theorem
 that for each non-constant function $f\in \Phi_n$, $n\geq 2$,
 the function  $f(|\cdot|)$ is strictly $X$-positive definite for any {\rm finite} subset $X$ of $\bR^n$.

\subsection{Completely  monotone functions and strong $X$-positive definiteness}
\begin{definition}
A  function
 $f\in C[0,\infty)\cap C^{\infty}{(0,+\infty)}$ is called   completely
 monotone on   $[0,\infty)$  if
$(-1)^kf^{(k)}(t)\ge 0$   for all  $k\in\mathbb{N}\cup \{0\}$ and  $t>0$. The set of  such functions is denoted by $M[0,\infty)$.
  \end{definition}

By Bernstein's theorem \cite{Akh65}, p. 204, a function $f$ on $[0,+\infty)$ belongs to the class  $M{[0,\infty)}$
if and only if there exists a finite positive Borel measure $\tau$ on $[0,+\infty)$ such that
   \begin{equation}\label{3.50}
f(t) = \int^{\infty}_0 e^{-ts}d\tau(s), \qquad   t\in[0,+\infty).
   \end{equation}
 The measure $\tau$ is then uniquely determined by the function $f$.

Schoenberg noted in  \cite{Sch38_1,Sch38} that a function $f$ on $[0,+\infty)$ belongs to
$\bigcap\limits_{n\in\mathbb{N}}\Phi_n$  if and only if
  $f(\sqrt{\cdot})\in M[0,\infty)$.
The following statement is an immediate consequence of
Schoenberg's result.

  \begin{proposition}
If  $f\in M [0,\infty)$, then
$f\in\bigcap\limits_{n\in\mathbb{N}}\Phi_n$.
       \end{proposition}
   \begin{proof}
For $s\ge 0$ the function $g_s(t) := e^{-s\sqrt{t}}$
is completely
monotone for $t>0$.  Schoenberg's result applies to $g_s(t^2)$ and shows that
$g_s(t^2)=e^{-s t}\in\bigcap\limits_{n\in\mathbb{N}}\Phi_n$. Therefore the integral representation \eqref{3.50} implies that
$f(\cdot)\in\bigcap\limits_{n\in\mathbb{N}}\Phi_n$.
        \end{proof}

For any sequence $X= \{x_k\}_1^\infty$ of points of ${\mathbb R^n}$ we set
 $$d_\ast(X):={\rm inf}_{k\neq j}~|x_k{-}x_j|.$$

The following  proposition  describes a large class of radial positive-definite functions
that are strongly $X$-positive-definite  for any sequence $X$ of points of ${\mathbb R^3}$ such that $d_*(X)>0$.
   \begin{theorem}\label{propositionstronglydef}
Let $f$ be a  nonconstant function of $ M[0,\infty)$ and let $\tau$ be the representing measure
in equation  (\ref{3.50}). Suppose that $X= \{x_k\}_1^\infty$ is
a sequence of points $x_k \in{\mathbb R}^3$.
Then:

(i)  If $d_\ast(X)>0$, then  the function
$f(|\cdot|)$ is strongly $X$-positive definite.

(ii)\ Suppose that $d_\ast(X)>0$ and
\begin{align}\label{fourthmoment}
\int_0^\infty (s + s^{-3})d\tau(s) < \infty.
  \end{align}
Then the Gram matrix $Gr_X(f) =  \big(f(|x_k-x_j|)\big)_{k,j\in
\bN}$ defines a bounded operator with bounded inverse on
$l^2(\bN).$

(iii)\ If the Gram matrix $Gr_X(f)$ defines a bounded
operator with bounded inverse on $l^2(\bN)$, then $d_*(X) > 0.$
   \end{theorem}

Theorem   \ref{propositionstronglydef} will be proved in
Section \ref{rieszbasisn1} below. We restate some results
derived in this proof  in the following corollary. Let
$\widetilde{\Phi}=\{\widetilde{\varphi}_j\}^\infty_{j=1}$, where
    \begin{equation}
\widetilde{\varphi}_j(x):=\frac{1}{\sqrt{2\pi}}\int_0^{+\infty}\frac{e^{-s|x-x_j|}}{|x-x_j|}~d\tau(s),
\qquad j\in \N.
    \end{equation}
   \begin{corollary}
Suppose  $X=\{x_j\}_{j=1}^\infty$ is a sequence of points of\  ${\mathbb
R}^3$ and $\tau$ is a  finite positive Borel measure  on
$[0,+\infty)$. Then:

(i) If $d_*( X)>0$ and $\tau((0,+\infty))>0$, then  $\widetilde{\Phi}$ forms a Riesz-Fischer sequence
in $L^2({\mathbb R}^3)$.

(ii) If $d_\ast(X)>0$ and (\ref{fourthmoment}) holds, then  ${\widetilde{\Phi}}$ is a Bessel sequence in $L^2({\mathbb R}^3)$.

(iii)  If $d_*( X)>0$  and (\ref{fourthmoment}) is satisfied, then ${\widetilde{\Phi}}$ forms a Riesz basis in
 its closed linear span.

(iv)\  If  the sequence $\widetilde{\Phi}$ is both  a
Riesz-Fischer and a Bessel sequence in $L^2({\mathbb R}^3)$, then
$d_*(X)>0$.
    \end{corollary}
 An immediate consequence of the preceding corollary is
    \begin{corollary}
    Let $f$, $X$ and $\tau$ be as in Theorem \ref{propositionstronglydef} and assume that condition \eqref{fourthmoment} holds.
    Then the sequence $\widetilde{\Phi}=\{\widetilde{\varphi}_j \}_1^\infty$ forms a Riesz basis in its closed linear
span if and only if $d_\ast(X)>0$.
\end{corollary}

    \begin{remark}
Let $f$ be an absolutely monotone function with integral
representation \eqref{3.50}. Then
   \begin{equation}\label{2.16A}
Gr_X(f) = \bigl(f(|x_j - x_k|)\bigr)_{j,k\in{\N}} =
\bigl(\langle\widetilde{\varphi}_j,\widetilde{\varphi}_k\rangle_{L^2({\R}^3)}\bigr)_{j,k\in{\N}}
= Gr_{\widetilde{\Phi}}.
   \end{equation}
    \end{remark}
%
%
    \begin{proposition}\label{setkrf}
Suppose that $f\in \Phi_n$  and let $\nu$ be the corresponding
representing measure from \eqref{schonberg}. Let
$X=\{x_k\}^\infty_1$ be an arbitrary sequence from $\bR^n$. Then
$f$ is  {\it strongly $X$-positive definite} if and only if there
exists a Borel subset  $\mathcal K \subset (0,+\infty)$ such that
$\nu(\mathcal K)>0$ and the system $\{e^{i(
\cdot,x_k)}\}_{k=1}^\infty$ forms a {\it Riesz-Fischer sequence}
in $L^2(S^n_{r};\sigma_n)$ for every $r\in \mathcal K.$
   \end{proposition}
\begin{proof} From
\eqref{schonberg} and \eqref{Omaga_n} it follows that for
$(\xi_1,\dots,\xi_m)\in \bC^m$ and  $m\in \bN$,
  \begin{equation}\label{intfxkxjAA}
\sum_{j,k =1}^{m}\xi_j\overline{\xi}_kf(|x_j - x_k|)=
 \int_{0}^{+\infty}\left(\;
 \int_{S^n}\left|\sum_{k=1}^{m}\xi_k e^{i(u,rx_k)}\right|^2\,d\sigma_n(u)
 \right)\,d\nu(r).
 \end{equation}

Suppose that there exists a set $\mathcal K$ as stated above.
Then for every $r \in \mathcal K$ there is a constant $c(r)>0$ such that
  \begin{equation}\label{2.18B}
 \big\|\sum_{k=1}^{m}\xi_k e^{i(u,rx_k)}\big\|^2_{L^2(S^n)}   \ge  c(r)
\sum_{k=1}^m | \xi_k|^2.
     \end{equation}
Choosing $c(r)$ measurable and
combining this inequality with  \eqref{intfxkxjAA} we obtain

  \begin{equation}\label{intfxkxjAB}
\sum_{j,k =1}^{m}\xi_j\overline{\xi}_kf(|x_j - x_k|)=
 \int_{\mathcal K} \left(\big\|\sum_{k=1}^{m}\xi_k e^{i(u,rx_k)}\big\|^2_{L^2(S^n)}
 \right)\,d\nu(r)  \ge  c
\sum_{k=1}^m | \xi_k|^2,
 \end{equation}
where $c := \int_{\mathcal K}c(r)d\nu(r).$ Since $\nu(\mathcal K)>0$ and $c(r)>0$, we  have $c>0$. That is, $f$ is strongly $X$-positive definite.

The converse follows easily from equation \eqref{intfxkxjAA}.
\end{proof}
%
%
  \begin{remark} Of course, the set $\mathcal K$ in Proposition \ref{setkrf} is not unique in general.
If the measure $\nu$ has an atom $r_0\in (0, +\infty)$, i.e.,
$\nu(\{r_0\})>0$, then one can choose $\mathcal K=\{r_0\}$. For
instance, for the function $f(\cdot)=\Omega_n(r_0 \cdot)$ the
representative measure from  formula (\ref{schonberg}) 
is the delta measure $\delta_{r_0}$ at $r_0$. Therefore,
  $f(\cdot)=\Omega_n(r_0 \cdot)$  is strongly  $X$-positive definite  if and only if
 the system $\{e^{i( \cdot, x_k)}\}_{k=1}^\infty$
forms a {\it Riesz-Fischer sequence} in $L^2(S^n_{r_0};\sigma_n).$
  \end{remark}

\subsection{Strong $X$-positive definiteness of  functions of the class  $\Phi_n$}

Let  $\Lambda=\{\lambda_k\}_1^\infty$ be a sequence of reals.
 For $r>0$ let $n(r)$ denote the largest number of points $\lambda_k$ that are contained in an interval of length $r.$
 Then the {\it upper density} of $\Lambda$ is defined by
$$
D^\ast(\Lambda)=\lim_{r\to +\infty} n(r)r^{-1}.
$$
Since $n(r)$ is subadditive, it follows   that this limit always exists (see e.g.  \cite{beurling}).

\noindent
In what follows we need the classical result  on Riesz-Fischer sequences of exponents
in  $L^2(-a,a)$.
   \begin{proposition}\label{rieszfischeralla}
Let $\Lambda=\{\lambda_k\}_1^\infty$ be a real sequence and $a>0$.
Set $E(\Lambda):=\{e^{i\lambda_k x}\}_1^\infty$.
   \begin{itemize}
\item[\em (i)] If $d_\ast(\Lambda)>0$ and $D^\ast(\Lambda)<a/\pi$,
then $E(\Lambda)$  is a Riesz-Fischer sequence in $L^2(-a,a)$.
\item[\em (ii)] If $E(\Lambda)$ is a Riesz-Fischer sequence in
$L^2(-a,a)$, then $d_\ast(\Lambda)>0$ and $D^\ast(\Lambda) \leq
a/\pi$.
\end{itemize}
  \end{proposition}
Assertion (i) of Proposition \ref{rieszfischeralla} is a theorem
of A. Beurling \cite{beurling}, while assertion (ii) is a result
of H.J. Landau \cite{landau}, see e.g. \cite{young98} and
\cite{seip}. Proposition \ref{rieszfischeralla}
yields the following statement. 
   \begin{corollary}\label{rieszfischeralla1}
If $d_\ast(\Lambda)>0$ and $D^\ast(\Lambda)=0$, then $E(\Lambda)$ is a Riesz-Fischer sequence in $L^2(-a,a)$ for all $a>0.$
   \end{corollary}
 From this corollary it follows  that  $E(\Lambda)$ is a
 Riesz-Fischer sequence in $L^2(-a,a)$ for all $a>0$ if $\lim_{k\to \infty} (\lambda_{k+1}{-}\lambda_k)=+\infty.$

Now we are ready to state the main result of this subsection.
    \begin{proposition}\label{stronglyxdet}
Let  $f \in \Phi_n$,  $f\not = const,$
and let $X=\{x_k\}_1^\infty$ be a sequence of points $x_k\in
\bR^n$, $n\geq 2,$  of the form $x_k=(0,x_{k2},\dots,x_{kn})$. If
the sequence $X_n:=\{x_{kn}\}_{k=1}^\infty$ of n-th coordinates
satisfies the conditions $d_\ast(X_n)>0$ and $D^\ast(X_n)=0$,
then $f$ is strongly $X$--positive definite.
  \end{proposition}
  \begin{proof}
By Schoenberg's theorem \ref{schoenbergtheorem},  $f$ admits a
representation  (\ref{schonberg}).
 Let $\xi=(\xi_1,\dots,\xi_m)\in \bC^m$, $m\in \bN$. It follows
from  \eqref{schonberg} and \eqref{Omaga_n}  that
  \begin{equation}\label{intfxkxj}
\sum_{k,j=1}^{m}\xi_k\overline{\xi}_jf(|x_k-x_j|)=
 \int_{0}^{+\infty}\left(\;
 \int_{S^n}\left|\sum_{k=1}^{m}\xi_k e^{i(u,rx_k)}\right|^2\,d\sigma_n(u)
 \right)\,d\nu(r).
 \end{equation}
Next, we  transform the integral over $S^n$ in \eqref{intfxkxj}.
Recall  that in terms of spherical coordinates 
  \begin{align*}
u_1&={\rm cos}~\varphi_1,\quad   u_{n-1} ={\rm sin}~\varphi_1
\cdots {\rm sin}~\varphi_{n-2}~{\rm cos}~\varphi_{n-1}, \quad
u_n={\rm sin}~\varphi_1 \cdots {\rm sin}~\varphi_{n-2}~{\rm
sin}~\varphi_{n-1},\\ &
~~~\varphi_1,\dots,\varphi_{n-2}\in [0,\pi] \qquad {\rm and}\qquad
\varphi_{n-1}\in [0,2\pi],
   \end{align*}
the surface measure $\sigma_n$ on the  unit sphere $S^{n}$ is
given by
  \begin{align*}
d\sigma_n(u)\equiv d\sigma_n(u_1,\dots,u_n)= {\rm sin}^{n-2}
\varphi_1~{\rm sin}^{n-3} \varphi_2 \cdots {\rm sin}~
\varphi_{n-2}~ d\varphi_1\cdots d\varphi_{n-1}.
 \end{align*}
Set $v=(u_2,\dots,u_n)$ and  $B_{n-1}:=\{v\in \R^{n-1}:\  |v|\leq 1 \}$. Writing
$u\in S^n$ as $u=(u_1,v) $, we derive from the previous  formula
    \begin{align}\label{surfaceuv}
d\sigma_n(u)= \frac{1}{\sqrt{1-|v|^2}}~ dv, \quad {\rm where}\quad
u_1^2+|v|^2=1, \qquad  v\in B_{n-1}.
   \end{align}
Further, we write  $v=(w,t)$, where $w\in \bR^{n-2}$ and $t\in
\bR$, and $x_k=(0,x_{2k},\dots,x_{nk})=(0,y_k,x_{kn})$, where
$y_k\in \bR^{n-2}$. Then we have $(u,rx_k)= r(w,y_k)+rtx_{kn}$.
Let $B_{n-2}$ denote the unit ball
 $B_{n-2}:=\{w\in \bR^{n-2}: |w|\leq 1\}$ in $\bR^{n-2}$.  Using the equality (\ref{surfaceuv}) we then compute
     \begin{align}\label{surfaceuv1}
\int_{S^n}\left|\sum_{k=1}^{m}\xi_k e^{i(u,rx_k)}\right|^2~d\sigma_n(u)&=
\int_{B_{n-1} } \left|\sum_{k=1}^{m}\xi_k e^{ir(w, y_k)} e^{i rtx_{nk}}\right|^2~\frac{1}{\sqrt{1{-}|v|^2}}~ dv\\ &\geq \int_{B_{n-1} } \left|\sum_{k=1}^{m}\xi_k e^{ir(w, y_k)} e^{i rtx_{nk}}\right|^2~~ dv\nonumber\\ & =
\int_{B_{n-2}} \left(\int_{-\sqrt{1{-}|w|^2} }^{\sqrt{1{-}|w|^2} }\left|\sum_{k=1}^{m}\xi_k e^{ir(w, y_k)} e^{irtx_{nk}}\right|^2~dt\right) dw\nonumber \\ &=\int_{B_{n-2}}r^{-1} \left(\int_{-r\sqrt{1{-}|w|^2} }^{r\sqrt{1{-}|w|^2} }
\left|\sum_{k=1}^{m}(\xi_k e^{ir(w, y_k)}) e^{isx_{nk}}\right|^2~ds\right) dw.\label{surfaceuv2}
       \end{align}
Since  $d_\ast(X_n)>0$ and $D^\ast(X_n)=0$ by assumption,
Corollary \ref{rieszfischeralla1} implies that  for any $a>0$ the
sequence $\{e^{i sx_{kn}}\}_{k=1}^\infty$ is a Riesz-Fischer sequence in
$L^2(-a,a)$. That is,  there exists a constant $c(a)>0$ such that
    \begin{align*}
\int_{-a}^a \left| \sum_{k=1}^{m}(\xi_k e^{ir(w, y_k)}) e^{isx_{nk}}\right|^2 ds
\geq c(a)\sum_{k=1}^m |~\xi_k e^{ir(w, y_k)}|^2
=c(a)\sum_{k=1}^m|~\xi_k|^2.
   \end{align*}
Inserting this inequality, applied
with $a=r\sqrt{1{-}|w|^2}>0$, into (\ref{surfaceuv2}) and then (\ref{surfaceuv2}) into (\ref{intfxkxj}) we obtain
    \begin{align*}
\sum_{k,j=1}^{m}\xi_k\overline{\xi}_jf(|x_k-x_j|)  &\geq
 \int_{0}^{+\infty}\left(\int_{B_{n-2}}r^{-1} \left(\int_{-r\sqrt{1{-}|w|^2} }^{r\sqrt{1{-}|w|^2} }
\left|\sum_{k=1}^{m}(\xi_k e^{ir(w, y_k)}) e^{isx_{nk}}\right|^2~ds\right) dw
 \right) d\tilde{\nu}(r)\\
&\geq \int_{0}^{+\infty}\left(\int_{B_{n-2}}r^{-1} c(r\sqrt{1{-}|w|^2}~)~\bigg(\sum_{k=1}^m|~\xi_k|^2\bigg)~ dw\right) d\tilde{\nu}(r)\\
& \geq \left( \int_{0}^{+\infty}\int_{B_{n-2}}r^{-1} c(r\sqrt{1{-}|w|^2}~)~ dw d\tilde{\nu}(r)\right)\sum_{k=1}^m|~\xi_k|^2~.
  \end{align*}
The double integral in front of the last sum is a finite constant,
say $\gamma$, by construction. Since $f$ is not constant by
assumption, $\tilde{\nu}((0,+\infty))>0$.  Therefore, since
$r^{-1}c(r\sqrt{1{-}|w|^2}~)>0$ for all $r>0$ and $|w|<1$, we
conclude that $\gamma >0$. This shows that $f$ is strongly
$X$-positive definite.
   \end{proof}

Assuming  $f\in\Phi_{n+1}$ rather than $f\in \Phi_{n}$ we obtain the following corollary.
     \begin{corollary}
Assume that $f\in\Phi_{n+1}$ and $f$ is not constant. Let $X=\{x_k\}_1^\infty$
be a sequence of points  $x_k=(x_{k1},x_{k2},\dots,x_{kn}) \in
\bR^n.$ If the sequence $X_n:=\{x_{kn}\}_{k=1}^\infty$ of n-th
coordinates satisfies the conditions $d_\ast(X_n)>0$ and
$D^\ast(X_n)=0$,
then $f$ is strongly $X$--positive definite.
   \end{corollary}
   \begin{proof}
We identify ${\mathbb R}^n$ with the subspace $\mathbf 0\oplus{\R}^n$ of ${\mathbb
R}^{n+1}$.  Then $X$ is identified with the sequence
$\hat{X} =\{(0,x_{k })\}_{k=1}^\infty$.
Since $f\in\Phi_{n+1}$,  Proposition \ref{stronglyxdet} applies
to the sequence $\hat{X}$, so $f$ is strongly
$\hat{X}$-positive definite. Hence it is strongly $X$-positive
definite.
   \end{proof}

The next proposition gives a more precise result for a sequence
 $X=\{x_k\}_{k=1}^\infty$ of $\bR^3$ which are located on a line.
    \begin{proposition}\label{pointsr3}
Suppose that  $\Lambda=\{\lambda_k\}_1^\infty$ is a real sequence and  $r>0$. Let $X$ be the sequence $X=\{x_k{:=}(0,0,\lambda_k)\}_{k=1}^\infty$ in $\bR^3$ and let $f_r(x):=\Omega_3(r|x|)$, $x\in \bR^3$.
 \begin{itemize}
\item[\em (i)]
If $d_\ast(\Lambda)>0$ and $D^\ast(\Lambda)< r/\pi$, then the functions $f_r$ is strongly $X$-positive definite.
\item[\em (ii)] If $f_r$ is strongly $X$-positive definite, then  $d_\ast(\Lambda)>0$ and $D^\ast(\Lambda)\leq r/\pi$.
\end{itemize}
  \end{proposition}
       \begin{proof}  Suppose that $\xi=(\xi_1,\dots,\xi_m)\in \bC^m$, $m\in \bN$.
We introduce spherical coordinates on the  unit sphere $S^2$ in $\bR^3$ by
  \begin{align*}
u_1={\rm sin}~\theta ~{\rm cos}~\varphi,\quad u_2={\rm sin}~
\theta ~{\rm sin}~\varphi, \quad u_3={\rm cos}~ \theta,\qquad {\rm
where}\quad  \theta\in [0,\pi], \varphi\in [0,2\pi].
  \end{align*}
Then the surface measure $\sigma_2$ on the sphere $S^2$ is given
by $d\sigma_2(u)={\rm sin}~\theta d \varphi d\theta$ and   $( u,r
x_k) =r \lambda_k~{ \rm cos}~\theta$. Using these facts and
equation  (\ref{Omaga_n})  we compute
\begin{align*}
\sum_{k,j=1}^m \xi_k\overline{\xi}_j~f_r(|x_k-x_j|)& =\sum_{k,j=1}^m \xi_k\overline{\xi}_j~\Omega_3(r|x_k-x_j|)=
 \int_{S^2}\left|\sum_{k=1}^{m}\xi_k e^{i(u, r x_k)}\right|^2\,d\sigma_2(u)\\ &=
  \int_0^{2\pi} \int_{0}^{\pi}\left|\sum_{k=1}^m \xi_k e^{i r\xi_k{\rm cos}~\theta}\right|^2{\rm sin}~\theta~ d \varphi d\theta= 2\pi \int_0^{\pi} \left|\sum_{k=1}^m \xi_k e^{i r\lambda_k{\rm cos}~\theta}\right|^2 {\rm sin}~\theta\,  d\theta~.
\end{align*}
Transforming the latter integral by setting
$t=r~{\rm cos}~\theta$ we obtain
\begin{align}\label{equalitym3}
\sum_{k,j=1}^m \xi_k\overline{\xi}_j~f(|x_k-x_j| = \frac{2\pi}{r} \int_{-r}^{r} \left|\sum_{k=1}^m \xi_k e^{i \lambda_k t}\right|^2 \,dt.
\end{align}
Equality  \eqref{equalitym3}  is the crucial step for the  proof of Proposition \ref{pointsr3}.

(i): Since $d_\ast (\Lambda)>0$ and $D^\ast(\Lambda)<r/\pi$, $E(\Lambda)=\{e^{i \lambda_k t}\}_{k=1}^\infty$ is is Riesz-Fischer sequence in $L^2(-r,r)$  by Proposition \ref{rieszfischeralla}(i). This means that  there exists a constant $c>0$ such that
\begin{align*}
\int_{-r}^r \left| \sum_{k=1}^{m}\xi_k e^{i\lambda_k t} \right|^2 dt \geq c\sum_{k=1}^m|~\xi_k|^2.
\end{align*}
Combined with (\ref{equalitym3}) it follows that $f$ is strongly $X$-positive definite.

(ii): Since $f$ is strongly $X$-positive definite, there is a constant $c>0$  such that
\begin{align*}
\sum_{k,j=1}^m \xi_k\overline{\xi}_j~f(|x_k-x_j| \geq c\sum_{k=1}^m|~\xi_k|^2
\end{align*}
Because of (\ref{equalitym3}) this implies that $E(\Lambda)$ is strongly $X$-positive definite. Therefore, $d_\ast (\Lambda)>0$ and $D^\ast(\Lambda)\leq r/\pi$ by Proposition \ref{rieszfischeralla}(ii).
\end{proof}
    \begin{corollary}\label{cor3.31}
Assume the conditions of Proposition  \ref{pointsr3} and $r_0>0.$
Then the functions $f_r$  are strongly $X$-positive definite for any
$r\in (0,r_0)$ if and only if
$d_\ast(\Lambda)>0$ and $D^\ast(\Lambda) =0$.
    \end{corollary}

\subsection{Boundedness of  Gram matrices}

Here  we discuss  the question of when the Gram matrix (\ref{grammatrix}) defines a  bounded operator on $l^2(\bN)$.
A standard criterion for showing that a matrix defines a bounded operator
is {\it Schur's test}. It can  be stated as follows:
   \begin{lemma}\label{schurtest}
Let $\kA=(a_{kj})_{k,j\in \bN}$ be an infinite hermitian matrix  satisfying
    \begin{align}\label{schurtestas}
C:= {\rm sup}_{j\in \bN}~\sum_{k=1}^\infty |a_{kj}| <\infty.
   \end{align}
Then the  matrix $\kA$ defines a  bounded self-adjoint operator $A$
on $l^2(\bN)$ and we have $\|A\|\leq C$.
     \end{lemma}
A proof of  Lemma \ref{schurtest} can be found, e.g., in
\cite{young80}, p. 159.
 \begin{lemma}\label{schurtest1}
Let $\kA=(a_{kj})_{k,j\in \bN}$ be an infinite hermitian matrix.
Suppose that $\{a_{kj}\}_{k=1}^\infty\in l^2(\bN)$ for all $j\in \bN$ and
\begin{align}\label{schurtestas1}
 \lim_{m\to \infty} \bigg({\rm sup}_{j\geq m}~\sum_{k\geq m} |a_{jk}|\bigg)=0 .
   \end{align}
Then the  hermitian matrix $\kA=(a_{kj})_{k,j\in \bN}$ defines a
compact self-adjoint operator on $l^2(\bN)$.
     \end{lemma}
\begin{proof}
For $m\in \bN$ let   $\kA_m$ denote the matrix
$(a^{(m)}_{kj})_{k,j\in \bN}$, where $a^{(m)}_{kj}:=0$ if either
$k\geq m$ or $j\geq m$ and $a^{(m)}_{kj}=a_{kj}$ otherwise.
Clearly, $\kA_m$ defines a bounded operator $A_m$ on $l^2(\bN)$.
From (\ref{schurtestas1}) it follows that  the  matrix $\kA-\kA_m$
satisfies condition (\ref{schurtestas}) for large $m$, so
$\kA-\kA_m$ defines a bounded  operator $B_m$. Therefore  
$\kA$ defines the bounded self-adjoint operator
$A:=A_m+B_m$.

Let $\varepsilon >0$ be given. By  \eqref{schurtestas1}, there
exists $m_0$ such that $\sum_{k\geq m} |a_{jk}|<\varepsilon$ for
$m>m_0$ and $j>m_0$. Using the latter,  the  Cauchy-Schwarz
inequality and the relation $a_{kj}=a_{jk}$ we derive
\begin{align*}
\|B_m x\|^2 &= \sum_{j>m} \bigg|\sum_{k>m} a_{jk}x_k\bigg|^2\leq \sum_{j>m} \bigg(\sum_{k>m} | a_{jk}|\bigg) \bigg(\sum_{k>m} |a_{jk}||x_k|^2\bigg)\\&\leq \varepsilon
\sum_{k>m} \sum_{j>m}|a_{kj}||x_k|^2\leq \varepsilon^2 \sum_{k>m} |x_k|^2\leq \varepsilon^2 \|x\|^2
\end{align*}
for $x=\{x_j\}_1^\infty\in l^2(\bN)$ and $m>m_0$. This proves that $\lim_m \|B_m\|=\lim_m \|A-A_m\|=0$.
Obviously, $A_m$ is compact, because it has finite rank. Therefore,  $A$ is compact.
\end{proof}
An immediate consequence of Lemma \ref{schurtest1}  is the following corollary.
    \begin{corollary}\label{schurcor}
If $\kA=(a_{kj})_{k,j\in \bN}$ is  an infinite hermitian matrix satisfying
\begin{align}\label{schurtestas2}
 \lim_{m\to \infty} \bigg({\rm sup}_{j\in \bN}~\sum_{k\geq m} |a_{jk}|\bigg)=0 ,
   \end{align}
then the  matrix $\kA$ defines a  compact self-adjoint operator
on $l^2(\bN)$.
     \end{corollary}
        \begin{proposition}\label{mfxbounded}
Let  $f\in \Phi_n$, $n\geq 2$, and let $\nu$ be the representing
measure  in  equation (\ref{schonberg}). Let $X=\{x_k\}_1^\infty$
be a sequence of pairwise different points $x_k\in \bR^n$. Suppose
that for each $j,k\in \bN$, $j\neq k$, there are  positive numbers
$\alpha_{kj}$ such that
   \begin{align}\label{conditionsum}
K&:={\rm sup}_{j\in \bN}~\sum\nolimits^\prime_{k\in \bN} \frac{1}{(\alpha_{kj}|x_k-x_j|)^{\frac{n-1}{2}}}<\infty,\\
L&:={\rm sup}_{j\in \bN}~\sum\nolimits^\prime_{k\in \bN} \nu([0,
\alpha_{kj}])<\infty,\label{condzero}
   \end{align}
Then the matrix  $Gr_X(f) := (f(|x_k-x_j|)_{k,j\in \bN}$
defines a bounded self-adjoint operator on $l^2(\bN)$.
     \end{proposition}
      \begin{proof}
By (\ref{kernel}) the function $\Omega_n(t)$ has an alternating
power series expansion and $\Omega_n(0)=1$. Therefore we have
$\Omega_n(t)\leq 1$ for $t\in [0, \infty)$.
It is well-known (see, e.g., \cite{rw}, p. 266) that the Bessel
function $  J_{\frac{n-2}{2}}(t)$ behaves asymptotically as
$\sqrt{\frac{2}{\pi t}}$ as $t\to \infty$. Therefore, it follows
from  (\ref{kernel}) that there exists a constant $C_n$ such that
   \begin{align}\label{besselsymp}
|\Omega_n(t)|\leq C_n t^{\frac{1-n}{2}}\qquad{\rm for}\qquad
t\in (0, \infty).
      \end{align}
Using these facts and the assumptions (\ref{conditionsum}) and
(\ref{condzero}) we obtain
\begin{align*}
\sum\nolimits^\prime_{k\in \bN}  f(|x_k-x_j|)&=\sum\nolimits^\prime_{k\in \bN}~ \int_0^{\infty} \Omega_n(r|x_k-x_j|) ~d\nu(r)  \\
&\leq  \sum\nolimits^\prime_{k\in \bN}
\left(\int_0^{\alpha_{kj}}
1~ d\nu(r)+ C_n\int_{\alpha_{kj}}^{\infty}\big(r|x_k-x_j|\big)^{\frac{1-n}{2}} ~d\nu(r)\right)\\
&\leq \sum\nolimits^\prime_{k\in \bN} \nu([0, \alpha_{kj}]) +
\sum\nolimits^\prime_{k\in \bN} C_n
\int_{\alpha_{kj}}^{\infty} \big(\alpha_{kj}|x_k-x_j|\big)^{\frac{1-n}{2}} ~d\nu(r)\\
&=L +  C_n\bigg(\sum\nolimits^\prime_{k\in \bN}
(\alpha_{kj}|x_k-x_j|)^{\frac{1-n}{2}}\bigg)  \nu(\bR) \leq
L+ C_nK\,\nu(\bR),
 \end{align*}
so that
  \begin{align}
{\rm sup}_{j\in \bN}~\sum_{k=1}^\infty f(|x_k-x_j|) \leq f(0)+ L + C_nK \,\nu(\bR)<\infty.
 \end{align}
This shows that the assumption (\ref{schurtestas}) of the Schur
test is fulfilled, so  the matrix $Gr_X(f)$ defines  a bounded
operator  by Lemma \ref{schurtest}.
   \end{proof}
The  assumptions (\ref{condzero}) and (\ref{conditionsum}) are a
growth condition of the measure $\nu$ at zero combined with a
density condition for the set  of points $x_k$. Let us assume that
$\nu([0,\varepsilon])=0$ for some $\varepsilon >0$. Setting
$\alpha_{kj}=\varepsilon$ in Proposition \ref{mfxbounded},
(\ref{condzero}) is trivially satisfied and (\ref{conditionsum})
holds whenever
 \begin{align}\label{conditionsum1}
{\rm sup}_{j\in \bN}~\sum\nolimits^\prime_{k\in \bN}
\frac{1}{|x_k-x_j|^{\frac{n-1}{2}}}<\infty.
\end{align}
Because of its importance we restate this result in the special
case  when $\nu= \delta_{r}$ is a delta measure at $r\in (0,
\infty)$ separately as
\begin{corollary}
If $X=\{x_k\}_1^\infty$ is a sequence of pairwise distinct points $x_k\in
\bR^n$ satisfying  (\ref{conditionsum1}), then
 for any $r>0$ the infinite matrix
$\big(\Omega_n(r |x_k-x_j|)\big)_{k,j\in \bN}$
defines a bounded operator on $l^2(\bN)$.
\end{corollary}

Applying the Schur test one can derive a number of further results when
the matrices $Gr_X(f)$
and  $\big(\Omega_n(r |x_k-x_j|)\big)_{k,j\in \bN}$
 define bounded operators on $l^2(\bN)$. An example is the
next proposition.

\begin{proposition}\label{boundedr3}
Suppose  $X=\{x_k\}_1^\infty$ is a sequence of distinct points
$x_k\in \bR^3$ such that
  \begin{align}\label{matrixs1}
K := {\rm sup}_{j\in \bN}~\sum\nolimits_{k\in \N}^\prime
\frac{1}{|x_k{-}x_j|}\, <\infty.
\end{align}
Let $r\in (0,+\infty)$ and let $\kA$ be the infinite matrix given
by
    \begin{align}\label{defkmx}
\Omega_3(t,X) := \big(\Omega_3(t(|x_k-x_j|)\big)_{k,j\in
\bN}=\left(\frac{\sin~(t|x_k{-}x_j|)}{t|x_k{-}x_j|}~\right)_{k,j\in
\bN},
  \end{align}
where we  set $\frac{\sin~ 0}{0}:=1$. If $r^{-1}K<1$, then $\kA$
defines a bounded self-adjoint operator $A$ on $l^2(\bN)$ with
bounded inverse; moreover,  $\|A\|\leq 1+r^{-1}K$ and $\|A^{-1}\|
\leq (1{-}r^{-1}K)^{-1}.$
 \end{proposition}
       \begin{proof}
Set $\kS\equiv(a_{kj})_{k,j\in\bN}:=\kA-I$, where $I$ is  the
identity matrix. 
Since $a_{kk}=0,$ one has
  \begin{align*}
{\rm sup}_{j\in \bN} ~\sum\nolimits_k |a_{kj}| 
={\rm sup}_{j\in \bN}~\sum\nolimits^\prime_k
\left|\frac{\sin(r|x_k{-}x_j|)}{r|x_k{-}x_j|}\right| 
\leq r^{-1}~ {\rm sup}_{j\in \bN} \sum\nolimits^\prime_k
~\frac{1}{|x_k{-}x_j|}= r^{-1}K.
  \end{align*}
This  shows that the Hermitean matrix $\kS$ satisfies the
assumption (\ref{schurtestas}) of Lemma \ref{schurtest} with
$C\leq r^{-1}K$. Thus $\kS$ is the matrix of a bounded self-adjoint
operator $S$ such that $\|S\|\leq r^{-1}K$. We  have $\kS:=\kA-I$.
This implies that $\kA$ is the matrix of a bounded self-adjoint
operator $A=I+S$ and $\|A\|\leq 1+r^{-1}K$. Since $r^{-1}K<1$, $A$
has a  bounded inverse 
and $\|A^{-1}\| 
\leq (1{-}r^{-1}K)^{-1}.$
\end{proof}


\section{Riesz bases of  defect  subspaces  and the property of strong $X$-positive definiteness}\label{rieszbasisn1}

Let  $\Delta$ denote the  Laplacian on $\bR^3$ with
 domain $\dom(-\Delta) =  W^{2,2}(\bR^3)$ in $L^2(\bR^3).$
It is well known that $-\Delta$ is self-adjoint. We fix
a sequence $X=\{x_j\}_1^\infty$ of pairwise distinct points
$x_j\in \mathbb{R}^3$ and denote by $H$ the   restriction

  \begin{equation}\label{min} 
H:=-\Delta   \!\upharpoonright \dom H, \qquad    \dom H = \{f \in
W^{2,2}(\bR^3): f(x_j)=0 \quad {\rm for~all}\quad j\in \bN\}.
  \end{equation}
We abbreviate $r_j:=|x-x_j|$ for  $x=(x^1,x^2,x^3)\in\bR^3$.
For $z\in \bC\setminus [0,+\infty)$ we denote by  $\sqrt{z}$
 the branch of the square root  of $z$ with positive
imaginary part.

Further, let us recall the formula for the resolvent $(-\Delta -
zI)^{-1}$   on $L^2(\R^3)$ (see \cite{{MasN}}):
      \begin{equation}\label{3.2S}
((-\Delta - z I)^{-1}f)(x) = \frac{1}{4\pi}\int_{{\mathbb
R}^3}\frac{e^{i\sqrt{z}|x-t|}}{|x-t|}~f(t)~dt,\qquad f\in
L^2(\bR^3).
      \end{equation}
     \begin{lemma}\label{varphicomplete}
 The sequence $E := \big\{\frac{1}{\sqrt{2\pi}}\varphi_j\big\}_{j=1}^\infty=\big\{\frac{1}{\sqrt{2\pi}}\frac{e^{-|x- x_j|}}{|x- x_j|}\big\}_{j=1}^\infty$ is
 normed and complete in the defect subspace $\mathfrak N_{-1}\bigl(\subset
 L^2({\mathbb R}^3)\bigr)$ of the operator $H$.
    \end{lemma}
     \begin{proof}
Suppose that $f\in\mathfrak N_{-1}$ and  $f \perp E$. Then $u := (I-\Delta)^{-1}f\in
W^{2,2}({\mathbb R}^3).$ By  \eqref{3.2S}, we have
      \begin{equation}\label{3.2SM}
u(x)=\frac{1}{4\pi}\int_{{\mathbb
R}^3}\frac{e^{-|x-t|}}{|x-t|}f(t)dt.
      \end{equation}
Therefore, the orthogonality condition $f\perp E$ means that
    \begin{eqnarray}\label{3.3S}
0 = \langle f, \varphi_j\rangle =
\frac{1}{4\pi}\int_{{\mathbb
R}^3}{f(t)}\frac{e^{-|t - x_j|}}{|t - x_j|}~dt   
= u(x_j), \qquad  j\in{\mathbb N}.
  \end{eqnarray}
By \eqref{3.3S}  and \eqref {min},  $u\in\dom(H)$ and   $f = (I-\Delta)u = (I+H)u\in\ran(I+H)$. Thus,
    \begin{equation*}
f\in\mathfrak N_{-1}\cap\ran(I+H)=\{0\},
    \end{equation*}
i.e. $f=0$ and the system $E$ is complete.

The function $e^{-|\cdot|}(\in W^{2,2}(\R^3)$)  is a (generalized)
solution of the equation $(I - \Delta)e^{-|x|} = 2
\frac{\exp({-|x|})}{|x|}$.     Therefore it follows from
\eqref{3.2SM} with $f= f_y(x) := \frac{e^{-|x-y|}}{|x-y|}$  that
      \begin{equation}\label{3.2AS}
\frac{e^{-|x-y|}}{2} = \frac{1}{4\pi}\int_{{\mathbb
R}^3}\frac{e^{-|x-t|}}{|x-t|}\cdot \frac{e^{-|t-y|}}{|t-y|}~dt .
      \end{equation}
Setting here $x=y =x_j$   we get $\|\varphi_j\|^2 = 2\pi$, i.e.,
the system $E$ is normed.
    \end{proof}

In order to state the next result we need the following definition.

   \begin{definition}
A sequence $\{f_j\}_1^\infty$ of vectors of a Hilbert space is
called {\it $w$-linearly independent} if for any complex sequence
$\{c_j\}_1^\infty$  the relations
   \begin{equation}\label{3.5$}
\sum^{\infty}_{j=1}c_j f_j=0 \qquad {\rm and}\qquad
\sum^{\infty}_{j=1}|c_j|^2\|f_j\|^2<\infty
    \end{equation}
 imply that $c_j=0$ for all $j\in{\mathbb N}$.
   \end{definition}
   \begin{lemma}\label{lem4.3}
Assume that  $X=\{x_j\}_1^\infty$ has no finite
accumulation points. Then the sequence $E = \big\{\frac{1}{\sqrt{2\pi}}\varphi_j\big\}_{j=1}^\infty =
\big\{\frac{1}{\sqrt{2\pi}}~\frac{e^{-|x - x_j|}}{|x-
x_j|}\big\}_{j=1}^\infty$ is $\omega$-linearly independent in $\mathfrak H =
L^2({\mathbb R}^3)$.
        \end{lemma}
    \begin{proof}
Assume that for some complex sequence $\{c_j\}_1^\infty$
conditions \eqref{3.5$} are satisfied with
$\varphi_j$ in place of $f_j$.  By
Lemma  \ref{varphicomplete}, $\|\varphi_j\| = \sqrt{2\pi}.$ Hence
the second of conditions \eqref{3.5$}  is equivalent to
$\{c_j\}\in l^2.$  Furthermore, since  each function
$\varphi_{j}(x)$ is harmonic in ${\R}^3\setminus \{x_j\}$, this
implies that the  series $\sum^{\infty}_{j=1}c_j\varphi_{j}$
converges uniformly on each compact subset of ${\R}^3\setminus X$.

Fix $k\in \bN$. Since the points $x_j$ are pairwise distinct  and
the  set $X$  has no finite accumulation points, there exist a
compact neighborhood  $U_k$ of $x_k$ and such that $x_j\notin U_k$
for all $j\not =k$. Then, by the preceding considerations, the series
$\sum_{j\not =k}c_j\varphi_{j}$  converges uniformly on $U_k$.

From  the first equality of (\ref{3.5$}) it follows that
  \begin{equation*}
-c_k = \sum\nolimits_{j\in \N}^\prime
c_j e^{-|x-x_j|}|x{-}x_j|^{-1}|x{-}x_k|
  \end{equation*}
for all $x\in U_k$, $x\neq x_k$. Therefore, passing to the limit
as $x\to x_k$ we obtain $c_k=0$.
        \end{proof}
           \begin{definition}
$(i)$   A sequence $\{f_j\}_1^\infty$ in the Hilbert space $\mathfrak
H$ is called minimal if for any $k$
     \begin{equation}
\dist\{f_k, \mathfrak H^{(k)}\}= \varepsilon_k>0, \qquad   \mathfrak H^{(k)} := {\rm
span}\{f_j:\ j\in \N\setminus \{k\}\},\quad k\in{\mathbb N}.
     \end{equation}

$(ii)$ A sequence $\{f_j\}_1^\infty$   is said to be uniformly
minimal if ${\rm inf}_{k\in \N}~\varepsilon_k>0$.

$(iii)$  A sequence $\{g_j\}_1^\infty\subset \mathfrak H$ is called
biorthogonal to $\{f_j\}_1^\infty$ if $\langle f_j,g_k\rangle
=\delta_{jk}$ for all $j,k\in \N$.
           \end{definition}

           Let us recall two well-known facts (see e.g. \cite{GK65}):
A biorthogonal sequence to $\{f_j\}_1^\infty$ exists if and only if
the sequence $\{f_j\}_1^\infty$ is minimal. If this is true, then the
biorthogonal sequence is uniquely determined if and only if the
set $\{f_j\}_1^\infty$ is complete in $\mathfrak H$.

Recall that the sequence $\{\varphi_{j}\}$  is complete in
$\mathfrak N_{-1}$ according to Lemma \ref{varphicomplete}.

    \begin{lemma}\label{accuminimal}
Assume that  $X=\{x_j\}_1^\infty$ has no finite
accumulation points.

\item $(i)$ The  sequence $E := \{\varphi_j\}_1^\infty$ is minimal
in $\mathfrak N_{-1}$.

\item $(ii)$  The corresponding biorthogonal sequence
$\{\psi_j\}_1^\infty$  is also complete in $\mathfrak N_{-1}$.
        \end{lemma}
   \begin{proof}
$(i)$  To prove minimality it suffices to construct a biorthogonal
system.  Since $X$ has no finite accumulation
point, for any $j\in{\mathbb N}$ there exists a function
${\widetilde u}_j\in C_0^\infty({\mathbb R}^3)$ such that
   \begin{equation}\label{3.8S}
{\widetilde u}_j(x_j) = 1  \quad\text{and}\qquad  {\widetilde
u}_j(x_k)=0\qquad \text{for}\qquad  k\not =j.
   \end{equation}
Moreover, ${\widetilde u}_j(\cdot)$ can be chosen  compactly
supported in a small neighbourhood of $x_j.$

Let ${\widetilde \psi}_j:=(I-\Delta){\widetilde u}_j,\  j\in{\mathbb
N}$.  In general, ${\widetilde\psi}_j\notin\mathfrak N_{-1}$. To avoid
this drawback we put
    \begin{equation}
\psi_j:= P_{-1}{\widetilde\psi}_j\in\mathfrak N_{-1}\qquad \text{and}
\qquad g_j:={\widetilde\psi}_j - \psi_j, \qquad j\in{\N},
    \end{equation}
where $P_{-1}$ is the orthogonal projection in $\mathfrak H$ onto $\mathfrak
N_{-1}$. Then $g_j\in\ran(I+H) = \mathfrak H\ominus \mathfrak N_{-1},\
j\in{\N}$. Setting  $v_j=(I-\Delta)^{-1}g_j,$ we get
$v_j\in\dom (H)\subset \dom (\Delta)$. Therefore, by the Sobolev embedding theorem, $v_j\in C(\R^3)$. 
Together with the sequence $\{{\widetilde u}_j\}_1^\infty$ we consider the
sequence of functions
   \begin{equation}\label{3.8AS}
u_j:={\widetilde u}_j-v_j\in W^{2,2}(\R^3), \qquad   j\in{\mathbb N}.
   \end{equation}
Since $v_j\in\dom(H)$, the functions $u_j$
satisfy relations \eqref{3.8S} as well.
Thus,
    \begin{equation}\label{3.9S}
-\Delta u_j + u_j= \psi_j\in\mathfrak N_{-1}\quad\text{and}\quad
u_j(x_k)=\delta_{kj} \quad\text{for}\quad  j,k\in{\N}.
   \end{equation}
Combining these relations with the resolvent formula \eqref{3.2S}
we get
    \begin{equation}\label{3.7$}
\langle \psi_j,\varphi_k  \rangle  = \frac{1}{4\pi}\int_{{\mathbb
R}^3}\psi_j(x)\frac{e^{-|x - x_j|}}{|x - x_j|}dx = (I
-\Delta)^{-1} \psi_j  = u_j(x_k) = \delta_{kj}, \qquad j,k\in{\mathbb
N}.
    \end{equation}
These  relations mean  that the sequence $\{\psi_j\}_1^\infty$ is
biorthogonal to the sequence $\{\varphi_j\}_1^\infty$. Hence the
latter is minimal.

$(ii)$  Let  $\mathfrak H_1$ denote the closed
linear span of the set $\{u_j;j\in \bN\}$ in $W^{2,2}({\mathbb
R}^3).$

We prove that  $W^{2,2}({\mathbb R}^3)$ is the closed linear span of its subspaces
$\mathfrak H_1$ and $ \dom(H).$
Indeed, assume that  $g\in W^{2,2}({\mathbb R}^3)$ and has a compact
support $K=\supp g$. Then the intersection $X \cap K$ is finite
since $X$ has no accumulation points. Therefore  the function
     \begin{equation}
g_1=\sum_{x_j\in K}g(x_j)u_j
     \end{equation}
is well defined and $g_1\in \mathfrak H_1$, It follows from
\eqref{3.9S}  that $g_0:= g - g_1\in\dom(H)$ and $g = g_1 + g_0$.
It remains to note that $C^{\infty}_0({\mathbb R}^3)$ is dense in
$W^{2,2}({\mathbb R}^3)$.

Suppose that $f\in\mathfrak N_{-1}$ and $\langle f,\psi_j \rangle =
0,\ j\in{\mathbb N}$. Then, by  \eqref{3.9S},
      \begin{equation}\label{4.18S}
0 = \langle  f,\psi_j \rangle  =  \langle f, (-\Delta + I)u_j
\rangle ,  \qquad j\in{\N}.
      \end{equation}
The inclusion  $f\in\mathfrak N_{-1}$ means that
$f\perp(I-\Delta)\dom(H)$.  Combining this   with
\eqref{4.18S} and using that $ W^{2,2}({\mathbb R}^3)$ is the  closure of $\mathfrak H_1+ \dom(H)$ as shown above, it follows that
$f\perp \ran(I-\Delta) = L^{2}({\mathbb R}^3)$. Thus
$f=0$ and the sequence  $\{\psi_j\}_1^\infty$ is complete.
   \end{proof}
   \begin{lemma}
If   $E= \{\varphi_j\}_1^\infty$ is uniformly minimal,
then $X$ has no finite accumulation points.
       \end{lemma}
    \begin{proof}
 Since   $\{\varphi_j\}_1^\infty$ is minimal in
$\mathfrak N_{-1}$,    
there  exists the biorthogonal sequence
$\{\psi_j\}_1^\infty$ in $\mathfrak N_{-1}$.
It was already mentioned that the uniform minimality of
$E=\{\varphi_j\}_1^\infty$ is equivalent to  $\sup_{j\in{\mathbb
N}}\|\varphi_j\|\cdot\|\psi_j\|<\infty$. Therefore,  since  $\|\varphi_j\|=2\sqrt \pi$, by Lemma \ref{varphicomplete}, the
sequence $(\psi_j;j\in \bN)$ is uniformly bounded, i.e.
$\sup_j\|\psi_j\| =: C_0 <\infty$.
Setting   $u_j = (I -\Delta)^{-1}\psi_j \in W^2_2({\mathbb R}^3)$ we
conclude that  the sequence $\{u_j\}_1^\infty$ is uniformly bounded in
$W^{2,2}({\mathbb R}^3),$ that is, $ \sup_{j\in \N}\|u_j\|_{W^{2,2}} = C_1<\infty.$

 Now assume  to the contrary that there is a finite
accumulation point $y_0$ of $X$. Thus, there exists  a
subsequence $\{x_{j_m}\}_{m=1}^\infty$ such that $y_0 =
\lim_{m\to\infty}x_{j_m}$.
By the Sobolev embedding theorem, the set
$\{u_j;j\in \bN\}$ is compact in $C({\mathbb R}^3)$. Thus there
exists a subsequence of $\{u_{j_m}\}$ which
converges uniformly to $u_0\in C({\mathbb R}^3)$. Without loss of generality
we assume that the sequence $\{u_{j_m}\}$ itself converges to $u_0$,
i.e. $\lim_{m\to\infty}||u_{j_m} - u_0\|_{C({\mathbb R}^3)}=0$. Hence
    \begin{eqnarray*}
1= u_{j_m}(x_{j_{m}})\underset{m\to\infty}{\to} u_{0}(y_0) = 1, \qquad  \
0=u_{j_m}(x_{j_{m-1}})\underset{m\to\infty}{\to} u_{0}(y_0)=0,
  \end{eqnarray*}
which is the desired contradiction.
     \end{proof}
\begin{lemma}\label{lemma4.7}
Suppose that $d_*(X) = 0$. If the matrix $\kT_1:=(\frac{1}{2}~e^{-|x_j -
x_k|})_{j,k\in{\mathbb N}}$ defines a bounded self-adjoint operator $T_1$ on $l^2(\N)$, then
 $0\in\sigma_c(T_1)$, hence $T_1$ has no bounded inverse.
   \end{lemma}
    \begin{proof}
Let $\varepsilon>0$. Since $d_*(X)=0$, there exist numbers
$ n_j\in{\mathbb N}$ such that
$r_{jk}:=|x_{n_j}-x_{n_k}|<\varepsilon$. Let $e_n$ denote the  vector
$e_n := \{\delta_{p,n}\}^{\infty}_{p=1}$ of $ l^2({\mathbb N})$. Then
$2~ T_1(e_j - e_k) = \{e^{-r_{pj}}-e^{-r_{pk}}\}^\infty_{p=1} \in l^2({\mathbb
N}).$

   Since
$|r_{pj}-r_{pk}|\le r_{jk}<\varepsilon $  by the triangle
inequality,  $e^{-\varepsilon}\le\exp(r_{pj}-r_{pk})\le
e^{\varepsilon}$ and hence
   \begin{equation*}
|e^{-r_{pj}} - e^{-r_{pk}}| = e^{-r_{pj}}|1 - e^{r_{pj} -
r_{pk}}|\le \varepsilon Ce^{-r_{pj}}, \qquad j,k,p\in{\N},
   \end{equation*}
where $C>0$ is a constant. Using the assumption that $T_1$ is
bounded we get
    \begin{equation}
4\|T_1(e_j - e_k)\|^2 \le\varepsilon^2 C^2 \sum\nolimits_p
e^{-2r_{pj}} = 4\varepsilon^2 C^2~\|T_1 e_j\|^2  \le 4\varepsilon^2
C^2~\|T_1\|^2.
  \end{equation}
Since $\varepsilon>0$ is arbitrary and  $\|e_j-e_k\| =\sqrt{2}$
for $j\neq k$, it follows that $0\in\sigma_c(T_1)$.
    \end{proof}
   \begin{theorem}\label{rieszbasisN1}
The sequence $E=
\{\varphi_j\}_1^\infty$  forms a Riesz basis of the Hilbert
space $\mathfrak N_{-1}$ if and only if $d_*(X) >0$.
       \end{theorem}
     \begin{proof}
\emph{Sufficiency.}  Suppose  that $d_\ast(X)>0$.   By Lemmas
\ref{varphicomplete} and
 \ref{accuminimal}, both sequences $\{\varphi_j\}_1^\infty$ and
$\{\psi_j\}_1^\infty$ are complete in $\mathfrak N_{-1}$. Therefore, by
\cite[Theorem 6.2.1]{GK65}),  the sequence $\{\varphi_j\}$ forms a
Riesz basis in $\mathfrak N_{-1}$ if and and only if
  \begin{equation}\label{4.22S}
\sum_{j=1}^\infty |\langle f, \varphi_j \rangle|^2 < \infty
\qquad\text{and}\qquad \sum_{j=1}^\infty | \langle f, \psi_j
\rangle|^2 < \infty\qquad\text{for all}\quad  f\in \mathfrak N_{-1}.
  \end{equation}

Let $B_j$ denote the ball in $\R^3$ centered at $x_j$ with the
radius $r=d_*(X)/3$, $j\in \N$. Clearly $B_j \cap B_k= \emptyset$
for $j\not = k.$  By the Sobolev embedding theorem, there is a
constant $C>0$ such that
   \begin{equation}\label{4.20SS}
|v(x_j)| \le C \|v\|_{W^{2,2}(B_j)},\qquad  v\in W^{2,2}(B_j),
\qquad  j\in \N,
    \end{equation}
where  $C$ is independent  of $j$ and $v\in W^{2,2}(B_j).$

Let  $f\in \mathfrak N_{-1}$ and set $u=(I-\Delta)^{-1}f$ $u\in
W^{2,2}(\R^3)$.  Combining   \eqref{4.20SS}  with the
representation \eqref{3.2S} for $u$  we get
  \begin{equation}
\sum_{j=1}^\infty |(f, \varphi_j)|^2 =  \sum_{j=1}^\infty
|u(x_j)|^2 \le C\sum_{j=1}^\infty \|u\|^2_{W^{2,2}(B_j)}\le
C\|u\|^2_{W^{2,2}(\R^3)}, \quad f\in \mathfrak N_{-1}.
  \end{equation}
This proves the first  inequality of  \eqref{4.22S}.

  We now derive the second inequality. Let   $B_0$ be the ball centered at zero  with the radius $r=d_*(X)/3$.
We choose a function ${\widetilde u}_0\in C^{\infty}_0({\mathbb R}^3)$
supported in  $B_0$ and satisfying ${\widetilde u}_0(0)=1$. Put
    \begin{equation}\label{418S}
{\widetilde u}_j(x) := {\widetilde u}_0(x-x_j), \qquad j\in{\mathbb
N}.
        \end{equation}
Clearly,  the sequence $\{\widetilde u_j\}_1^\infty$ satisfies
conditions \eqref{3.8S}. Then repeating the reasonings of the
proof of Lemma \ref{accuminimal}(i)  we find a sequence
$\{v_j\}_1^\infty$ of vectors from $\dom(H)$  such that the new sequence $\{u_j:= {\widetilde u}_j - v_j\}_1^\infty$
 satisfies relations \eqref{3.9S}. Hence for any $f\in{\mathfrak
N}_{-1}$ we have
   \begin{eqnarray}\label{4.20S}
\langle f,\psi_j \rangle = \langle f, (-\Delta + I)u_j \rangle
 = \langle f, (-\Delta + I) ({\widetilde u}_j - v_j) \rangle =
\langle f, (-\Delta+I){\widetilde u}_j \rangle,  \quad j\in{\mathbb
N}.
   \end{eqnarray}
Since ${\widetilde u}_j(\cdot)$ is supported in the ball $B_j,$ it
follows from \eqref{418S} and relations \eqref{4.20S}  that
  \begin{eqnarray*}
\sum_{j=1}^\infty | \langle f, \psi_j \rangle |^2 =
\sum_{j=1}^\infty | \langle f, (-\Delta+I){\widetilde u}_j
\rangle|^2
\le C\sum_{j=1}^\infty \|f \|^2_{L^{2}(B_j)} \|\widetilde u_j\|^2_{W^{2,2}(B_j)} \nonumber  \\
= C\sum_{j=1}^\infty \|f \|^2_{L^{2}(B_j)}\|\widetilde
u_0\|^2_{W^{2,2}(B_0)} =  C\|\widetilde u_0\|^2_{W^{2,2}(B_0)}
\sum_{j=1}^\infty \|f \|^2_{L^{2}(B_j)} \le C\|\widetilde
u_0\|^2_{W^{2,2}(B_0)}\|f\|^2_{L^{2}(\R^3)}.  
  \end{eqnarray*}
Thus, the second  inequality of  \eqref{4.22S} is also proved, hence
$\{\varphi_j\}$  forms  a Riesz basis.

\emph{Necessity.}  Suppose  that $d_\ast(X)=0$.
  By \cite[Theorem 6.2.1]{GK65}, a sequence
$\Psi=\{\psi_j\}^{\infty}_1$ of vectors is a Riesz basis of a Hilbert space $\mathfrak
H$ if and only if it is complete in $\mathfrak H$ and its Gram matrix
$Gr_{\Psi}:=(\langle\psi_j,\psi_k\rangle)_{j,k\in{\N}}$ defines a
bounded  operator on $l^2(\mathbb N)$ with bounded
inverse.

By \eqref{3.2AS},
 $E = \{\varphi_j\}_1^\infty$ has the Gram matrix
$Gr_{E} = (\langle\varphi_j, \varphi_k\rangle)_{j,k\in{\N}} =
(\pi e^{-|x_j - x_k|})_{j,k\in{\N}} = 2\pi \kT_1.$ Therefore, by Lemma
\ref{lemma4.7}, if  $Gr_{E}$ defines a bounded operator,  this operator is not boundedly invertible.
Hence  $E=\{\varphi_j\}_1^\infty$ is not a Riesz basis by the preceding theorem.
    \end{proof}
   \begin{remark}
Note that the proof of uniform minimality of the system $E$ is
much simpler. Combining \eqref{418S} with \eqref{4.20S} we
obtain
   \begin{equation}
| \langle f,\psi_j \rangle |
\le\|f\|_{L^2}\cdot\|(I-\Delta){\widetilde
u}_j\|_{L^2}\le\|f\|_{L^2}\|{\widetilde u}_j\|_{W^{2,2}(\R^3)} =
\|f\|_{L^2}\|{\widetilde u}_0\|_{W^{2,2}(\R^3)}, \qquad j\in{\N}.
   \end{equation}
Since  $f\in{\mathfrak N}_{-1}$ is arbitrary, one  has
$\sup_{j\in \N}\|\psi_j\|_{L^2(\R^3)}\le\|{\widetilde
u_0}||_{W^{2,2}(\R^3)},$ so  $\{\psi_j\}_{j\in \N}$
is uniformly minimal.
       \end{remark}

Next we set
    \begin{equation}
\varphi_{j,z}(x) := \frac{e^{i\sqrt{z}|x-x_j|}}{|x-x_j|}
\qquad\text{and}\qquad  e_{j,z}(x) :=  e^{i\sqrt{z}|x-x_j|},\quad j\in \N.
    \end{equation}
Clearly, $\varphi_{j,-1} = \varphi_{j},\ j\in \N.$
    \begin{corollary}\label{cor4.9}
Suppose that  $d_*(X) >0$. Then for any $z\in \bC\setminus
[0,+\infty),$  the sequence $E_z :=
\{\frac{1}{\sqrt{2\pi}}\varphi_{j,z}\}_{j=1}^\infty$ forms a Riesz
basis  in the deficiency subspace $\mathfrak N_z$ of the operator $H$.
Moreover, for $z=-a^2<0$ ($a>0$)  the system ${\sqrt{a}}E_{-a^2} =
\{\frac{\sqrt{a}}{\sqrt{2\pi}}\varphi_{j,-a^2}\}_{j=1}^\infty$ is
normed.
    \end{corollary}
    \begin{proof}
 It is easily seen that
    \begin{equation}
\int_{{\mathbb R}^3}\frac{e^{-|x-y|}}{|x-y|}\cdot
\frac{e^{i\sqrt{z}|y-x_j|}}{|y-x_j|} dy = \int_{{\mathbb
R}^3}\frac{e^{i\sqrt{z}|x-y|}}{|x-y|}\cdot
\frac{e^{-|y-x_j|}}{|y-x_j|} dy, \qquad j\in{\mathbb N}.
    \end{equation}
Using \eqref{3.2S}  we can rewrite this equality as
    \begin{equation}
(I - \Delta)^{-1}\varphi_{j,z} = (-\Delta - z)^{-1}\varphi_j,
\qquad j\in{\N}, \quad z\in \C\setminus \overline\R_+.
  \end{equation}
Therefore,  we have
  \begin{align}\label{varphiuz}
\varphi_{j,z}=U_{z}\varphi_j,\qquad \text{\rm where}\qquad
U_z:=(I-\Delta)(-\Delta - z)^{-1} = I-(1+z)(\Delta + z)^{-1}.
     \end{align}
Obviously, $U_z$ is a continuous bijection of $\mathfrak N_{-1}$ onto
$\mathfrak N_z$. Therefore, since $E = E_{-1} =\{\varphi_j\}_{j\in
\N}$ is Riesz basis of $\mathfrak N_{-1}$ by Theorem
\ref{rieszbasisN1}, $E_z=\{\varphi_{j,z}\}_{j=1}^\infty$ is a
Riesz basis of $\mathfrak N_z$.

  To prove the second statement we note that for  any $a>0$
the function $e^{-a|\cdot|}(\in W^{2,2}(\R^3)$) is a (generalized)
solution of the equation $(a^2I - \Delta)e^{-a|x|} = 2a
\frac{\exp({-a|x|})}{|x|}$.  Taking this equality into account we
obtain from  \eqref{3.2S} with $z=-a^2$ and $f= f_y(x) :=
\frac{e^{-a|x-y|}}{|x-y|}$ that
      \begin{equation}\label{3.26a}
\frac{e^{-a|x-y|}}{2a} = \frac{1}{4\pi}\int_{{\mathbb
R}^3}\frac{e^{-a|x-t|}}{|x-t|}\cdot \frac{e^{-a|t-y|}}{|t-y|}~dt,
\qquad a>0.
      \end{equation}
Setting here $x=y =x_j$   we get $\|\varphi_{j,-a^2}\|^2 =
2\pi/a$, i.e., the system ${\sqrt{a}}E_{-a^2}$ is normed.
    \end{proof}
Now we are ready to prove  Theorem \ref{propositionstronglydef}.
\smallskip

{\it Proof of Theorem \ref{propositionstronglydef}.}

(i): Suppose that $s\in (0,+\infty)$ and set
 \begin{align*}
 g_s(x) := s^{-1}e^{-s|x|},\qquad  {\widetilde{\varphi}}_{j,s}(x):=\frac{1}{\sqrt{2\pi}}\varphi_{j,-s^2}(x)=\frac{1}{\sqrt{2\pi}}~
\frac{e^{-s|x- x_j|}}{|x-
x_j|}~,\qquad j\in \N.
\end{align*}
Equation  \eqref{3.2AS} shows that  $Gr_X(g_s) =
\big(g_s(x_k{-}x_j)\big)_{k,j\in \bN}$ is the Gram matrix of
the sequence $E_{-s^2} :=\{
\widetilde{\varphi}_{j,s}\}_{j=1}^\infty$. Since $d_\ast(X)>0$ by
assumption,  $E_{-s^2}$ forms a Riesz basis by Corollary
\ref{cor4.9}.  Therefore it follows from \cite[Theorem
6.2.1]{GK65} that for any $s>0$ the Gram matrix
$\bigl(\langle{\widetilde\varphi}_{j,s},{\widetilde
\varphi}_{k,s}\rangle_{L^2({\R}^3)}\bigr)_{j,k\in{\N}} =
Gr_X(g_s)$ defines a bounded operator on $l^2(\bN)$ with
bounded inverse. Hence for any $s>0$ there exist numbers $C(s)>0$
and $c(s)>0$ such that
      \begin{equation}\label{3.51}
 C(s)\sum^m_{j=1}|\xi_j|^2 \ge  \sum^m_{j,k=1}\langle{\widetilde{\varphi}}_{j,s},{\widetilde{\varphi}}_{k,s}\rangle_{L^2({\R}^3)}\xi_j\overline{\xi}_k
 =  \sum^m_{j,k=1}s^{-1} e^{-s|x_j-x_k|}\xi_j\overline{\xi}_k\  \ge \
 c(s)\sum^m_{j=1}|\xi_j|^2
          \end{equation}
for all  $(\xi_1,\cdots,\xi_m)\in \bC^m$ and $m\in \bN$. Clearly,
the function $c(s)$ on $(0,+\infty)$ can be chosen  to be
measurable. Since $c(s)>0$ on $\R_+$ and $\tau(\R_+)>0$, we have
$c:=  \int_{(0,+\infty)} s c(s)d\tau(s)>0.$ Combining \eqref{3.50}
with  \eqref{3.51} we arrive at the inequality
   \begin{align}\label{firstesti}
\sum^m_{j,k=1}f(|x_j &- x_k|)\xi_j\overline{\xi}_k =
\int^{\infty}_0 \left(\sum^m_{j,k=1}
e^{-s|x_j-x_k|}\xi_j\overline{\xi}_k\right) d\tau(s) \nonumber\\ &
\ge  \int^{\infty}_0 s\left(c(s)\sum^m_{j=1}|\xi_j|^2
\right)d\tau(s) =  c~\sum^m_{j=1}|\xi_j|^2.
\end{align}
This means that the function $f(|\cdot|)$ is strongly $X$-positive
definite.

(ii):  By  \eqref{varphiuz}, $U_{-s^2}
=(I-\Delta)(-\Delta+s^2)^{-1}$, hence  $\|U_{-s^2}\|={\rm
max}~(1,s^{-2})$. 
Moreover, by \eqref{varphiuz}, $\widetilde{\varphi}_{j,s}=
U_{-s^2}\widetilde{\varphi}_{j,1}$.  Using  the preceding facts
we derive
    \begin{align}\label{est2}
&\sum^m_{j,k=1}f(|x_j - x_k|)\xi_j\overline{\xi}_k  =\int^{\infty}_0 \left(\sum^m_{j,k=1}
e^{-s|x_j-x_k|}\xi_j\overline{\xi}_k\right) d\tau(s)\\&=
\sum^m_{j,k=1}\int_0^{+\infty}s~\langle{\widetilde{\varphi}}_{j,s},{\widetilde{\varphi}}_{k,s}\rangle
\xi_j\overline{\xi}_k~ d\tau(s)\nonumber =
\int_0^{+\infty}s~ \big\|\sum_{j=1}^m
\xi_j\widetilde{\varphi}_{j,s}\big\|^2 ~d\tau(s)\\&=
\int_0^{+\infty}s~ \big\|U_{-s^2}\big(\sum_{j=1}^m
\xi_j\widetilde{\varphi}_{j,1}\big)\big\|^2 ~d\tau(s)\nonumber
\leq \int_0^{+\infty}s~ \|U_{-s^2}\|^2\big\|\sum_{j=1}^m
\xi_j\widetilde{\varphi}_{j,1}\big\|^2 ~d\tau(s)\\&=
2\int_0^{+\infty}s~
\|U_{-s^2}\|^2\sum^m_{j,k=1}\langle{\widetilde{\varphi}}_{j,1},{\widetilde{\varphi}}_{k,1}\rangle
\xi_j\overline{\xi}_k~ d\tau(s)\nonumber \\&\leq \int_0^{+\infty}
s(1 + s^{-4})C(1)\bigg(\sum_{j=1}^m |\xi_j|^2\bigg) ~d\tau(s)=
C~\sum_{j=1}^m |\xi_j|^2~,
    \end{align}
where $C:= C(1)\int_0^{+\infty} (s{+}s^{-3})~d\tau(s)< \infty$ by
assumption (\ref{fourthmoment}).

It follows from (\ref{firstesti}) and (\ref{est2}) that the matrix
$Gr_X(f)$ defines a bounded operator with bounded inverse.

(iii) Suppose that $d_*(X)=0$.  Assume to the contrary that the Gram matrix $Gr_X(f)$
defines a bounded operator, say $ T$, with bounded inverse on $l^2(\mathbb N)$.

Fix $\varepsilon \in (0, \tau([0, \infty)) )$. Since the measure $\tau$ is finite, there exists $s_0>0$ such that
    \begin{equation}\label{4.31}
\int_{[s_0,\infty)}d\tau(s)<\varepsilon < \tau([0, \infty)).
    \end{equation}
By the assumption $d_*(X)=0$ we can find points $x_k,
x_j\in X$, $k,l\in \bN$, such that
$r_{jk} = |x_j - x_k|\le s^{-1}_0 \ln\bigl(1 +
\varepsilon ~(\tau([0,s_0]))^{-1}\bigr).$
Fix a number $l\in {\N}$. First suppose $r_{jl}\le r_{kl}$.  Then
    \begin{equation}\label{4.32}
0\le ( 1 - e^{-s(r_{kl}-r_{jl})})^2 \le 1 - e^{-sr_{kj}}
\le \frac{\varepsilon~(\tau([0,s_0]))^{-1}} {1 + \varepsilon~(\tau([0,s_0]))^{-1}}  \le \varepsilon~(\tau([0,s_0]))^{-1},\quad
s\in[0,s_0].
    \end{equation}
Using (\ref{4.31}) and (\ref{4.32}) we derive
    \begin{eqnarray}\label{4.33}
\left( \int^{\infty}_0 (e^{-sr_{jl}}-e^{-sr_{kl}})d\tau(s) \right)^2
= \left(\int^{\infty}_0(1 - e^{-s(r_{kl}-r_{jl})}
)e^{-sr_{jl}}d\tau(s)\right)^2 \nonumber\\
\le
\left(\int^{\infty}_{s_0}\bigl(1-e^{-s(r_{kl}-r_{jl})}\bigr)^2d\tau(s)
+ \int^{s_0}_0
\bigl(1-e^{-s(r_{kl}-r_{jl})}\bigr)^2d\tau(s)\right)\bigg( \int_0^\infty e^{-2sr_{jl}}d\tau(s)\bigg)\nonumber \\ \le
2\varepsilon \int_0^\infty e^{-2sr_{jl}}d\tau(s).
   \end{eqnarray}
If
$r_{jl}>r_{kl}$ then the same reasoning yields
    \begin{equation}\label{4.34}
\left(\int^{\infty}_0(e^{-sr_{jl}}-e^{-sr_{kl}})d\tau(s)\right)^2
\le 2\varepsilon\int_0^\infty e^{-2sr_{kl}}d\tau(s).
    \end{equation}
Summing over $l$ in \eqref{4.33} resp.  \eqref{4.34}  we obtain
    \begin{eqnarray}
\| T (e_j - e_k)\|^2_{l^2(\N)} = \sum_l ~\left|\langle T (e_j-e_k),e_l\rangle\right|^2 =\sum_l
\left(\int^{\infty}_0(e^{-sr_{jl}}-e^{-sr_{kl}})d\tau(s)\right)^2
\nonumber \\
\le2\varepsilon \sum_l
\left(\int^{\infty}_0e^{-2sr_{jl}}d\tau(s) +
\int^{\infty}_0e^{-2sr_{kl}}d\tau(s)\right)   =2 \varepsilon
\bigl(\| T e_j\|^2 + \| Te_k\|^2\bigr)
 \le 4\varepsilon \| T\|^2.
    \end{eqnarray}
and hence
\begin{equation}
4=\|e_j-e_k\|^2 \leq \|T^{-1}\|^2 \|T(e_j-e_k)\|^2\leq 4\varepsilon \|T^{-1}\|^2 \|T\|^2
\end{equation}
for $j\neq k$. Since $\varepsilon>0$ is arbitrary, this is a contraction.   $\hfill \Box$

\medskip

Now we return to the considerations related to  Theorem \ref{rieszbasisN1} and
recall the following
    \begin{definition}
A basis $\{f_j\}_1^\infty$ of a  Hilbert space $\mathfrak H$ is called a {\it Bari basis}
if there exists an orthonormal basis $\{g_j\}_1^\infty$ of
$\mathfrak H$ such that
    \begin{equation}
\sum_{j\in \N}\|f_j-g_j\|^2<\infty.
    \end{equation}
   \end{definition}
It is known that each Bari  basis is  a Riesz basis. The converse
statement is not true.

   \begin{proposition}
Assume that  $X$ has no finite  accumulation points.  Then the
sequence $E := \big\{\frac{1}{\sqrt{2\pi}}\varphi_j\big\}_{j=1}^\infty:=
\big\{\frac{1}{\sqrt{2\pi}}\frac{e^{-|x- x_j|}}{|x-
x_j|}\big\}_{j=1}^\infty$  forms a Bari basis of ${\mathfrak N}_{-1}$ if and
only if
    \begin{equation}\label{4.28S}
\sum_{j,k\in \N, j\not =k}e^{-2|x_j-x_k|}<\infty.
    \end{equation}
Moreover, this condition is equivalent to
    \begin{equation}
D_\infty:= {\rm lim}_{n\to\infty}D(\varphi_1,\ldots,\varphi_n)>0,
    \end{equation}
where $D(\varphi_1,\ldots,\varphi_n)$   denotes the determinant
of the matrix $\bigl(\langle \varphi_j,\varphi_k\rangle \bigr)^n_{j,k=1}$.
     \end{proposition}
   \begin{proof}
By  \eqref{3.2AS}, we have $\langle\varphi_j,\varphi_k \rangle =
2\pi\exp(-|x_j-x_k|)$ for $ j,k\in{\mathbb N}$. By  Lemma  \ref{lem4.3},
the system $E$ is $\omega$-linearly independent.  Therefore, by
\cite[Theorem 6.3.3]{GK65}, $E$  is a Bari basis if and only if
$$
\bigl(\langle\varphi_j,\varphi_k \rangle - 2\pi
\delta_{jk}\bigr)^{\infty}_{j,k=1} = 2\pi \bigl( \exp(-|x_j-x_k|)
- \delta_{jk}\bigr)^{\infty}_{j,k=1} \in\mathfrak S_2(l^2),
$$
i.e.  condition \eqref{4.28S}  is satisfied.  The second statement
follows from \cite[Theorem 6.3.1.]{GK65}
      \end{proof}

\section{Operator-Theoretic Preliminaries}\label{preliminaries}
\subsection{Boundary  triplets and  self-adjoint
relations}

Here  we briefly recall basic notions and facts on boundary
triplets (see \cite{DerMal91, GG,Sch2012} for
 details). In what follows  $A$   denotes a
densely defined closed symmetric operator  on a Hilbert space
$\gotH$, $\mathfrak N_z:= \mathfrak N_z(A) = \ker(A^*-z)$, $z\in
\C_{\pm},$ is the defect subspace. We also assume that $A$ has
equal deficiency indices $n_+(A) := \dim(\mathfrak N_i) = \dim(\mathfrak
N_{-i}) =: n_-(A).$

    \begin{definition}\cite{GG}\label{bound}%
\,\,A  {\emph boundary triplet} for the adjoint operator $A^\ast$
is a triplet $\Pi = \{\kH,\gG_0,\gG_1\}$  of  an auxiliary Hilbert
space $\kH$ and of linear mappings
$\Gamma_0,\Gamma_1:\  \dom(A^*)\rightarrow\kH$  such that\\
 $(i)$ the following abstract   Green identity holds:
\begin{equation}\label{GI}
(A^*f,g)_\gotH - (f,A^*g)_\gotH = (\gG_1f,\gG_0g)_\kH -
(\gG_0f,\gG_1g)_\kH,\qquad f,g\in\dom(A^*).
\end{equation}
$(ii)$ the mapping $(\Gamma_0,\Gamma_1): \dom(A^*)
\rightarrow \kH \oplus\kH$ is surjective.
   \end{definition}
   With   a  boundary   triplet   $\Pi$   one  associates two    self-adjoint extensions of $A$  defined  by
  \begin{equation}\label{2.2}
 A_0:=A^*\!\upharpoonright\ker(\gG_0)\quad  \text{and}\quad
A_1:=A^*\!\upharpoonright\ker(\gG_1).
     \end{equation}
    \begin{definition}
$(i)$ A closed extension $\widetilde{A}$ of $A$ is called
\emph{proper} if $A\subset \widetilde{A} \subset A^*$.  The set of
all  proper  extensions of  $A$
is  denoted  by $\Ext_A$.

$(ii)$ Two proper extensions $\widetilde{A}_1$ and
$\widetilde{A}_2$ of $A$ are called \emph{disjoint}
 if $\dom(\widetilde{A}_1)\cap\dom(\widetilde{A}_2)=\dom(A)$ and \emph{transversal} if, in
 addition,
 $\dom(\widetilde{A}_1)\dotplus\dom(\widetilde{A}_2)=\dom(A^*)$.
           \end{definition}
       \begin{remark}
$(i)$  If the symmetric operator $A$ has equal deficiency indices
$n_+(A)=n_-(A)$, then a boundary triplet $\Pi =
\{\kH,\gG_0,\gG_1\}$ for $A^*$ always exists and we have  $\dim
\kH = n_{\pm}(A)$ \cite{GG}.

 $(ii)$ For each self-adjoint extension
 $\widetilde{A}$ of $A$  there  exists  a  boundary   triplet
 $\Pi = \{\mathcal{H},\Gamma_0,\Gamma_1\}$  such  that
 $\widetilde{A}=A^*\upharpoonright\ker(\Gamma_0) = A_0.$

 $(iii)$ If $\Pi =  \{\mathcal{H},\Gamma_0,\Gamma_1\}$ is a boundary triplet for  $A^*$  and $B=B^*\in
 \kB(\kH)$,
 then the     triplet  $\Pi_B = \{\mathcal{H},\Gamma_0^B,\Gamma_1^B\}$   with   $\Gamma_1^B:=\Gamma_0$  and
 $\Gamma_0^B:=B\Gamma_0-\Gamma_1$ is also  a  boundary  triplet   for  $A^*$.
 \end{remark}

 Boundary triplets for $A^*$  allow  one  to  parameterize
 the  set  $\Ext_A$  in terms of closed linear relations.
For this we recall the following definitions.
\begin{definition}
$(i)$  A    linear relation $\Theta$ in  $\mathcal{H}$ is
a linear subspace of $\mathcal{H}\oplus\mathcal{H}$. It is called closed if the corresponding subspace is closed in $\mathcal{H}\oplus\mathcal{H}$.

$(ii)$  A linear  relation  $\Theta$  is called symmetric
 if  $(g_1,f_2)-(f_1,g_2)=0$  for all  $\{f_1,g_1\},  \{f_2,g_2\}\in  \Theta$.

 $(iii)$ The adjoint relation
$\Theta^*$ of a  linear relation $\Theta$ in $\mathcal{H}$ is
defined by
\begin{equation*}
\Theta^*= \big\{ \{k,k^\prime\}: (h^\prime,k)=(h,k^\prime)\,\, \text{for all}\, \,
\{h,h^\prime\}
\in\Theta\big\}.
\end{equation*}

 $(iv)$  A closed linear   relation $\Theta$  is  called  self-adjoint
 if  $\Theta=\Theta^*$.

{$(v)$ The inverse of a relation $\Theta$ is the relation $\Theta^{-1}$ defined by $\Theta^{-1}=\big\{ \{h^\prime,h\}:\{h,h^\prime\}
\in\Theta\big\}. $}
    \end{definition}
\begin{definition}
{Let $\Theta$ be a closed relation in $\cH$. The resolvent set $\rho(\Theta)$ is the set of complex  numbers $\lambda$ such that the relation
$(\Theta-\lambda I)^{-1}:=\big\{ \{h^\prime-\lambda h,h\}: \{h,h^\prime\}\in \Theta \big\}$ is the graph of a bounded operator of $\kB(\cH)$. The complement set $\sigma(\Theta):=\bC\backslash \rho(\Theta)$ is called the spectrum of $\Theta$.}
\end{definition}

For a relation $\Theta$ in $\cH$ we define the domain $\dom(\Theta)$ and the multi-valued part $\mul(\Theta)$ by
\begin{align*}
\dom(\Theta)=\big\{ h\in \cH: \{h,h^\prime\}\in \Theta ~~{\rm for~ some}~~ h^\prime\in \cH\big\},~~~\mul(\Theta)=\big\{ h^\prime \in \cH: \{0,h^\prime\} \in \Theta\big\}.
\end{align*}
{Each closed relation $\Theta$  is the orthogonal sum of
$\Theta_\infty:=\big\{ \{0,f'\}\in\Theta\big\}$ and $\Theta_{\rm op}:=\Theta\ominus \Theta_\infty$. Then $\Theta_{\rm op}$ is the graph of a closed operator, called the operator part of $\Theta$ and denoted also by $\Theta_{\rm op}$,  and $\Theta_\infty$ is a ``pure'' relation, that is, $\mul(\Theta_\infty)=\mul(\Theta)$.}

{Suppose that $\Theta$ is a self-adjoint relation in $\cH$. Then
 $\mul(\Theta)$  is the orthogonal complement of
$\dom(\Theta)$ in $\cH$  and  $\Theta_{\rm op}$ is
 a self-adjoint operator in the Hilbert space
$\cH_{\rm op}:=\overline{\dom(\Theta)}$. That is, $\Theta$ is the orthogonal sum of an ''ordinary'' self-adjoint operator $\Theta_{\rm op}$ in  $\cH_{\rm op}$ and a ``pure'' relation $\Theta_\infty$
in  $\cH_\infty:=\mul(\Theta)$.
}

   \begin{proposition}[\cite{DerMal91,GG,Sch2012}]\label{propo}
Let  $\Pi = \{\mathcal{H},\Gamma_0,\Gamma_1\}$  be   a boundary
triplet for $A^*$. Then  the  mapping
   \begin{equation}\label{s-aext}
\Ext_A\ni\widetilde{A} := A_\Theta \rightarrow
\Theta:=\Gamma(\dom(\widetilde{A}))=\{\{\Gamma_0f,\Gamma_1f\}:\,\,f\in
\dom(\widetilde{A})\}
   \end{equation}
is a  bijection of the  set
$\Ext_A$ {of all proper extensions of $A$}  and  the set of all
closed   linear  relations
$\widetilde{\mathcal{C}}(\mathcal{H})$
in $\mathcal{H}$.  Moreover, the following equivalences hold:

$(i)$
   $(A_\Theta)^*=A_{\Theta^*}$
 for any linear  relation $\Theta$ in $\mathcal{H}.$

$(ii)$
  $A_\Theta$  is  symmetric
if  and only  if \,  $\Theta$  is symmetric.
Moreover, $n_{\pm}(A_\Theta) = n_{\pm}(\Theta).$  In particular, $A_\Theta$  is 
self-adjoint if  and only  if  $\Theta$  is self-adjoint.

$(iii)$ The closed extensions $A_\Theta$ and
$A_0$ are disjoint if and only if $\Theta = B$ is a closed operator. 
In this case
\begin{equation}\label{bijop}
A_\Theta=A_B=A^*\!\upharpoonright\dom(A_B) ,\quad \dom(A_B) =
\ker\bigl(\Gamma_1-B\Gamma_0\bigr).
   \end{equation}
 \end{proposition}

\subsection{Weyl   function, $\gamma$-field  and spectra of proper extensions}

The notion of the  Weyl function and the $\gamma$-field  of a boundary triplet was introduced
 in \cite{DerMal91}.
\begin{definition}[\cite{DerMal91,Sch2012}]\label{Weylfunc}
Let $\Pi =  \{\kH,\gG_0,\gG_1\}$ be a boundary triplet  for $A^*.$
The operator-valued functions $\gamma(\cdot) :\
\rho(A_0)\rightarrow  \kB(\kH,\gotH)$ and  $M(\cdot) :\
\rho(A_0)\rightarrow  \kB(\kH)$ defined by
  \begin{equation}\label{2.3A}
\gamma(z):=\bigl(\Gamma_0\!\upharpoonright\mathfrak N_z\bigr)^{-1}
\qquad\text{and}\qquad M(z):=\Gamma_1\gamma(z), \quad
z\in\rho(A_0),
      \end{equation}
are called the {\em $\gamma$-field} and the {\em Weyl function},
respectively, of $\Pi = \{\kH,\gG_0,\gG_1\}.$
      \end{definition}
Note that the $\gamma$-field $\gamma(\cdot)$ and  the Weyl function
$M(\cdot)$ are  holomorphic on $\rho(A_0)$.%

Recall that a symmetric operator $A$ in $\frak H$ is said to be {\it simple} if
there is no non-trivial subspace which reduces it to a
self-adjoint operator. In other words,  $A$ is simple if it does not admit an (orthogonal) decomposition $A=A'\oplus S$ where $A'$ is a symmetric operator and $S$ is a selfadjoint operator acting on a nontrivial Hilbert space.

It is easily seen (and well-known) that $A$
is simple if and only if  $\Span\{\frak N_z(A): z\in\C\setminus\Bbb R \}=\frak H$.

If $A$ is  simple, then  the Weyl function $M(\cdot)$ determines
the boundary triplet  $\Pi$ uniquely up to the unitary equivalence  (see \cite{DerMal91}).
In particular, $M(\cdot)$ contains the full information about the
spectral properties of $A_0$.
Moreover, the spectrum  of a  proper  (not
necessarily self-adjoint) extension $A_\Theta\in \Ext_A$ can be
described  by means of $M(\cdot)$ and  the   boundary relation $\Theta$.
  \begin{proposition}[{\cite{DerMal91},\cite{Sch2012}}]\label{prop_II.1.4_spectrum}
Let $A$ be a simple densely defined symmetric operator in $\gotH$,
$\Theta\in \widetilde{\mathcal{C}}(\mathcal{H})$, 
and $z\in \rho(A_0)$. Then:

(i)\  $z\in \rho(A_\Theta)$ if and only if
$0\in \rho(\Theta-M(z));$

(ii) $z\in\sigma_\tau(A_\Theta)$ if and only if  $0\in
\sigma_\tau(\Theta-M(z))$,  \ $\tau\in\{ p,\ c\}.$

(iii) $f\in \ker(A_\Theta - z)$ if and only if  $\Gamma_0f\in
\ker(\Theta - M(z))$ and\ \
$$\dim\ker(A_\Theta {-}z) =
\dim\ker(\Theta {-} M(z)).
$$
%
      \end{proposition}

For any boundary triplet $\Pi =\{\cH,\gG_0,\gG_1\}$ for $A^*$ and
any proper   extension ${A}_{\Theta}\in \Ext_A$ with non-empty
resolvent set the following Krein-type resolvent formula holds
(cf. \cite{DerMal91,Sch2012})
     \begin{equation}\label{2.8}
(A_\Theta-z)^{-1} = (A_0 - z)^{-1} + \gamma(z)\bigl(\Theta -
M(z)\bigr)^{-1}\gamma(\overline{z})^*,\qquad
z\in\rho(A_\Theta)\cap\rho(A_0).
   \end{equation}
It should be emphasized that formulas \eqref{2.2}
\eqref{s-aext}, and \eqref{2.3A} express all data occuring in
\eqref{2.8} in terms of the boundary triplet.
These expressions allow one to apply formula \eqref{2.8} to boundary value problems.

The following result is deduced from  \eqref{2.8}.
   \begin{proposition}[{\cite[Theorem 2]{DerMal91}}]\label{prop2.9}
Let $\Pi = \{\cH,\gG_0,\gG_1\}$  be a boundary triplet for $A^*$
and let $\Theta',\Theta \in \widetilde{\mathcal C}(\cH)$. Suppose
that $\rho(A_{\Theta'})\cap \rho(A_{\Theta})\not = \emptyset$ and
$\rho({\Theta'})\cap \rho({\Theta})\not = \emptyset$.

\item[\;\;\rm (i)]  For   $z\in \rho(A_{\Theta'})\cap
\rho(A_{\Theta})$,   $\zeta\in \rho({\Theta'})\cap
\rho({\Theta})$, and $ p \in [0,\infty]$ the following equivalence is valid:
  \begin{equation}\label{2.31}
(A_{\Theta'} - z)^{-1} - (A_{\Theta} - z)^{-1}\in{\mathfrak S}_p(\mathfrak
H)  \Longleftrightarrow\  \bigl(\Theta' - \zeta\bigr)^{-1}-
\bigl(\Theta - \zeta \bigr)^{-1}\in{\mathfrak S}_p(\cH).
    \end{equation}
In particular,  $(A_{\Theta}- z)^{-1} - (A_0 -
z)^{-1}\in{\mathfrak S}_p(\mathfrak H)$ 
if and only if $\bigl(\Theta - \zeta\bigr)^{-1} \in{\mathfrak
S}_p(\cH)$ for $\zeta \in \rho(\Theta)$.

\item[\;\;\rm (ii)] If $\dom(\Theta') = \dom(\Theta)$, then the
following implication holds
   \begin{equation}\label{2.32}
\overline{\Theta' - \Theta} \in{\mathfrak S}_p(\cH)
\Longrightarrow (A_{\Theta'}-z)^{-1} -
(A_{\Theta}-z)^{-1}\in{\mathfrak S}_p(\mathfrak H),\quad z\in  \rho(A_{\Theta'})\cap \rho(A_{\Theta})
   \end{equation}
 In particular, if
$\Theta',\Theta\in \kB(\cH)$, then \eqref{2.31} is equivalent to
$\Theta' - \Theta\in \mathfrak S_p(\cH)$.
\end{proposition}

\subsection{Extensions of  nonnegative
symmetric operators}\label{sss_II.1.5_nno}

In this subsection we assume that  the symmetric operator $A$ on $\mathfrak{H}$
is \emph{nonnegative}.
Then the set $\Ext_A(0,\infty)$ of all nonnegative
self-adjoint extensions of $A$ on $\mathfrak{H}$  is  not empty.
Moreover, there exists
a maximal nonnegative  extension $A_{\rm F}$,  called the \emph{Friedrichs'}
extension, and  a  minimal nonnegative
extension $A_{\rm K}$, called the \emph{Krein's} extension, in the set $\Ext_A(0,\infty)$ and
$$
(A_F+x)^{-1} \le (\widetilde A + x)^{-1} \le (A_K + x)^{-1},
\qquad x\in (0,\infty), \qquad  \widetilde A\in  \Ext_A(0,\infty).
$$
(For details we refer the reader to \cite[Chapter 8]{AkhGlz78}, \cite[Section 6.2.3]{Kato66} or \cite[Sections 13.3, 14,8]{Sch2012}.)
   \begin{proposition}[\cite{DerMal91}]\label{prop_II.1.5_01}
Let  $\Pi = \{\cH,\Gamma_0,\Gamma_1\}$ be a boundary triplet for
$A^*$ such that $A_0\geq 0$ and let $M(\cdot)$ be the
corresponding Weyl function.

$(i)$ There exists a lower semibounded self-adjoint linear relation
$M(0)$ in $\cH$ which is  the strong resolvent limit of $M(x)$ as
$x\uparrow 0$. Moreover,  $M(0)$ is associated with the closed
quadratic form
    \begin{eqnarray*}
\mathfrak t_0[h]:= \lim_{x\uparrow 0}\bigl(M(x)h,h\bigr), \quad
\dom(\mathfrak t_0)=\{h: \lim_{x\uparrow 0}\bigl(M(x)h,
h\bigr)<\infty\}=\dom\bigl((M(0) - M(-a))^{1/2}\bigr).
    \end{eqnarray*}

(ii) The Krein extension $A_K$ is given by
    \begin{equation}\label{2.10}
A_K = A^*\!\upharpoonright \dom(A_K),\qquad
\dom(A_K)=\{f\in \dom (A^\ast):\{\Gamma_0f, \Gamma_1 f\}\in M(0)\}.
    \end{equation}
The extensions $A_K$ and $A_0$ are disjoint if and only if
$M(0)\in {\mathcal C}(\cH)$. In this case  \
$\dom(A_K)=\ker\bigl(\Gamma_1-M(0)\Gamma_0\bigr).$

$(iii)$  $A_0=A_{\rm F}$\ if and only if \  $\lim_{x\downarrow-\infty}(M(x)f,f)=-\infty$ for
$f\in\cH\setminus\{0\}$.

$(iv)$  $A_0=A_K$ if and only if \  \ $\lim_{x\uparrow0}(M(x)f,f)=+\infty$ for  $f\in\cH\setminus\{0\}$.

       \end{proposition}
If $A_\Theta$ is lower semibounded, then  $\Theta$ is lower semibounded  too.  The converse is
not true in general. In order to state the corresponding result we
introduce the following definition.

We shall say that $M(\cdot)$ \emph{tends uniformly  to $-\infty$} as
$x\to-\infty$ if for any $a>0$ there exists $x_a<0$ such that
$M(x_a) < -a\cdot I_{\cH}$. In this case we write
$M(x)\rightrightarrows-\infty$ as $x\to-\infty$.
%
        \begin{proposition}[\cite{DerMal91}]\label{prop_II.1.5_02}
Suppose that $A$ is a non-negative symmetric operator on $\mathfrak{H}$ and $\Pi=\{\cH,\Gamma_0,\Gamma_1\}$ is a boundary triplet
for $A^*$ such that $A_0=A_{\rm F}$. Let $M$ be
the corresponding Weyl function. Then the  two  assertions
\item $(i)$ \ a linear relation $\Theta\in \widetilde{\mathcal
C}_{\rm self}(\cH)$ is semibounded below, 
\item $(ii)$ \ a self-adjoint extension $A_\Theta$ is semibounded
below,
\item are equivalent if and only if
$M(x)\rightrightarrows-\infty$ for $x\to-\infty$.
    \end{proposition}
Recall that the order relation  for lower semibounded self-adjoint operators $T_1,T_2$  is defined by
   \begin{equation}\label{4.9A}
T_1 \ge T_2 \quad \text{if} \quad \dom(\mathfrak t_{T_1})
\subset \dom(\mathfrak t_{T_2})\quad \text{and}\quad  \mathfrak
t_{T_1}[u]\ge\mathfrak t_{T_2}[u],\quad  u\in\dom(\mathfrak t_{T_1}),
      \end{equation}
where $\mathfrak t_{T_j}$ is the quadratic form associated with $T_j$.

If $T$ is a self-adjoint operator with spectral measure $E_T$,   put $\kappa_-(T)
:= \dim\ran\bigl(E_T(-\infty,0)\bigr)$. For a self-adjoint relation
$\Theta$ we set
$\kappa_-(\Theta) := \kappa_-(\Theta_{op}),$ where $\Theta_{op}$ is
the operator part of $\Theta$. For a quadratic form $\mathfrak t$ we
denote by $\kappa_-(\mathfrak t)$ the number of negative squares of
$\mathfrak t$ (cf. \cite{Mal92}).
   \begin{proposition}[\cite{DerMal91}]\label{prkf}
Suppose $A$ is a densely defined nonnegative symmetric operator
on $\gotH$ and    $\Pi=\{\kH,\gG_0,\gG_1\}$ is a  boundary
 triplet for  $A^*$ such that $A_0=A_F$.
Let    $M$ be the
Weyl function of this boundary triplet and let $\Theta$ be a self-adjoint relation on $\kH$.
Then:

$(i)$  The self-adjoint  extension $A_\Theta $   
is nonnegative if and only if  $\Theta \ge M(0)$.

$(ii)$   If $A_\Theta$ is lower semibounded and  $\dom(\mathfrak
t_{\Theta})\subset \dom(\mathfrak t_{M(0)})$, then
$\kappa_{-}(A_\Theta) = \kappa_{-}(\mathfrak t_\Theta - \mathfrak
t_{M(0)})$. If,  in  addition,  $M(0)\in\kB(\cH)$,   then   
$\kappa_{-}(A_\Theta)=\kappa_{-}(\Theta-M(0))$.
   \end{proposition}

\subsection{Absolutely continuous spectrum and the Weyl function}\label{prelim4}

 In what follows we will denote
  \begin{equation*}
M_h(z):=(M(z)h,h),\quad z\in\mathbb{C}_+,\quad \text{and} \quad
M_h(x + i0) := \lim_{y\downarrow0}M_h(x+iy), \quad
h\in\mathcal{H}.
  \end{equation*}
Since  $\imm(M_h(z))>0,\ z\in \mathbb{C}_+,$ the limit $M_h(x+i0)$
exists and is finite for  a.e. $x\in\mathbb{R}.$ We put
   \begin{equation*}
 \Omega_{ac}(M_h):=\{x\in\mathbb{R}:\,0<\imm M_h(x)<+\infty\}. \quad
\end{equation*}
We  also set $d_M(x):=\rank(\imm(M(x+i0)))\le \infty$ provided
that the weak limit  $M(x + i0) := w-\lim_{y\downarrow0}M(x+iy)$
exists.
     \begin{proposition}[\cite{BraMal02}]\label{spectrum}
Let $A$  be a simple densely defined closed  symmetric operator on
a  separable Hilbert space  $\gotH$  and
let $\Pi=\{\kH,\gG_0,\gG_1\}$ be a boundary triplet  for $A^*$ with Weyl function
$M$. Assume  that
$\{h_k\}_{k=1}^N,\quad 1\leq N\leq\infty,$  is a total set in
$\mathcal{H}$. Recall that $A_0$ is the self-adjoint operator defined by
$A_0=A^*\upharpoonright\ker(\Gamma_0)$.

$(i)$ $A_0$ has  no point spectrum in the
interval $(a,b)$ if  and only if
$\lim\limits_{y\downarrow 0}yM_{h_k}(x+iy)=0$ for all $ x\in(a,b)$ and $k\in \{1,2,..,N\}.$

$(ii)$  $A_0$ has  no singular  continuous spectrum in the   interval  $(a,b)$ if
the set $(a,b)\setminus\Omega_{ac}(M_{h_k})$ is
 countable for each $k\in\{1,2,..,N\}$.
  \end{proposition}

To state the next proposition we need the concept of the
$ac$-closure $\cl_{ac}(\delta)$ of a Borel subset
$\delta\subset \R$ introduced independently in \cite{BraMal02} and \cite{Ges08}.
We refer to \cite{Ges08, MalNei09} for the definition of this
notion as well as for its basic properties.
       \begin{proposition}[\cite{MalNei11, MalNei09}]\label{Prop-ac}
Retain the assumptions of   Proposition  \ref{spectrum}. Let
$B$ be  a self-adjoint operator on $\mathcal{H}$, \
$A_B=A^*\upharpoonright\ker(\Gamma_1-B\Gamma_0),$ and $M_B(z) :=
(B-M(z))^{-1}.$

 $(i)$
 If the
limit $M(x+i0) := w-\lim\limits_{y\downarrow0}M(x+iy)$  exists
a.e. on $\mathbb{R}$, then
$\sigma_{ac}(A_0)=\cl_{ac}(\supp(d_M(x)))$.

$(ii)$ For  any  Borel subset   $\mathcal{D}\subset \mathbb{R}$
the  $ac$-parts $A_0E^{ac}_{A_0}(\mathcal{D})$ and
$A_BE^{ac}_{A_B}(\mathcal{D})$ of the operators
$A_0E_{A_0}(\mathcal{D})$  and $A_BE_{A_B}(\mathcal{D})$ are
unitarily equivalent if and only if\,\, $d_M(x)=d_{M_B}(x)$ a.e.
on $\mathcal{D}.$
\end{proposition}

\section{Three-dimensional Schr\"{o}dinger  operator with  point  interactions}\label{d=3}

First we collect some notation and assumptions that will be kept
in this section. Throughout the section  we fix a sequence
$X=\{x_k\}_1^\infty$  of points  $x_k\in\bR^3$ satisfying
   \begin{align*}
d_\ast(X) = {\rm inf}_{k,j\in \N, k\neq j}~|x_k-x_j|>0,
  \end{align*}
denote  by $H$  the restriction of $-\Delta$  given by
 (\ref{min}), and set
    \begin{equation}\label{5.1}
\varphi_{j,z}(x) = \frac{e^{i\sqrt{z}|x-x_j|}}{|x-x_j|} \quad
\text{and}\quad e_{j,z}(x) =  e^{i\sqrt{z}|x-x_j|}, \quad \ z\in
\bC \backslash [0,+\infty),\quad  j\in \bN.
    \end{equation}
Clearly, $\varphi_j=\varphi_{j,-1}$ and $e_j=e_{j,-1}$.
Recall from Lemma \ref{lemma4.7} that $T_1$ is the bounded operator on $l^2(\bN)$ defined by the matrix
$\kT_1:= (2^{-1}e^{-|x_j-x_k|})_{j,k\in\bN}$.

 \subsection{Boundary  triplets  and Weyl  functions}\label{sec5.1}

The  following lemma is a special case of Example 14.3 in \cite{Sch2012}.
   \begin{lemma}\label{technicallemma}
Let $A$ be a densely defined closed symmetric operator on $\mathfrak
H$. Suppose that $\widetilde A$ is a self-adjoint extension of $A$
on $\mathfrak H$ and  $ - 1\in \rho(\widetilde  A)$. Then:
\begin{align*}
(i)\qquad\qquad\qquad\qquad \qquad\qquad \dom (A^\ast) = \dom A &{\dot+} \ker (A^\ast{+}I){\dot+}(\widetilde  A{+}I)^{-1}\mathfrak N_{-1},\qquad\qquad\qquad\qquad\quad\\
A^\ast(f_A + f_0+ (\widetilde  A{+}I)^{-1}f_1) & = Af_A - f_0+
\widetilde  A(\widetilde  A{+}I)^{-1}f_1 , 
   \end{align*}
where  $f_A\in \dom (A)$ and  $f_0, f_1\in  \mathfrak N_{-1}:=\ker (A^\ast{+}I).$ 

\item (ii) Define $\cH^\prime = \mathfrak N_{-1}$ and \ $\Gamma^\prime_j(f_A + f_0+ (\widetilde  A{+}I)^{-1}f_1))=f_j$ \ for $j=0,1$. Then
 $\Pi' = \{\cH^\prime,\Gamma^\prime_0,\Gamma^\prime_1\}$  forms a
boundary triplet for $A^*$.
  \end{lemma}
   \begin{proof}
   Assertion (i) is  well known in extension theory (see e.g.  \cite[formula (14.17)]{Sch2012}), so we prove only assertion (ii).
Let $f=f_A+f_0+(I+\widetilde A)^{-1}f_1$\  and   $g = g_A + g_0
+(I+\widetilde A)^{-1}g_1$, where $f_0, f_1, g_0, g_1 \in \mathfrak
N_{-1}$. Then
%
%
   \begin{eqnarray}\label{4.21S}
\langle A^* f,g\rangle -\langle f,A^* g\rangle = \langle\widetilde
A(I+\widetilde A)^{-1}f_1,g_0\rangle  -  \langle
f_0,  (I+\widetilde A)^{-1}g_1\rangle  \nonumber    \\
+ \  \langle\widetilde A(I + \widetilde A)^{-1}f_1, (I +\widetilde
A)^{-1}g_1\rangle  -  \langle f_0,\widetilde A(I + \widetilde
A)^{-1}g_1\rangle  \nonumber  \\
+ \  \langle (I+\widetilde A)^{-1}f_1,  g_0\rangle   - \langle (I
+ \widetilde A)^{-1}f_1,  \widetilde A(I +\widetilde
A)^{-1}g_1\rangle    \nonumber  \\
=  - \  \langle f_0,(I+\widetilde A)(I+\widetilde
A)^{-1}g_1\rangle  +  \langle (I+\widetilde A)(I+\widetilde
A)^{-1}f_1, g_0\rangle  \nonumber  \\
= - \  \langle f_0,g_1\rangle_{\cH^\prime}  +  \langle f_1,
g_0\rangle_{\cH^\prime}  =  \langle \Gamma^\prime_1 f,\Gamma^\prime_0 g\rangle_{\cH^\prime} -
\langle \Gamma^\prime_0 f,\Gamma^\prime_1 g\rangle_{\cH^\prime}.
   \end{eqnarray}
The surjectivity of the mapping $(\Gamma^\prime_0,\Gamma^\prime_1)$ is
obvious.
  \end{proof}

Next we apply Lemma \ref{technicallemma}  to the minimal
Schr\"odinger operator $A=H.$
    \begin{proposition}\label{pr1}
Suppose $H$ is the minimal Schr\"odinger operator  defined by
\eqref{min} and  $d_*(X)>0.$  Let $T_1$ be  the bounded operator
on $l^2(\bN)$ defined by the matrix $\kT_1:=
(2^{-1}e^{-|x_j-x_k|})_{j,k\in\bN}$.  Then

$(i)$     $H$  is  a closed symmetric operator with deficiency
indices  $(\infty, \infty)$.  
The  defect subspace $\mathfrak {N}_{-1}  = \ker(H^* +I)$  is given by
    \begin{equation}\label{4.18}
\mathfrak N_{-1} =\left\{ \sum\limits_{j=1}^\infty c_j\varphi_{j}: \{c_j\}_1^\infty\in l^2(\N) \right\}  .
         \end{equation}

 $(ii)$   $\dom(H^*)$  is the direct sum of vector spaces $\dom H$, $\mathfrak N_{-1}$ and $(-\Delta+I)^{-1} \mathfrak N_{-1}$,  that is,
   \begin{align}\label{eq3}
\dom(H^*)& =\{ f = f_H + f_0 +(-\Delta{+}I)^{-1}f_1; f_H \in \dom H,~ f_0,f_1\in \mathfrak N_{-1} \}\\
&= \left\{ f = f_H +
\sum\limits^\infty_{j=1}\bigl(\xi_{0j}\varphi_j+\xi_{1j} e_j
\bigr) : f_H\in\dom H,~ \xi_{0} := \{\xi_{0j}\},\  \xi_{1} =
\{\xi_{1j}\}\in l^2(\N)\right\}, \nonumber
   \end{align}
\begin{equation}\label{H_max}
H^*f  = -\Delta f_H - f_0 + (-\Delta)(- \Delta {+} I)^{-1}f_1
=-\Delta f_H + \sum\limits^\infty_{j=1}\bigl(-\xi_{0j}\varphi_j +
\xi_{1j}(\varphi_j - e_j/2)\bigr).  
  \end{equation}

 $(iii)$   The triplet $\widetilde{\Pi}=
  \{\cH, \widetilde{\Gamma}_0, \widetilde{\Gamma}_1\}$, where
    \begin{equation}\label{4.37S}
\cH = l^2(\N), \qquad \widetilde{\Gamma}_0f = \xi_0,  \qquad
\widetilde{\Gamma}_1 f = T_1\xi_1, \qquad f\in  \dom(H^*),
    \end{equation}
is a boundary triplet for $H^*$.
   \end{proposition}
     \begin{proof} (i):
 By the Sobolev embedding theorem,
$f\rightarrow f(x_j)$ is a continuous linear functional on $W^{2,2}(\mathbb{R}^3)$ (see \cite[Chapter 2.5]{MasN}). Therefore,
$\dom(H)=W^{2,2}(\mathbb{R}^3)\upharpoonright\bigcap\limits_{j=1}^\infty\ker(\delta_{x_j})$
is closed in the  graph norm of $-\Delta,$ so the operator $H$
is closed. Since $-\Delta$ is self-adjoint, $H$ is symmetric.

  Since $d_\ast(X)>0$ by assumption,
Theorem \ref{rieszbasisN1}  applies and  shows that
$\{\varphi_{j}\}_1^\infty$ is a Riesz basis of the Hilbert space
$\mathfrak N_{-1}$. In particular, $n_\pm(H)= \infty.$

(ii): All assertions of (ii) follow  from (i) and Lemma
\ref{technicallemma}(i), applied to the self-adjoint operator
$A{=}-\Delta$ on $L^2(\R^3)$. For the formula of $H^\ast
f$ we recall that $e_j/2 = (-\Delta+I)^{-1}\varphi_j$ and
therefore, $H^\ast e_j= -\Delta (-\Delta{+}I)^{-1}\varphi_j =
\varphi_j - e_j/2$.

(iii) From \eqref{3.2AS} it follows that $\langle
\varphi_j,\varphi_k\rangle = 2^{-1}e^{-|x_j-x_k|}$, i.e., the Gram
matrix of  $E=\{\varphi_j\}_{j\in \N}$ is $\kT_1$. By Lemma \ref{lemma4.7},
$\kT_1$ defines the bounded operator $T_1$ on $l^2(\bN)$ with
bounded inverse. Hence  $\widetilde{\Gamma}_0$ and
$\widetilde{\Gamma}_1$ are well--defined and the map
$(\widetilde{\Gamma}_0, \widetilde{\Gamma}_1): \dom(A^*)
\rightarrow \kH \oplus\kH$ is surjective.

Next we verify  the  Green formula. Let $f,g\in \dom(H^*).$ By
\eqref{eq3}, these vectors are of the form
$$
f = f_H + f_0 +(-\Delta{+}I)^{-1}f_1, \quad g = g_H + g_0 +
(-\Delta{+}I)^{-1}g_1
  $$
with  $f_H, g_H \in \dom H$ and  $f_0, f_1, g_0, g_1  \in \mathfrak
N_{-1}.$ By  \eqref{4.18},  $f_0,f_1,g_0,g_1$ 
can be written as%
$$
f_0=\sum_{j=1}^\infty ~\xi_{0j}\varphi_j,~f_1=\sum_{j=1}^\infty
~\xi_{1j}\varphi_j,\quad  g_0=\sum_{j=1}^\infty~
\eta_{0j}\varphi_j, \quad g_1=\sum_{j=1}^\infty~
\eta_{0j}\varphi_j,
$$
where $\{\xi_{0j}\}_{j\in \N}, \{\xi_{1j}\}_{j\in \N}, \{\eta_{0j}\}_{j\in \N}, \{\eta_{1j}\}_{j\in \N} \in
l^2(\bN)$. Using the Green identity for the boundary triplet $\Pi'
= (\kH^\prime, \Gamma_0^\prime, \Gamma_1^\prime)$ in Lemma
\ref{technicallemma}, applied to $A=H$ and
$\widetilde{A}=-\Delta$, we derive the identity
   \begin{align*}
\langle H^\ast &f,g\rangle -\langle f,H^\ast g\rangle =
\langle\Gamma_1^\prime f,\Gamma_0^\prime g\rangle-\langle
\Gamma_0^\prime f,\Gamma_1^\prime g\rangle
=  \langle
f_1,g_0\rangle_{\mathfrak N_{-1}}-\langle f_0,g_1\rangle_{\mathfrak
N_{-1}}    \\
& =\sum_{j,k=1}^\infty~ (\xi_{1j}\overline{\eta_{0k}} -
\xi_{0j}\overline{\eta_{1k}})\langle \varphi_j,\varphi_k\rangle~
 = \sum^{\infty}_{k=1}\big((T_1\xi_1)_k\overline{\eta_{0k}} - \xi_{0k}\overline{(T_1\eta_1)_k} ~\big)  \\
 &=
\langle T_1\xi_1,  \eta_0\rangle - \langle \xi_1, T_1
\eta_{0}\rangle = \langle \widetilde{\Gamma}_1 f,
\widetilde{\Gamma}_0 g\rangle_{\cH}- \langle \widetilde{\Gamma}_0
f, \widetilde{\Gamma}_1 g\rangle_{\cH},
\end{align*}
which completes the proof.
\end{proof}

However,  we prefer to work with another boundary triplet.  For
this purpose we define
  \begin{align}\label{defvarphimatrix1}
(T_0(\xi_j))_k = - \xi_k + \sum_{j\in \N, j\neq k}
~\xi_j~\frac{e^{-|x_k{-}x_j|}}{|x_k{-}x_j|},\qquad \{\xi_j\}_{j\in \N}\in
l^2(\bN).
   \end{align}
It  follows from the assumption $d_\ast(X)>0$ and the fact that
the matrix $ (2^{-1}e^{-|x_j-x_k|})_{j,k\in\bN}$ defines a bounded
operator $T_1$ on $l^2(\bN)$ by Lemma \ref{lemma4.7}, that $T_0$
is a bounded self-adjoint  operator on $l^2(\bN)$.

Next we slightly modify the boundary triplet $\widetilde{\Pi}=
 \{\cH, \widetilde{\Gamma}_0, \widetilde{\Gamma}_1\}$   and express
the trace  mappings $\widetilde{\Gamma}_j$ in terms of the
"boundary values". We abbreviate
   \begin{equation} \label{4.44A}
\widetilde{G}_{\sqrt{z}}(x)=\left\{%
\begin{array}{ll}
    \frac{e^{i\sqrt{z}|x|}}{|x|}, & \hbox{x$\neq$0;} \\
    0, & \hbox{x=0.} \\
\end{array}
\right.
   \end{equation}

   \begin{proposition}\label{prop4.14}
Let $H$ be the  Schr\"odinger operator  defined by \eqref{min}. Suppose that
$d_*(X)>0.$

$(i)$ The triplet $\Pi = \{\cH,\Gamma_0,\Gamma_1\}$, where  $\cH =
l^2(\N),$
   \begin{align}\label{otherdefbtr}
\Gamma_0 f =  \big\{\lim_{x\to x_k}
f(x)|x{-}x_k|\big\}_1^\infty =: \{\xi_{0k}\}_1^\infty , \qquad  
\Gamma_1 f=\big\{\lim_{x\to x_k}(f(x) - \xi_{0k}|x{-}x_k|^{-1}) \big\}_1^\infty,
   \end{align}
is a boundary triplet for $H^*.$

$(ii)$ The  deficiency subspace $\mathfrak {N}_{z}  = \mathfrak
{N}_{z}(H)$ is $\mathfrak N_{z} = \left\{ \sum\limits_{j=1}^\infty
c_j\varphi_{j,z}: \{ c_j\}_1^\infty \in l^2(\N) \right\}$, $z\in
\bC\backslash \bR$.

$(iii)$  The  gamma field  $\gamma(\cdot)$ of the
triplet $\Pi = \{\cH,\Gamma_0,\Gamma_1\}$ is given by
\begin{align}\label{gammabtr}
\gamma(z)(\{c_j\})=\sum_{j=1}^\infty c_j \varphi_{j,z}, \quad
\{c_j\}_1^\infty \in l^2(\bN),\quad  z\in\mathbb{C}\backslash [0,+\infty).
\end{align}

$(iv)$ The corresponding  Weyl function acts by
\begin{align}\label{W3Operator}
(M(z)\{ c_j\})_k = c_k i\sqrt{z}
 + \sumprime_{j\in \N\ }~c_j\frac{e^{i\sqrt{z}|x_k{-}x_j|}}{|x_k{-}x_j|}~,\quad\{c_j\}_{j\in \N} \in l^2(\bN),
 \quad  z\in\mathbb{C}\backslash [0,+\infty),
\end{align} that is,  the operator $M(z)$ is given by the matrix
    \begin{equation}\label{W3}
 \kM(z)=\left({i\sqrt{z}}\delta_{jk} + \widetilde{G}_{\sqrt{z}}(x_j-x_{k})\right)_{j,k=1}^\infty.
   \end{equation}
   \end{proposition}
   \begin{proof}
(i)  Since $T_0 = T_0^*\in [\mathcal H]$ and $
\widetilde{\Pi}$ is boundary triplet for $H^\ast$ by   Proposition
\ref{pr1}(iii), so is the triplet
$\Pi' =
\{\cH,\Gamma'_0,\Gamma'_1\}$, where
    \begin{equation}\label{4.42S}
\cH = l^2(\N),  \qquad \Gamma'_0 := \widetilde{\Gamma}_0, \qquad
\text{and}\qquad \Gamma'_1 = \widetilde{\Gamma}_1 + T_0
\widetilde{\Gamma}_0.
    \end{equation}
It therefore
suffices to show that ${\Gamma}_j = \Gamma'_j,\  j=0,1.$

Let $f\in \dom H^\ast$. By Proposition \ref{pr1}(ii), $f$ is of
the form   $f=f_H+f_0+(-\Delta{+}I)^{-1}f_1$, where $f_H\in \dom
(H)$, $f_0=\sum_{j\in\N} \xi'_{0j}\varphi_j$ and $f_1=\sum_{j\in\N} \xi_{1j}
\varphi_j$. Then $(-\Delta{+}I)^{-1}f_1 = 2^{-1}\sum_j \xi_{1j}
e_j$.

Fix  $k\in \bN$. Since the series $f_0=\sum_{j\in\N} \xi'_{0j}\varphi_j$
converges uniformly on compact subsets of $\R^3\setminus X$ and
$f_H\in W^{2,2}(\bR^3)$ is continuous and $f_H(x_j)=0$    by
(\ref{min}),  we get
    $$
 \xi_{0k}  = \lim_{x\to x_k} f(x)|x{-}x_k|=\xi'_{0k} =
(\widetilde{\Gamma}_0f)_k = (\Gamma'_0f)_k.
   $$
This proves the first formula of \eqref{otherdefbtr}. The second
formula is derived by
  \begin{align*}
\lim_{x\to x_k}&\big(f(x)- \xi_{0k}|x{-}x_k|^{-1}\big) =
\lim_{x\to x_k}~\left(\xi_{0k}~ \frac{e^{-|x{-}x_k|}-1}{|x{-}x_k|}
+ ~\sum_{j\neq k}^\infty
~\xi_{0j}~\frac{e^{-|x{-}x_j|}}{|x{-}x_j|}
 + 2^{-1}\sum_{j=1}^\infty~ \xi_{1j}~e^{-|x-x_j|}\right)\nonumber \\ & =
 - \xi_{0k}+ \sum_{j\neq k}^\infty ~\xi_{0j}~\frac{e^{-|x_k{-}x_j|}}{|x_k{-}x_j|} + ~2^{-1}\sum_{j=1}^\infty~ \xi_{1j}~e^{-|x_k-x_j|}
 = (T_0(\xi_{0j}))_k + (T_1(\xi_{1j}))_k  = (\Gamma'_1 f)_k,
   \end{align*}
where $T_0$ is defined by \eqref{defvarphimatrix1}, and $T_1$ is
introduced in Proposition \ref{pr1}.

(ii) follows at once from  Corollary \ref{cor4.9}.

(iii)   Clearly,  $\lim_{x\to x_k}
(\varphi_{k,z}(x){-}\varphi_k(x))|x{-}x_k|=0$. Therefore, by
(\ref{otherdefbtr}), $\Gamma_0(\varphi_{k,z} - \varphi_k)=0$ and
so $\Gamma_0 \varphi_{k,z}= \Gamma_0 \varphi_k=e_k$, where
$e_k=\{\delta_{jk}\}_{j=1}^\infty$ is the standard orthonormal
basis of $l^2(\bN)$. Hence, by \eqref{2.3A} combined with (ii),
the gamma field is of  the form given in (\ref{gammabtr}).

 (iv) Next we prove the formula for the Weyl function. Since $M$ is linear and
bounded, it suffices to prove this formula for the vectors $e_l$,
$l\in \bN$.
Fix $l\in \bN$. The function $\varphi_{l,z}\in \dom (H^\ast)$ is
of the form \eqref{eq3}, i.e., $\varphi_{l,z} = f_{H,z}  + f_{0,z}
+ (-\Delta{+}I)^{-1}f_{1, z},$ where $f_{0,z} =\sum_{j\in\N}
\xi_{0j}(z)\varphi_j$ and $f_{1,z} = \sum_{j\in\N}
\xi_{1j}(z)\varphi_j$. Then, by \eqref{otherdefbtr} and
\eqref{5.1},
    \begin{equation}\label{5.14}
 \xi_{0j}(z) =\lim_{x\to x_j} \varphi_{l,z}(x)|x{-}x_j| = \delta_{jl}, \qquad j\in \N, \qquad\text{i.e.},\quad
 f_{0,z}(x) = {|x{-}x_l|}^{-1}{e^{-|x{-}x_l|}},
   \end{equation}
so $f_{0,z}$ does not depend on $z$. Since  $\xi_{0k}(z) = 0$ for
$k\neq l$,   \eqref{otherdefbtr} and \eqref{5.1} yield
\begin{align*}\label{gammavarphi}
(\Gamma_1\varphi_{l,z})_k = \lim_{x\to x_k}(\varphi_{l,z} -
\xi_{0k}|x{-}x_k|^{-1})  = \lim_{x\to x_k}\varphi_{l,z}(x)
= \frac{e^{i\sqrt{z}|x_l{-}x_k|}}{|x_l{-}x_k|} , \qquad k\neq l,\
k,l\in \bN.
\end{align*}
Similarly, using that $\xi_{0l}(z) = 1$ it follows from
\eqref{otherdefbtr} and \eqref{5.1} that
$(\Gamma_1\varphi_{l,z})_l = i\sqrt{z}$.\ Inserting these
expressions into \eqref{2.3A} with account of \eqref{gammabtr} we
arrive at the formula \eqref{W3Operator} for the Weyl function.
     \end{proof}
   \begin{remark}
(i) Statement $(i)$ in Lemma \ref{technicallemma}  goes back to
the paper by M.I.  Vishik \cite{Vi63} and was systematically used in
the works of M. Birman  and G. Grubb \cite{Gr68}. Statement $(ii)$
is contained in a slightly different form in \cite[Remark
4]{DerMal91}.

(ii) Proposition \ref{pr1}(i)  was obtained in \cite [Lemma
4.1]{LyaMaj} for $m=1$ and for  $m <\infty$ in \cite[Theorem
1.1.2]{AGHH88}. In the  case $m = \infty$  another description of
$\dom(H_B)$ with diagonal $B=B^*$ is contained in \cite[Theorem
3.1.1.2]{AGHH88}.

 (iii)\  For $m < \infty$  another  construction of a boundary
triplet for $H_3^*$ is contained in  \cite[Proposition
4.1]{GMZ11},  while even in this case  the proof of Proposition
\ref{pr1}(iii) is simpler. In the case $m =1$ other constructions
can be found in \cite[Theorem 2.1]{LyaMaj}, \cite{BMN08}  and
\cite{HK09}.

 Another construction  of a    boundary   triplet for general  elliptic  operators
 with  boundary   conditions  on   a set  of zero  Lebesgue measure   can be found
in \cite{Koc82}.  However this construction does not allow to
compute the Weyl function and obtain other spectral results.

(iv)\   In the case $m < \infty$ the  Weyl  function in the form
\eqref{W3}
 appeared  in \cite[chapter II.1]{AGHH88}.
     In this connection  we also mention   the  paper  by  Posilicano \cite[Example  5.3]{Pos08}.
In the  case  $m=1$ the  Weyl  function was also computed
by  another  method in \cite[Section 10.3]{Ash10}.
    \end{remark}

 \subsection{Some spectral properties of self-adjoint realizations}\label{sec5.2}

In this subsection we apply the theory of boundary triplets to
describe and study self-adjoint extensions of the  minimal
Schr\"odinger  operator $H$ of the form \eqref{min}.
   \begin{proposition}\label{prop4.17}
Let $\Pi =\{\kH,\Gamma_0,\Gamma_1\}$ the boundary triplet for $H^*$
defined in Proposition \ref{prop4.14} (see \eqref{otherdefbtr}).
Let   $T_0$ be defined by \eqref{defvarphimatrix1} and   $T_1
= 2^{-1}(e^{-|x_j-x_k|})_{j,k\in \bN}$. Then:

$(i)$ The set of  self-adjoint realizations $\widetilde H\in \Ext_H$  is  
parameterized by the set of linear relations $\Theta = \Theta^*
\in \widetilde{\mathcal C}(\kH)$  as follows:
 $H_\Theta=H^*\upharpoonright \dom(H_\Theta)$, where
     \begin{equation}\label{s-adjconcr}
\dom (H_\Theta) =\left\{f=f_H +
\sum\limits^\infty_{j=1}\left(\xi_{0j} \frac{e^{-|x-x_j|}}{
|x-x_j|}  + \xi_{1j}e^{-|x-x_j|} \right) :\, f_H\in \dom(H), \,
(\xi_0, T_0\xi_0 + T_1\xi_1)\in\Theta \right\}.
  \end{equation}
Moreover,  we have   
 $\Theta = \Theta_{op}\oplus\Theta_{\infty}$ where $\Theta_{op}$
is the graph of an operator $B =B^*$  in  $\kH_0 :=
\overline{\dom(\Theta)}$ and  $\Theta_{\infty}$ is  the
multivalued part of $\Theta,$ and     $\kH =\kH_0
\oplus\kH_\infty$, where $\kH_\infty:={\rm mul} (\Theta)$ and
    \begin{gather}\label{op}
\Theta_{\infty} := \{0, \kH_\infty\} := \{\{0,
T_1\xi''_1\}:\,\xi''_1\perp T_1 \xi_0,\
 \xi_0\in \kH_0 \}, \\  
\Theta_{op} = \{\{\xi_0, T_0\xi_0 + T_1\xi'_1\}:\, \xi_0\in
\kH_0,\,\xi'_1 = T_1^{-1}( B\xi_0 - T_0\xi_0)\}. \label{mul}
  \end{gather}
In particular, $\widetilde H =
H_\Theta$ is disjoint with $H_0$ if and only
if\,   $\overline{\dom(\Theta)} = \cH = l^2(\N)$. In this case
 $\Theta=\Theta_{op}$ is the graph of $B,$ so that $
H_\Theta=H^*\upharpoonright(\ker(\Gamma_1 -B\Gamma_0)).$

$(ii)$  Let $z\in \C\setminus \overline \R_+.$ Then
$z\in\sigma_p(H_\Theta)$  
if and only if\, $0\in\sigma_p \left( \Theta -
\bigl({i\sqrt{z}}\delta_{jk}
+\widetilde{G}_{\sqrt{z}}(x_j-x_{k})\bigr)_{j,k=1}^\infty
\right).$
The corresponding   eigenfunctions $\psi_z$ have the form
   \begin{equation}\label{eigen}
\psi_z=\sum\limits_{j=1}^\infty \xi_j
|x-x_j|^{-1}{e^{i\sqrt{z}|x-x_j|}}, \qquad \text{where} \qquad
(\xi_j)\in \ker(\Theta-M(z))\subset l^2(\N).
  \end{equation}

$(iii)$  The resolvent of the extension $-\Delta_{\Theta,X} :=
H_\Theta$ admits the integral representation
    \begin{equation}\label{5.19MMM}
\bigl((-\Delta_{\Theta,X} - z)^{-1}f\bigr)(x) =
\int_{\R^3}T_{\Theta,X}(x,y;z)f(y)dy, \qquad z\in
\rho(-\Delta_{\Theta,X}),
    \end{equation}
with  kernel $T_{\Theta,X}(\cdot,\cdot;z)$ defined by
\begin{equation}\label{5.20}
T_{\Theta,X}(x,y;z) = \frac{e^{i\sqrt{z}| x-y|}}{4\pi|x-y|} +
\sum_{j,k}\Theta_{jk}(z)\frac{e^{i\sqrt{z}|y-x_j|}}{|y-x_j|}\cdot
\frac{e^{i\sqrt{z}|x-x_k|}}{|x-x_k|},
\end{equation}
where $\bigl(\Theta_{jk}(z)\bigr)_{j,k\in{\mathbb N}}$ is the matrix
representation  of the operator $\bigl(\Theta-M(z)\bigr)^{-1}$ on
$l^2(\N)$.
     \end{proposition}
        \begin{proof}
(i)  Formula  \eqref{s-adjconcr} is immediate from Proposition
\ref{propo}, formula \eqref{s-aext}.

Both formulas \eqref{op} and \eqref{mul} are proved by direct
computations. We show that \eqref{op} and \eqref{mul} imply the
self-adjointness of $\Theta$; the proof of the converse
implication is similar. Indeed, it follows from \eqref{op} and
\eqref{mul} that  
$(T_1\xi_1'',\xi_0) =0 =(\xi_0, T_1\xi_1'')$ and
   \begin{equation}
(T_1\xi'_1,\xi_0)=( B\xi_0 - T_0\xi_0,\xi_0)=(\xi_0,  B\xi_0 -
T_0\xi_0)=(\xi_0, T_1\xi'_1).
   \end{equation}
Hence we have $(T_1\xi_1,\xi_0)=(\xi_0, T_1\xi_1)$ for all
$(\xi_0,\xi_1)\in \Theta$. It is easily checked    that the latter
condition is equivalent to the  self-adjointness    of the
relation $\Theta$.

(ii) The   symmetric operator  $H$ is in general not simple. It
admits  a direct sum decomposition $H = \widehat{H}\oplus H'$
where $\widehat{H}$ is a simple symmetric operator  and $H'$ is
self-adjoint. Define
$\widehat\Pi=\{\cH,\widehat\Gamma_0,\widehat\Gamma_1\}$, where
$\widehat\Gamma_j := \Gamma_j\upharpoonright \dom(\widehat{H}^*),
 \ j\in\{0,1\}.$  Clearly,  $\widehat\Pi$
is a boundary triplet  for $\widehat{H}^*$ and the corresponding
Weyl function $\widehat{M}(\cdot)$ coincides with  the Weyl
function $M(\cdot)$ of $\Pi$. Further,  any proper extension
$\widetilde H = H_\Theta$ of $H$ admits  a decomposition $H_\Theta
= \widehat{H}_\Theta\oplus H'.$  Being a part of  $H_0$, the
operator $H'$ is non-negative. Therefore, for  $z\in \C\setminus
\overline \R_+$, we have $z\in\sigma_p(H_\Theta)$ if and only if $z\in\sigma_p(\widehat{H}_\Theta)$. Thus, it
suffices to prove the assertion  for  extensions
$\widehat{H}_\Theta$ of the \emph{simple symmetric operator}
$\widehat{H}.$ But then the statement follows from  Propositions
\ref{prop_II.1.4_spectrum} and \ref{prop4.14}(ii) and
formula \eqref{gammabtr}.

(iii) Noting that ${i\sqrt{\overline z}} = \overline{i\sqrt{z}}$
it follows from  \eqref{5.1} that
$\varphi_{j,\overline{z}}=\overline{\varphi_{j,z}}$. Therefore,  \eqref{gammabtr} implies that
    \begin{equation}\label{5.21}
\gamma^*(\overline{z})f=\sum^{\infty}_{k=1}\left(\int_{{\mathbb
R}^3}f(x)\overline{\varphi_{k,\overline{z}}(x)}dx\right) e_k =
\sum^{\infty}_{k=1}\left(\int_{{\mathbb
R}^3}f(x)\frac{e^{i\sqrt{z}|x-x_k|}}{|x-x_k|}dx\right) e_k,
    \end{equation}
where $e_k=\{\delta_{jk}\}^{\infty}_{j=1}$ is the standard
 basis of $l^2(\mathbb N)$.

Inserting  \eqref{5.21} and  \eqref{gammabtr} into the
Krein type formula \eqref{2.8} and  applying the formula
\eqref{3.2S} for the resolvent of the free Hamiltonian $-\Delta$,
we obtain
    \begin{eqnarray*}\label{5.19M}
\bigl((-\Delta_{\Theta,X} - z)^{-1}f\bigr)(x) =
\int_{\R^3}\frac{e^{i\sqrt{z}| x-y|}}{4\pi|x-y|}\ f(y)dy + \
\sum^{\infty}_{j,k}[\bigl(\Theta - M(z)\bigr)^{-1}]_{j,k} \bigl(f,
\overline{\varphi_{k,z}}\bigr)\varphi_{j,z}(x).
    \end{eqnarray*}
Cleraly, the latter  is equivalent to the representations
\eqref{5.19MMM}--\eqref{5.20}.
      \end{proof}

Next we turn to nonnegative or lower semibounded self-adjoint
extensions of $H$. For this  we need the following technical result.
   \begin{lemma}
Retain the assumptions of Proposition \ref{prop4.14} and let  $\Pi
= \{\cH,\Gamma_0,\Gamma_1\}$ be the boundary triplet for
 $H^*$ defined therein. Then:

(i)\  There exists a lower semibounded self-adjoint operator
$M(0)$ on $\kH=l^2(\bN)$ which is the limit of $M(-x)$ in the
strong resolvent convergence as $x\to +0$.

(ii)\  The  quadratic form  $\mathfrak t_{M(0)}$ of $M(0)$ is given by
     \begin{align}\label{5.17}
\mathfrak t_{M(0)}[\xi] = \sum_{|j-k|>0}
\frac{1}{|x_j-x_k|}~\xi_j\overline{\xi}_k, &\quad \dom(\mathfrak
t_{M(0)}) = \big\{\xi= \{\xi_j\}\in l^2(\N): \sum_{|j-k|>0}
\frac{1}{|x_j-x_k|}~\xi_j \overline{\xi}_k < \infty \big\}.
   \end{align}

(iii) The operator $M(0)=M(0)^*$  associated with the form
$\mathfrak t_{M(0)}$ is uniquely determined by the following
conditions:\  $\dom\bigl(M(0)\bigr)\subset \dom(\mathfrak t_{M(0)})$
and
   \begin{eqnarray}\label{5.19}
\bigl(M(0)\xi,\eta\bigr) = \sum_{|j-k|>0}
\frac{1}{|x_j-x_k|}~\xi_j\overline{\eta}_k, \qquad \xi =
 \{\xi_j\}\in\dom\bigl(M(0)\bigr),\  \eta =  \{\eta_j\}\in \dom(\mathfrak t_{M(0)}).
\end{eqnarray}

(iv)\  If, in addition,
$\sum\nolimits^{'}_{j\in \N}|x_j-x_k|^{-2}<\infty \  \text{for every}\
k\in{\N},$
then $e_k\in \dom(M(0)), \ k\in \N,$ where
$e_k=\{\delta_{jk}\}^{\infty}_{j=1}$ is the standard orthonormal basis
of $l^2(\N)$ ,
and the matrix
   \begin{equation}
\mathcal M'(0):=   \left( \frac{1 -\delta_{kj}}{|x_k-x_j| +
\delta_{kj}}\right)_{j,k=1}^{\infty},
   \end{equation}
defines a (minimal) closed symmetric operator $M'(0)$ on
$l^2({\N})$.
Moreover,
   \begin{equation}
\dom\bigl(M'(0)^*\bigr) = \left\{\{\xi_j\}\in l^2(\N):
\sum\nolimits_{j\in \N} \big|\sum\nolimits^{'}_{k\in \N}{|x_j-x_k|^{-1}}\xi_k
\big|^2  < \infty \right\}.
   \end{equation}

(v) The  operator $M'(0)$ is semibounded from below and  its
Friedrichs extension $M'(0)_F$ coincides with $M(0)$, that is,
$M'(0)_F = M(0)$.
   \end{lemma}
\begin{proof}
(i)  The assertion follows  by combining Propositions
\ref{prop_II.1.5_01}(i) and \ref{prop4.14}(iv) (cf.
formulas \eqref{W3} and \eqref{4.44A}).

(ii)  By Proposition \ref{prop_II.1.5_01}(i),
   \begin{equation}\label{5.22}
\mathfrak t_{M(0)}[\xi] := \lim_{t\downarrow
0}\bigl(M(-t)\xi,\xi\bigr), \qquad \xi\in\dom(\mathfrak t_{M(0)}) :=
   \{\eta: \lim_{t\downarrow
0}\bigl(M(-t)\eta, \eta\bigr) < \infty\}. 
   \end{equation}
Let us denote for the moment the form defined in \eqref{5.17} by $\mathfrak
t_0$. We have  to show that $\mathfrak t_0 = \mathfrak t_{M(0)}.$

Note that the function $f(t) = (1-e^{-t})/t=\int^1_0 e^{-st}ds$ is
absolutely monotone, $f\in M[0,\infty).$ Hence $f\in\Phi_3$. This
fact together with  \eqref{W3} and \eqref{5.17}
yields
    \begin{equation}\label{5.23}
\mathfrak t_0[\xi]-\bigl(M(-t)\xi,\xi)= \sum_{|k-j|>0}\frac{1 -
e^{-t|x_j-x_k|}}{|x_j-x_k|}\, \xi_{j}\overline{\xi_{k}} > 0, \quad
t>0, \quad \xi = \{\xi_j\}_{1}^\infty\in\dom(\mathfrak t_0).
    \end{equation}
Thus, for any $\xi\in \dom(\mathfrak t_0)$ the limit
$\lim_{t\downarrow 0}\bigl(M(-t)\xi,\xi)$ is finite
and by \eqref{5.22}, $\dom(\mathfrak t_0)\subset\dom(\mathfrak t_{M(0)})$.

Now we prove  that
$\mathfrak t_{M(0)}[\xi]=\mathfrak t_0[\xi]$ for all $\xi\in\dom(\mathfrak t_0)$.
 For finite  vectors this  follows at once  from \eqref{5.23} and
\eqref{5.22}. Fix $\xi\in\dom(\mathfrak t_0)$. Given $\varepsilon >0 $ it follows from \eqref{5.17} and \eqref{5.22}  that there exists $N\in{\mathbb N}$ such that the
finite vector  $\xi^{(N)}:=\{\xi_j\}^N_1$ satisfies
   \begin{equation}\nonumber
|\mathfrak t_0[\xi] -\mathfrak t_0[\xi^{(N)}]|<\varepsilon \qquad
\text{and}\qquad |\mathfrak t_{M(0)}[\xi]-\mathfrak
t_{M(0)}[\xi^{(N)}]|<\varepsilon.
   \end{equation}
Then $|\mathfrak t_0[\xi]-\mathfrak t_{M(0)}[\xi]|<2 \varepsilon.$ Since $\varepsilon >0$ was arbitaray, this implies that $\mathfrak t_{M(0)}[\xi]=\mathfrak t_0[\xi]$.

The equality  $\dom(\mathfrak t_0) = \dom(\mathfrak t_{M(0)})$ is obvious.

(iii) follows from (ii) and  the first form representation theorem
(cf. \cite[Theorem 6.2.1]{Kato66}).

(iv)  By the assumption
$\sum\nolimits^{'}_{j\in \N}|x_j-x_k|^{-2}<\infty,$  we have $e_k\in
\dom(M(0))$. Now \cite[Theorem 56.4]{AkhGlz78} gives the first
assertion, while the
second  follows from \cite[Theorem 56.2]{AkhGlz78}.

(v) Define a quadratic form $\mathfrak t'_0$ by  $\mathfrak t'_0[\xi] :=
\bigl(M'(0)\xi,\xi\bigr),$\  $\xi\in \dom(\mathfrak t'_0) =
\dom(M'(0)).$ Clearly, the finite vectors are dense in $\dom(\mathfrak t_{M(0)})$ with respect to  the
norm $\|\xi\|^2_+ := \mathfrak t_{M(0)}[\xi] +  C\|\xi\|^2$ for
sufficiently large $C>0$. Since  $\mathfrak t'_0[\eta] = \mathfrak
t_{M(0)}[\eta]$, the closure of the form $\mathfrak t'_0$ is $\mathfrak
t_{M(0)}$. Since
$M(0)=M(0)^*$ and $\dom\bigl(M(0)\bigr)\subset\dom\mathfrak t_{M(0)},$ this completes the proof.
     \end{proof}
   \begin{remark}
As above, let $f(t) = (1-e^{-t})/t.$ By Theorem
\ref{propositionstronglydef}, $f(|\cdot|)$ is
 strictly $X$-positive definite, hence the quadratic form
 $\mathfrak t_0 - \mathfrak t_{M(-t)}$ in \eqref{5.23}
 is strictly positive definite. However, note that this form  is bounded
 from above if and only if $M(0)$ is bounded.
The latter depends on the set $X$ and shows that the assumption
\eqref{fourthmoment}  in Theorem  \ref{propositionstronglydef} is
essential.
     \end{remark}
%
  \begin{theorem}\label{prop4.20}
Let  $\Pi = \{\cH,\Gamma_0,\Gamma_1\}$ be the boundary triplet for
$H^*$ defined in Proposition \ref{prop4.14}, $M$ the corresponding
Weyl function  and let $\Theta$ be  a self-adjoint
relation on $\kH$. Then    

$(i)$   The operator $H_0:= H^*\upharpoonright \ker\Gamma_0$
is  the free Laplacian   $H_0 = -\Delta,$\ $\dom(H_0) =
\dom(\Delta) = W^{2,2}({\mathbb R}^3)$. Moreover, $H_0$ is the
Friedrichs extension  $H_F$ of $H$ and $\dom(\mathfrak t_{H_0}) =
W^{1,2}({\R}^3)$.

$(ii)$ The operator $H_{M(0)}$ is the  Krein  extension $H_K$ of
$H$ and given by $H_K = H^*\!\upharpoonright \dom(H_K)$, where the
domain $\dom(H_K)$  is the direct sum of $ \dom(H)$ and the vector
space
\begin{align*}
   \left\{\sum\limits^\infty_{j=1}\bigl(\xi_{0j}\varphi_j
+\xi_{1j} e_j \bigr) :\   \{\xi_{1j}\} = T_1^{-1}\bigl(M(0)
-T_0\bigr)\xi_0, \qquad
\{\xi_{0j}\}\in \dom(M(0)) \right\}. \nonumber
     \end{align*}
The extensions $H_0 = H_F$ and $H_K$  are disjoint. They are
transversal if and only if the operator $M(0)$  is bounded on
$l^2(\N).$ For instance, this is true  whenever condition
\eqref{matrixs1} is satisfied.

(iii)  $H_{\Theta}\ge 0$
if and only if  $\Theta$ is semibounded below, $\dom(\mathfrak
t_{\Theta}) \subset \dom(\mathfrak t_{M(0)})$ and $\mathfrak t_{\Theta}
\ge \mathfrak t_{M(0)}.$

In particular, $H_{\Theta}\ge 0$ when
$\dom(\Theta)\subset\dom\bigl(M(0)\bigr)$ and $\Theta-M(0)\ge 0$.

$(iv)$  $H_{\Theta}$ is lower semibounded  if and only if $\Theta$
is. In this case the  quadratic from $\mathfrak t_{H_{\Theta}}$ is
   \begin{align}\label{4.65B}
\dom(\mathfrak t_{H_{\Theta}}) &= W^{1,2}({\mathbb R}^3) \dotplus \bigg\{
\sum^{\infty}_{j=1}\xi_j\varphi_{j}:~~\xi=\{\xi_j\}_{j\in\N} \in
\dom(\mathfrak t_{\Theta})\subset l^2(\N)\bigg\},\\
\mathfrak t_{H_{\Theta}}[f]  +  \|f\|^2_{L^2} &= \int_{{\mathbb
R}^3}\bigl(|\nabla g(x)|^2 +  |g(x)|^2\bigr)dx  + \mathfrak
t_{\Theta}[\xi] -
\sum_{|k-j|>0}\frac{e^{-|x_j-x_k|}}{|x_j-x_k|}\,
\xi_{j}\overline{\xi_{k}} ~,\label{4.67}
  \end{align}
where $f = g + \sum_{j\in \N}\xi_j\varphi_{j} \in\dom(\mathfrak
t_{H_{\Theta}})$ with $g\in W^{1,2}({\mathbb R}^3)$ and
$\xi=\{\xi_j\}_{j\in\N} \in \dom(\mathfrak t_{\Theta})$.

(v) In particular, for the quadratic form  $\mathfrak t_{H_K}=\mathfrak
t_{H_{M(0)}}$ we  have
  \begin{align}\label{4.65}
\dom(\mathfrak t_{H_K}) = W^{1,2}({\mathbb R}^3) &\dotplus
\bigg\{ \sum^{\infty}_{j=1}\xi_j\varphi_j:~~ \{\xi_j\}_1^\infty\in l^2(\bN),~ \sum\nolimits_{|k-j|>0} |x_j-x_k|^{-1}\xi_j\overline{\xi}_k < \infty\bigg\},\\
\mathfrak t_{H_K}[f]  +  \|f\|^2_{L^2}& = \int_{{\mathbb R}^3}|\nabla
g(x)|^2 dx  +  \|g\|^2_{L^2} + \sum\nolimits_{|k-j|>0}\frac{1 -
e^{-|x_j-x_k|}}{|x_j-x_k|}\,
\xi_{j}\overline{\xi_{k}}~,\label{4.67BB}
    \end{align}
where $f = g + \sum_{j\in\N}\xi_j\varphi_{j} \in\dom(\mathfrak t_{H_{M(0)}})$
with $g\in W^{1,2}({\mathbb R}^3)$ and $\{\xi_j\}_{j\in\N}\in \dom (\mathfrak
t_{M(0)})$.

$(vi)$  If $\Theta$
is lower semibounded  and $\dom(\mathfrak t_{\Theta})\subset\dom(\mathfrak
t_{M(0)})$, then $\kappa_-(H_{\Theta}) = \kappa_-(\mathfrak
t_{\Theta -M(0)})$. \\ If, in addition,
$\dom(\Theta)\subset\dom\bigl(M(0)\bigr)$, then $\kappa_{-}(H_\Theta)= \kappa_{-}(\Theta-M(0))$.

(vii) If $M(0)$ is bounded, i.e. $H_K$ and $H_F$ are transversal,
we have the implication
    \begin{equation}\label{5.33A}
\bigl(\Theta-M(0)\bigr)E_{\Theta-M(0)}(-\infty,0)\in\mathfrak S_p(\cH)
\quad \Longrightarrow \quad H_{\Theta}E_{H_{\Theta}}(-\infty,0)
\in \mathfrak S_p(\mathfrak H).
    \end{equation}
For instance, implication  \eqref{5.33A}  holds  whenever
condition \eqref{matrixs1} is satisfied.
    \end{theorem}
  \begin{proof}
$(i)$ The first statement is immediate from  \eqref{eq3} and
definition \eqref{otherdefbtr} of $\Gamma_0.$

Further, integrating by parts one gets
    \begin{equation}\label{4.69}
\mathfrak t'_H[f] + \|f\|^2_{L^2} := (Hf, f) + \|f\|^2_{L^2} =
\int_{{\mathbb R}^3}|\nabla f(x)|^2 dx + \|f\|^2_{L^2} =:
\|f\|^2_{W^{1,2}}, \quad  f\in \dom(H).
    \end{equation}
Since $\dom(H)$ is dense in $W^{1,2}({\mathbb R}^3),$ the closure
$\mathfrak t_H$ of $\mathfrak t'_H$ is defined by \eqref{4.69} on the
domain $\dom(\mathfrak t_H)=W^{1,2}({\R}^3)$. 
Noting  that $\dom(\mathfrak t_{H_0})= W^{1,2}({\mathbb R}^3) =
\dom(\mathfrak t_H)$ we get the result.

We present another proof that is based on  the Weyl function. It
follows from \eqref{W3} and \eqref{4.44A} that
$\lim_{x\downarrow-\infty}\bigl(M(x)h,h\bigr) = -\infty$ for
$h\in\cH\setminus\{0\}.$  It remains to apply Proposition
\ref{prop_II.1.5_01}(iii).

(ii) By Proposition \ref{prop_II.1.5_01},
$\dom(H_K)=\ker\bigl(\Gamma_1-M(0)\Gamma_0\bigr)$ since $H_K$ and $H_0 = H_F$ are disjoint. Inserting the expressions from \eqref{otherdefbtr} and \eqref{4.42S} for $\Gamma_1$ and $\Gamma_0$
we get the result.

(iii) follows immediately from Proposition \ref{prkf}(i).

(iv): Let $\xi=\{\xi_j\}_1^\infty \in l^2[\bN)\}$. Set
$|\xi|:=\{|\xi_j|\}_{j\in\N}$. Then we derive from \eqref{W3}
   \begin{align}\label{5.33}
&\big|\langle M(-t^2)\xi,\xi\rangle +
\frac{t}{4\pi}\|\xi\|_{l^2}^2 ~\big| \le\bigg|
\sum\nolimits_{|k - j|>0}\frac{e^{-t|x_j-x_k|}}{|x_j-x_k|}\xi_j\overline{\xi}_k\bigg|
\le\frac{1}{d_*(X)} \sum_{j,k\in\N} {e^{-t|x_j-x_k|}}  |
\xi_j\overline{\xi}_k|    \nonumber   \\
& \le d_*(X)^{-1}e^{-(t - 1)d_*(X)}
\sum\nolimits_{j,k\in \N} {e^{-|x_j-x_k|}} |\xi_j\overline{\xi}_k| \nonumber  \\
& =  d_*(X)^{-1}e^{-(t - 1)d_*(X)} 2\big|\langle T_1 |\xi|, |\xi|
\rangle _{l^2(\bN)} \big|  \nonumber  \\
& \le  d_*(X)^{-1}e^{(1-t)d_*(X)}2 \cdot\|T_1\|\cdot
\|\xi\|_{l^2(\bN)}^2.
   \end{align}
For any $\varepsilon>0$,  $\varepsilon < \|T_1\|d_*(X)^{-1}$, we define
$t_0 = t_0(\varepsilon)$ by
   \begin{equation}\label{4.53S}
 t_0 = t_0(\varepsilon) = 1 - \ln\bigl(\varepsilon
d_*(X)\|T_1\|^{-1}\bigr).
       \end{equation}
Then it follows from \eqref{5.33} that
   \begin{equation}
\bigl(M(-t^2)\xi,\xi\bigr)\ge
-(\frac{t}{4\pi}+\varepsilon)\|\xi\|^2_{l^2},\qquad   t\ge t_0,
   \end{equation}
and hence $M(-t^2)\rightrightarrows -\infty$. Now  Proposition \ref{prop_II.1.5_02} yields the first assertion.

Next we  prove the second statement.  By  \cite[Theorem 1]{Mal92},
the domain $\dom(\mathfrak{t}_{H_{\Theta}})$ is a  direct sum
\begin{equation}\label{4.73}
\dom(\mathfrak{t}_{H_{\Theta}}) = \dom(\mathfrak{t}_{H}) \dotplus\gamma(-
\varepsilon^2)\dom(\mathfrak{t}_{\Theta}), \quad \varepsilon>0.
    \end{equation}
Hence any $f\in\dom(\mathfrak{t}_{H_{\Theta}})$ can be written as $f=g
+ \gamma(-\varepsilon^2)h,$ where $g\in \dom(\mathfrak{t}_{H})$ and
$h\in\dom(\mathfrak{t}_{\Theta})$. Noting that
$\dom(\mathfrak{t}_H) = W^{1,2}({\mathbb R}^3),$ and combining \eqref{4.73}
with \eqref{gammabtr}\  yields  \eqref{4.65B}.

Further, by  \cite[Theorem 1]{Mal92} we have the
equality
     \begin{equation}
\mathfrak{t}_{H_{\Theta}}[f] +\|f\|^2 = \mathfrak{t}_H[g] +
\|g\|^2 +
\mathfrak{t}_{\Theta}[h]-\bigl(M(-1)h,h\bigr), \qquad f :=
g + \gamma(-1)h.
  \end{equation}
Using  Proposition \ref{prop4.14}(iv)  and the equality
$\mathfrak{t}_H[g] = \int_{{\mathbb R}^3}|\nabla g(x)|^2 dx$  we obtain \eqref{4.67}.

(v) follows from (iv) with $\Theta = M(0).$

(vi)  By (i),  $H_0= H_F.$ Hence the assertion is immediate from
Proposition \ref{prkf}(ii).

(vii)  Since $H_0$ is the Friedrichs extension  of $H,$ \cite[Theorem 3]{Mal92} implies the assertion.
    \end{proof}
  \begin{remark}
It follows from  \eqref{4.65} and \eqref{eq3} that the inclusion
    \begin{equation}
\dom(\mathfrak t_{H_K}) = W^{1,2}({\mathbb R}^3)\dotplus
\gamma(-1)\dom\mathfrak t_{M(0)}\supset W^{2,2}({\mathbb R}^3)\dotplus
\mathfrak N_{-1} = \dom(H^*)
    \end{equation}
holds  if and only if the operator $M(0)$ is bounded. This fact illustrates
the following general result: for any non-negative operator $A$ the inclusion
 $\dom(\mathfrak t_{A_K})\supset \dom(A^*)$ holds if and only if
$A_K$ and $A_F$ are transversal (see
 \cite[Remark 3]{Mal92}).
  \end{remark}
      \begin{remark}
(i)\ The Krein type formula \eqref{5.19MMM}-\eqref{5.20}
was established in  \cite[Theorem 3.1.1.1]{AGHH88}  for a special family
$H^{(3)}_{X,\alpha}$ of self-adjoint extensions by approximation method. 
In our notation this family is parameterized by the set
of self-adjoint  diagonal matrices $B_\alpha =
\diag(\alpha_1,..,\alpha_m,\ldots ).$ In this  case
   \begin{equation}\label{fam}
H^{(3)}_{X,\alpha}= H^*\upharpoonright\bigg\{
f=f_H+\sum\limits^\infty_{j=1}\xi_{0j} \frac{e^{-|x-x_j|}}{
|x-x_j|} + \sum\limits^\infty_{k,j=1} b_{jk}(\alpha)\xi_{0k}
e^{-|x-x_j|} \bigg\},\quad
   \end{equation}
where $\widetilde{B}_\alpha =(b_{jk}(\alpha))_{j,k=1}^\infty =
T_1^{-1}(B_\alpha - T_0)$. It is proved in \cite[Theorem
3.1.1.1]{AGHH88} that $H^{(3)}_{X,\alpha}$ is self-adjoint.
Other  parameterizations  of the set of  self-adjoint realizations
are  also contained in \cite{Koc82} (see also the references
therein) and \cite[Example 3.4]{Pos01}. Another version of
formula \eqref{5.19MMM}-\eqref{5.20}  as well as an
abstract  Krein-like formula for resolvents  can also be found in \cite{Pos01}.

(ii)\  In the case of finitely many point interactions
$(m<\infty)$ different descriptions of non-negative realizations
has been obtained in \cite{ArlTse05, HK09, GMZ11}.

(iii)\  In connection with Theorem \ref{prop4.20}(iv) we
mention the papers \cite{KosMal10} and \cite{Gru2011}  where
similar statements have been obtained for realizations of $1D$
Schr\"odinger operators \eqref{eq0} with $d_*(X)\ge 0$ and elliptic
operators in exterior domains, respectively.
   \end{remark}

\subsection{Ac-spectrum of   self-adjoint extensions}\label{AcSection}

   \begin{theorem}\label{spec1}
Let  $d_\ast(X)>0$  and let $\Pi=\{\cH, \Gamma_0,\Gamma_1\}$ be the
boundary triplet for $H^*$ defined in Proposition \ref{prop4.14}.
 Suppose  that $\Theta$ is a self-adjoint relation on $\kH$.
Then

$(i)$  For any $p\in(0,\infty]$ we have
the following equivalence:
    \begin{equation}\label{4.77A}
(H_{\Theta} - i)^{-1} - ( H_0   - i)^{-1} \in \mathfrak S_p(\mathfrak H)
\Longleftrightarrow ({\Theta} - i)^{-1}  \in \mathfrak S_p(\cH).
    \end{equation}

$(ii)$  If $(\Theta - i)^{-1}\in\mathfrak S_1(\cH)$, then the
non-negative  $ac$-part $H^{ac}_{\Theta} =
H^{ac}_{\Theta}E_{H_{\Theta}}(\overline \R_+)$ of the  operator
$H_{\Theta} = H_{\Theta}^*$ is unitarily equivalent to the
Laplacian  $-\Delta.$

$(iii)$  Suppose that $(\Theta-i)^{-1}\in\mathfrak S_{\infty}(\cH)$ and  condition \eqref{matrixs1} is
satisfied, i.e.,
  \begin{align}\label{matrixs14A}
C_1 :=   {\rm sup}_{j\in \bN}~ \sumprime_{k\in \N\ }
\frac{1}{|x_k{-}x_j|}\, <\infty.
\end{align}
Then the  $ac$-part $H^{ac}_{\Theta} =
H^{ac}_{\Theta}E_{H_{\Theta}}(\overline \R_+)$ of $H_{\Theta}$ is
unitarily equivalent to the  Laplacian   $-\Delta$.
    \end{theorem}
          \begin{proof}
(i) This assertion follows at  once   from Proposition \ref{prop2.9}.

(ii) By Proposition \ref{prop4.20}(i) $H_0 = -\Delta$. Therefore,
by \eqref{4.77A} with $p=1,$  $[(H_{\Theta} - i)^{-1} - ( -\Delta
- i)^{-1}] \in \mathfrak S_1(\mathfrak H)$. It remains to apply the
Kato-Rosenblum theorem (see \cite{Kato66}).

(iii)  Let $z=t+iy\in{\mathbb C}_+, t>0$, and
$\sqrt{z}=\alpha+i\beta$. Clearly, $\alpha>0, \beta>0$ and
$i\sqrt{z}=i\alpha-\beta$. It follows from \eqref{4.44A} that
    \begin{equation}\label{4.66}
\widetilde G_{\sqrt{z}}(|x_j-x_k|) = \frac{|e^{(-
\beta+i\alpha)|x_j-x_k|}|}{|x_j-x_k|} =
\frac{e^{-\beta|x_j-x_k|}}{|x_j-x_k|}\, , \qquad j\not = k.
    \end{equation}
It follows from  \eqref{W3}  combined with \eqref{matrixs14A}
and \eqref{4.66} that
   \begin{eqnarray*}
\|M(t+iy)\|\le\sqrt{\alpha^2+\beta^2}+e^{-\beta} {\rm sup}_{j\in
\bN}~~ \sumprime_{k\in \N\ }
\frac{1}{|x_k{-}x_j|}   \nonumber   \\
=\sqrt{\alpha^2+\beta^2}+ C_1 e^{-\beta}\le\sqrt{t}+1 + C_1,
\qquad y\in[0,1].
   \end{eqnarray*}
Thus,  for any fixed $t>0$ the family $M(t+iy)$ is uniformly
bounded for $y\in(0,1]$, hence the weak limit $M(t+i0) :=
w-\lim_{y\downarrow 0}M(t+iy)$ exists and
    \begin{equation}
w-\lim_{y\downarrow 0}M(t+iy) =:  M(t+i0) =: M(t) =
i\sqrt{t}I+\bigl(\widetilde
G_{\sqrt{t}}(|x_j-x_k|)\bigr)^{\infty}_{j,k=1}.
    \end{equation}
From \eqref{4.77A}, applied with $p=\infty$, we  conclude  that
$[(H_\Theta-z)^{-1} - (H_0 - z)^{-1}] \in \mathfrak S_{\infty}(\mathfrak H)$
since $\bigl(\Theta - i\bigr)^{-1} \in \mathfrak S_{\infty}(\cH).$ To
complete the proof it suffices to  apply \cite[Theorem 4.3]{MalNei09} to
$H_\Theta$ and $H_0 = -\Delta.$
          \end{proof}

To prove  the  next result we need the following auxiliary lemma
which is of interest in itself.

   \begin{lemma}\label{lem4.22}
Suppose that $A$ is a \emph{simple }  symmetric operator in $\mathfrak H$  and $\{\cH,\Gamma_0,\Gamma_1\}$ is
a boundary triplet for $A^*$ with Weyl function  $M$.  Assume that for any $t\in
(\alpha, \beta)$ the uniform limit
    \begin{equation}\label{4.80}
M(t):=M(t+i0):=u-\lim_{y\downarrow 0}M(t+iy)
    \end{equation}
exists and $0\in\rho\bigl(M_I(t)\bigr)$ for $t\in(\alpha,\beta)$. Then
the spectrum of any self-adjoint extension $\widetilde A $ of $A$ on $\mathfrak H$
 in the interval  $(\alpha,\beta)$ is
purely absolutely continuous, i.e.,
   \begin{eqnarray}\label{4.81}
\sigma_s(\widetilde A)\cap(\alpha,\beta) = \emptyset.
   \end{eqnarray}
The operator  $\widetilde A E_{\widetilde A}(\alpha,\beta) = \widetilde A^{ac} E_{\widetilde
A}(\alpha,\beta)$ is unitarily equivalent to $A_0
E_{A_0}(\alpha,\beta)$, where $A_0 = A^*\lceil \ker\Gamma_0.$
    \end{lemma}
  \begin{proof}
  Without loss  of generality we can asume that the extensions
 $\widetilde A$ and $A_0$ are disjoint. Then, by Proposition
\ref{propo}(iii),  there is  a self-adjoint operator $B$ on $\kH$ such that $\widetilde{A}= A_B$, where $A_B =
A^*\!\upharpoonright\ker\bigl(\Gamma_1-B\Gamma_0\bigr)$.

We set $M_B(t+iy):=\bigl(B-M(t+iy)\bigr)^{-1}$ and note that
    \begin{equation}\label{4.82}
\imm(M_B(t + iy))=(B - M(t + iy))^{-1}\imm (M(t + iy))(B - M^*(t +
iy))^{-1},\quad y\in \R_+.
    \end{equation}

Fix $t\in(\alpha,\beta)$. By  assumption we have $0\in\rho\bigl(M_I(t)
\bigr)$, i.e., there exists $\varepsilon = \varepsilon(t)$ such
that
   \begin{equation}\label{4.83}
\langle M_I(t)h, h\rangle\ge\varepsilon\|h\|^2, \qquad    h\in\cH.
   \end{equation}
It follows from \eqref{4.80} that there exists $y_0\in{\R}_+$ such
that
   \begin{equation}\label{4.84}
\|M_I(t+iy)-M_I(t)\| \le \varepsilon/2 \quad \text{for}\quad
y\in[0,y_0).
   \end{equation}
Combining \eqref{4.83} with \eqref{4.84} we get
   \begin{eqnarray*}
\langle M_I(t+iy)h,h\rangle = \langle M_I(t)h,h\rangle
+ \langle\bigl(M_I(t+iy)-M_I(t)\bigr)h,h\rangle  \\
\ge 2^{-1}\varepsilon \|h\|^2,\qquad  y\in[0,y_0).
   \end{eqnarray*}
Hence, for any $h\in\dom(B)$,
   \begin{align*}
&\|\bigl(M(t+iy) - B\bigr)h\|\cdot\|h\|\ge|\langle\bigl(M(t+iy)-B\bigr)h,h\rangle|  \\
&\ge \imm\langle\bigl(M(t+iy) - B\bigr)h,h\rangle = \langle
M_I(t+iy)\bigr)h,h\rangle \ge 2^{-1}\varepsilon\|h\|^2,\qquad
y\in[0,y_0).
   \end{align*}
Since $0\in \rho(M(t+iy)-B),$  the latter  inequality is
equivalent to
     \begin{equation}\label{4.85}
\|\bigl(M(t+iy)-B\bigr)^{-1}\|\le 2\varepsilon^{-1}, \qquad
y\in[0,y_0).
     \end{equation}
It follows  that
    \begin{align*}\label{4.86}
&\|\bigl(B-M(t+iy)\bigr)^{-1} - \bigl(B-M(t)
\bigr)^{-1}\|  \nonumber \\
& = \|\bigl(B - M(t+iy)\bigr)^{-1} [M(t+iy)-M(t)] \bigl(B-
M(t)\bigr)^{-1}\| \\
& \le 4\varepsilon^{-2} \|M(t+iy)-M(t)\|,  \qquad y\in[0,y_0).
   \end{align*}
Hence
   \begin{equation}\label{4.86A}
u-\lim_{y\downarrow 0}\bigl(B-M(t+iy)\bigr)^{-1} = \bigl(B-M(t)
\bigr)^{-1}.
      \end{equation}

Next, it is easily seen that $\Pi_B =\{\mathcal{H}, \Gamma_0^B,\Gamma_1^B\},$ 
where $\Gamma_0^B=B\Gamma_0-\Gamma_1,\  \Gamma_1^B=\Gamma_0,$\
is a generalized  boundary triplet for $A_*\subset A^*$,
$\dom(A_*) = \dom(A_0) + \dom(A_B)$ (see \cite{DerMal91} for the definitions). The corresponding Weyl function
is $M_B(\cdot)=(B - M(\cdot))^{-1}$. Therefore, combining
\eqref{4.86A} with \cite[Theorem 4.3]{BraMal02}, we get
$\tau_s(A_B)\cap(\alpha,\beta)=\emptyset$, i.e., $\widetilde A
E_{\widetilde A}(\alpha,\beta) = \widetilde A^{ac} E_{\widetilde
A}(\alpha,\beta)$.

Moreover, passing to the limit in \eqref{4.82} as $y\downarrow 0,$
and using \eqref{4.80} and \eqref{4.86A}, we obtain
    \begin{equation}\label{4.87}
\imm(M_B(t + i0)) = (B - M(t + i0))^{-1} M_I(t + i0)(B - M^*(t +
i0))^{-1}, \qquad t\in (\alpha,\beta).
    \end{equation}
Since $\ker\bigl(B-M^*(t+i0)\bigr)^{-1}=\{0\},$ we have
   \begin{equation}
\rank \bigl(\imm(M_B(t + i0))\bigr)  = \rank \bigl(\imm(M_I(t +
i0))\bigr),  \qquad t\in(\alpha,\beta).
   \end{equation}
 By Proposition \ref{Prop-ac}  the
operators $A_B E_{A_B}(\alpha,\beta)$ and $A_0
E_{A_0}(\alpha,\beta)$ are unitarily equivalent.
  \end{proof}

Now we are ready to prove the main result of this section.
        \begin{theorem}\label{AcspecTheorem}
Let   ${\widetilde H}$ be a self-adjoint extension of ${H}.$
Suppose that
  \begin{equation}\label{4.90HScond}
C_2:= \sum_{|k-j|>0}\frac{1}{|x_j-x_k|^2} <\infty.
  \end{equation}

$(i)$ Then the part ${\widetilde H}E_{\widetilde H}(C_2,\infty)$
of ${\widetilde H}$  is absolutely continuous, i.e.,
    \begin{equation}\label{4.90}
\sigma_s({\widetilde H})\cap(C_2,\infty)=\emptyset.
    \end{equation}

Moreover,   ${\widetilde H}E_{\widetilde H}(C_2,\infty)$ is
unitarily equivalent to the part $-\Delta E_{
-\Delta}(C_2,\infty)$ of $-\Delta.$

    $(ii)$ Assume,  in addition, that the conditions in  Proposition
\ref{stronglyxdet} are satisfied, i.e.,  $d_*(X_n) > 0$ and
$D^*(X_n)=0.$  Then ${\widetilde H}_+ := {\widetilde
H}E_{\widetilde H}(\R_+)$ is unitarily equivalent to $H_0 =
-\Delta.$ In particular, ${\widetilde H}_+$ is purely absolutely
continuous, ${\widetilde H}_+ = {\widetilde H}^{ac}_+$.
    \end{theorem}
\begin{proof}
As in the proof of Proposition \ref{prop4.17}(ii)  we decompose
the symmetric operator $H$  in  a direct sum $H =
\widehat{H}\oplus H'$ of a simple symmetric operator $\widehat{H}$
and a self-adjoint operator $H'$. Next we define a boundary
triplet $\widehat\Pi=\{\cH,\widehat\Gamma_0,\widehat\Gamma_1\}$
for $\widehat{H}^*$ by setting $\widehat\Gamma_j :=
\Gamma_j\upharpoonright \dom(\widehat{H}^*),
 \ j\in\{0,1\},$  and note that  
the corresponding Weyl function $\widehat{M}(\cdot)$ coincides
with the Weyl function $M(\cdot)$ of $\Pi$. Further,  any proper
extension $\widetilde H = H_\Theta$ of $H$ admits  a decomposition
$H_\Theta = \widehat{H}_\Theta\oplus H'.$ In particular,  the
operator $H_0 = -\Delta$ is decomposed  as  $H_0=\widehat
H_0\oplus H'$, where $\widehat H_0 = \widehat
H^*\upharpoonright\ker (\widehat\Gamma_0) = \widehat H_0^*$. Being
a part of
 $H_0$, the operator $H'= (H')^*$ is absolutely continuous and $\sigma(H') = \sigma_{ac}(H')\subset \R_+$,
 because  $\sigma(H_0) = \sigma_{ac}(H_0)= {\overline \R_+}.$
Therefore, it suffices to prove all assertions  for
self-adjoint extensions  $\widehat{H}_\Theta$ of the \emph{simple
symmetric operator} $\widehat{H}.$

$(i)$ To prove  \eqref{4.90} for any extension of  $\widehat{H}$
it suffices to verify the conditions of Lemma \ref{lem4.22} noting
that $\widehat M(\cdot) =  M(\cdot).$  First we prove that for any
$t\in{\mathbb R}_+$ the uniform limit
    \begin{equation}\label{4.70A}
M(t + i0) :=  u-\lim_{y\downarrow 0}M(t + iy) \cong
 \left( i\sqrt{t}\delta_{kj}+ \frac{e^{i\sqrt{t}|x_k -
x_j|}-\delta_{kj}}{|x_k-x_j| +
\delta_{kj}}\right)_{j,k=1}^\infty,\quad t\in\mathbb{R},
    \end{equation}
exists, where the symbol $T\cong \kT$ means that the operator $T$  has the matrix  $\kT$ with respect
to the standard basis of $l^2(\bN)$.

 Indeed, it follows from \eqref{W3}  that for any $\xi,\ \eta\in
l^2(\mathbb N)$,
   \begin{eqnarray}\label{4.70}
\langle
\bigl(M(t+iy)-M(t)\bigr)\xi,\eta\rangle = (\sqrt{t+iy}-\sqrt{t}) \langle\xi,\eta\rangle
+ \ \sumprime_{j, k\in \N\ }\bigl(e^{-\beta|x_j-x_k|}-1 \bigr)
\frac{e^{i\alpha|x_j-x_k|}}{|x_j-x_k|}\xi_j\overline{\eta}_k.
   \end{eqnarray}
Fix $\varepsilon>0$. By to the assumption \eqref{4.90HScond} there exists $N=N(\varepsilon)\in{\N}$ such that
     \begin{equation}\label{4.94}
\sum_{j\ge N} \sumprime_{k\in \N\ }  
\frac{1}{|x_j - x_k|^2} +  \sum_{k\ge N} \sumprime_{j\in \N\ }
\frac{1}{|x_j - x_k|^2} <(\varepsilon/2)^2.
     \end{equation}
Then
     \begin{eqnarray}\label{4.71}
\sum_{j\ge N} \sumprime_{k\in \N\ }\frac{1}{|x_j -
x_k|}|\xi_j\overline{\eta}_k| +  \sum_{k\ge N}
\sumprime_{j\in \N\ }
\frac{1}{|x_j - x_k|}|\xi_j\overline{\eta}_k| \nonumber  \\
\le \bigl(\sum_{j\ge N}|\xi_j|^2\bigr)^{1/2}
\bigl(\sum^{\infty}_{k=1}|\eta_k|^2\bigr)^{1/2} \left(\sum_{j\ge
N} \sumprime_{k\in \N\ }
\frac{1}{|x_j - x_k|^2}\right)^{1/2}  \nonumber  \\
+ \bigl(\sum_{j\ge N}|\eta_k|^2\bigr)^{1/2}(\sum^{\infty}_{j = 1}
|\xi_k|^2\bigr)^{1/2}\left(\sum_{k\ge N} \sumprime_{j\in \N\ }
\frac{1}{|x_j - x_k|^2}\right)^{1/2}   \nonumber \\
\le 2^{-1}\varepsilon\|\xi\|_{l^2}\cdot\|\eta\|_{l^2}.
     \end{eqnarray}

On the other hand, since     
$d_*(X)>0,$  we can find $\beta_0=\beta_0(N)$ such that
    \begin{eqnarray}\label{4.72}
\sum^N_{j,k=1}\frac{(1-e^{-\beta|x_j-x_k|})}{|x_j-x_k|} \le
\varepsilon d_*(X)^{-1} \qquad\text{for}\qquad
\beta\in(0,\beta_0).
    \end{eqnarray}
Combining \eqref{4.70} with \eqref{4.71}  and \eqref{4.72} we get
     \begin{equation}
|\langle \bigl(M(t+iy)-M(t)\bigr)\xi,\eta\rangle |
 \le \varepsilon\bigl(1 + d_*(X)^{-1}\bigr)\|\xi\|_{l^2}\cdot\|\eta\|_{l^2},\qquad
 y\in(0, y_0),
     \end{equation}
that is,
     \begin{equation}\label{4.98}
\|M(t+iy)-M(t)\|
 \le \varepsilon\bigl(1 + d_*(X)^{-1}\bigr)\quad\text{for}\quad
 y\in(0, y_0).
     \end{equation}
Thus, the uniform limit \eqref{4.70A} exists for any $t\in \R_+.$

Further, it follows from \eqref{4.70A} that
    \begin{equation}\label{im1}
M_I(t) := M_I(t + i0) \cong \sqrt{t}
\left(\delta_{kj}+ \frac{\sin(\sqrt{t}|x_k-x_j|)}{\sqrt{t} (|x_k-x_j| +
\delta_{kj})}\right)_{j,k=1}^\infty ,   \qquad t
\in\mathbb{R}_+ .
     \end{equation}
This relation  combined with  assumption \eqref{4.90HScond}
yields $0\in\rho\bigl(M_I(t)\bigr)$ for $t > C_2$. The assertion follows now by applying Lemma \ref{lem4.22} to the operator
$\widehat{H}$  and the interval $(C_2, \infty).$

(ii)
 By \eqref{3.13}  the  function  $\Omega_3(t) = \frac{\sin
 t}{t}$ is  in $\Phi_3$. Hence, by Proposition \ref{stronglyxdet},  the matrix function
$\Omega_3(t\|\cdot\|)$ is strongly $X$-positively definite for
any $t>0,$ i.e., the matrix $\Omega_3(t\|x_j - x_k\|)_{j,k\in
\N}$ is positively definite for any $t>0.$  By
\eqref{im1} we have
  $$
M_I(t):= M_I(t+i0) \cong \sqrt t  \Omega_3(\sqrt t \|x_j -
x_k\|)_{j,k\in \N}, \qquad t\in \R_+.
  $$
Hence $M_I(t)$ is positively definite for  $t\in \R_+.$ It remains
to apply Lemma \ref{lem4.22} to the boundary triplet
$\widehat{\Pi}$ and the interval $\R_+.$
     \end{proof}
%
%
 Next we present another result on the ac-spectrum of self-adjoint extensions that is based on Corollary \ref{schurcor}.
%
%
        \begin{theorem}\label{AcspecTheorem2}
Let ${\widetilde H}$ be an arbitrary self-adjoint extension of
$H.$  Assume that
    \begin{equation}\label{3.34CompactnessCond2}
\lim_{p\to\infty}\left({\rm sup}_{j\in \bN}~\sumprime_{k\in \N\ }
\frac{1}{|x_k{-}x_j|}\right) = 0
  \end{equation}
and let $C_1$ be defined by  \eqref{matrixs14A}.  Then:

$(i)$ The part ${\widetilde H}E_{\widetilde H}(C_1^2,\infty)$ of ${\widetilde H}$ is absolutely continuous, i.e.
    \begin{equation}\label{4.90B}
\sigma_s({\widetilde H})\cap(C_1^2,\infty) = \emptyset.
    \end{equation}
Moreover,   ${\widetilde H}E_{\widetilde H}(C_1^2,\infty)$ is
unitarily equivalent to the part $-\Delta E_{
-\Delta}(C_1^2,\infty)$ of $-\Delta.$

    $(ii)$ Assume,  in addition, that the conditions of  Proposition
\ref{stronglyxdet} are fulfilled, i.e.  $d_*(X_n) > 0$ and
$D^*(X_n)=0.$  Then ${\widetilde H}_+  = {\widetilde
H}E_{\widetilde H}(\R_+)$ is unitarily equivalent to $H_0 =
-\Delta.$ In particular, ${\widetilde H}_+$ is purely absolutely
continuous, i.e. ${\widetilde H}_+ = {\widetilde H}^{ac}_+$.
    \end{theorem}
    \begin{proof}
(i)\  The proof is similar to that of Theorem
\ref{AcspecTheorem}(i).  Indeed,   by  assumption
\eqref{3.34CompactnessCond2}, for any  $\varepsilon>0$ one can find
$N=N(\varepsilon)\in{\N}$ such that
     \begin{equation}\label{4.103}
\sup_{j\ge N} \sumprime_{k\in \N\ }
\frac{1}{|x_j - x_k|}\  + \  \sup_{k\ge N}
\sumprime_{j\in \N\ }
\frac{1}{|x_j - x_k|} < \varepsilon/2.
     \end{equation}
Starting with \eqref{4.103} instead  of \eqref{4.94}  and applying
Corollary  \ref{schurcor}  we derive
     \begin{eqnarray}\label{4.71AA}
\sum_{j\ge N} \sumprime_{k\in \N\ }
\frac{1}{|x_j -
x_k|}|\xi_j\overline{\eta}_k| +  \sum_{k\ge N}
\sumprime_{j\in \N\ }\frac{1}{|x_j -
x_k|}|\xi_j\overline{\eta}_k|  \le
2^{-1}\varepsilon\|\xi\|_{l^2}\cdot\|\eta\|_{l^2}
     \end{eqnarray}
which implies \eqref{4.98}. That the operator $M_I(\cdot)$ has a
bounded inverse if $t> C_1^2$  follows from  \eqref{im1} and
Proposition \ref{boundedr3}.  It remains to apply Lemma
\ref{lem4.22} to the operator $\widehat{H}$  and the interval
$(C_1^2, \infty).$

(ii) follows by arguing in  a similar manner as in  the proof of Theorem
\ref{AcspecTheorem}(ii).
      \end{proof}
  \begin{remark}
(i)\  The assertions of Theorems \ref{AcspecTheorem}(iii) and  \ref{AcspecTheorem2}(iii) remains  valid if the
sequence $X$ satisfies the assumptions of
Proposition \ref{pointsr3}(i)  instead of Proposition
\ref{stronglyxdet}. The proof of Theorem \ref{AcspecTheorem}(ii)
shows that Propositions \ref{stronglyxdet} and  \ref{pointsr3}(i)
guarantee the absence of singular continuous spectrum and of
eigenvalues embedded in the $ac$-spectrum for any self-adjoint
extension $\widetilde H$ of $H$.

(ii) \ For sets  $X=\{x_j\}_1^m$ of finitely many points  a description  of  the
$ac$-spectrum and  the point spectrum of  self-adjoint
realizations of  $\mathfrak{L}_3$        
was obtained by different methods in \cite[Theorem 1.1.4]{AGHH88}
and \cite{GMZ11}.  For this purpose a connection with
radial positive definite functions was exploited for the first
time and  strong  $X$-positive definiteness of some functions
$f\in \Phi_3$ was used in  \cite{GMZ11}.
   \end{remark}
    \begin{remark}
At first glance it seems that Theorem \ref{AcspecTheorem}  might
contradict the classical Weyl -- von Neumann theorem \cite[Theorem
X.2.1]{Kato66}, \cite[Theorem 13.16.1]{ReeSim78}  which states the
existence of an additive perturbation $K= K^*\in \mathfrak S_2$
such that the operator $H + K$ has a purely point spectrum. In
fact, Theorem \ref{AcspecTheorem} yields explicit examples showing
that the analog of the Weyl -- von Neumann theorem does not hold
for non-additive (singular) compact perturbations. Under the
assumptions of Theorem \ref{AcspecTheorem}(ii), for \emph{any}
self-adjoint extension $\widetilde{H}$ of $H$, the part
${\widetilde H}E_{\widetilde H}(\R_+)$ is purely absolutely
continuous and ${\widetilde H}E_{\widetilde H}(\R_+)$ is unitarily
equivalent to $H = -\Delta.$
\emph{This shows that  
both the ac-spectrum $\sigma(H)$ and its multiplicity cannot  be
eliminated by some perturbations $K_{\widetilde H} := (\widetilde
H - i)^{-1}- (H_0 - i)^{-1}$ with $\widetilde H = \widetilde H^*
\in \Ext_H$. That is, the operator $H = -\Delta$ satisfies the
property of $ac$-minimality  in the sense of \cite{MalNei09}.
Moreover, if $K_{\widetilde H}$ is compact, then ${\widetilde
H}E_{\widetilde H}(\R_+)$ is even unitarily equivalent to $H =
-\Delta.$} A similar result was obtained for realizations in
$L^2(\R_+,\cH)$ of the differential expression $\mathcal L =
\frac{d^2}{dx^2} +T$ with unbounded non-negative operator
potential $T=T^*\in \mathcal C(\cH)$ in \cite{MalNei09}.
However, in contrast to our Theorems \ref{AcspecTheorem},
\ref{AcspecTheorem2},   the non-negative spectrum of some
realizations of $\mathcal L$   \emph{might contain a singular
part} (see  \cite{MalNei09}).

Note also that in contrast to the $3D$-case one dimensional sparse
point interactions (as well as ordinary potentials) may lead to singular
spectrum.
   \end{remark}
       \begin{remark}
The absolute continuity of self-adjoint realizations $\widetilde H$
of $H$ has been studied only  for special configurations
 $X = Y+\Lambda$, where $Y=\{y_j\}^N_1\in{\mathbb R}^3$ is a finite
set and $\Lambda = \{\sum^3_1 n_j a_j\in{\R}^3:
(n_1,n_2,n_3)\in{\mathbb Z}^3\}$ is the Bravais  lattice. It was first proved in \cite{GrHKM80} that in the case
 $N=1$ the spectrum of local periodic
realizations is absolutely continuous and contains at most two
bands (see also \cite[Theorems 1.4.5, 1.4.6]{AGHH88}). Further
development  can be found in
 \cite{AGHH87, AGHHK86, HHM83, HHJ84, Kar83},
The most complete result in this direction was obtained in
\cite{AG2000}. It was proved in \cite{AG2000} that the spectrum of
some (not necessarily local) realizations $\widetilde H$  is
absolutely continuous and has a band structure with a finite
number of gaps (for the negative part of the energy axis this
result  was proved earlier in \cite{AGHHK86, HHM83}).
 In particular, these results
confirm the Bethe-Sommerfeld conjecture on the finiteness of bands
for the case of periodic perturbations.
       \end{remark}

{\bf Acknowledgement}
We express our gratitude to Fritz Gesztesy for his careful reading of
the manuscript and numerous useful remarks.
We also thank L.L. Oridoroga for his help
in proving  Lemma \ref{lemma4.7}.

A part of this work was done while  the first named author  was visiting
the Department of Mathematics at the University of Leipzig. His visit
was supported by the DFG  grant Schm 009/4-1.


\quad

Mark Malamud,\\
\emph{Institute of Applied Mathematics and Mechanics, NAS of Ukraine},\\
\emph{R. Luxemburg str. 74},\\
\emph{83114 Donetsk,Ukraine}\\
 \emph{e-mail:} mmm@telenet.dn.ua
 \\

 Konrad Schm\"udgen,\\
\emph{ Institut of Mathematics,    University of Leipzig },\\
\emph{Johannisgasse 26},\\
\emph{04109 Leipzig, Germany}\\
 \emph{e-mail:}schmuedgen@math.uni-leipzig.de

\end{document}